\documentclass{amsart}

\usepackage[all,cmtip]{xy}
\usepackage{tikz}
\usepackage{tikz-cd}
\usetikzlibrary{shapes.geometric}
\usetikzlibrary{shapes.multipart} 
\usetikzlibrary{arrows}
    \tikzstyle{int}=[circle, draw,fill=black,outer sep=0,minimum size=3pt, inner sep=0]
    \tikzstyle{ext}=[circle, draw=black,outer sep=0,inner sep=1pt]
    \tikzstyle{black}=[circle,fill=black!,inner sep=0pt,minimum size=2mm]
    \tikzstyle{sblack}=[circle,fill=black!,inner sep=0pt,minimum size=1.5mm]
    \tikzstyle{mblack}=[circle,fill=black!,inner sep=0pt,minimum size=1mm]
    \tikzstyle{white}=[circle,inner sep=0pt,draw=black,minimum size=2mm]
    \tikzstyle{invisible}=[coordinate,inner sep=0pt,minimum size=2mm]
    \tikzstyle{biw}=[circle split,inner sep=1pt,draw=black,minimum size=8mm]
    \tikzstyle{biwb}=[rectangle,inner sep=1pt,draw=black,minimum size=8mm]
    \tikzstyle{uniw}=[circle,inner sep=1pt,draw=black,minimum size=8mm]
    \tikzstyle{unimin}=[circle,inner sep=1pt,draw=black,minimum size=5mm]
\usepackage{latexsym}
\usepackage{amssymb}
\usepackage{amsmath}
\usepackage{amsthm}
\usepackage{amscd}
\usepackage{fancyhdr}
\usepackage{fullpage}
\usepackage{mathrsfs}
\usepackage{tikz-cd}

\xyoption{arc}
\usepackage[center]{caption}

\usepackage{xcolor}
    \definecolor{mygray}{gray}{0.8}

\newcommand{\GC}{\mathsf{GC}}
\newcommand{\fcGC}{\mathsf{fcGC}}

\newcommand{\dGC}{\mathsf{dGC}}
\newcommand{\WGC}{\mathsf{wGC}}
\newcommand{\wGC}{\mathsf{wGC}}
\newcommand{\OGC}{\mathsf{oGC}}
\newcommand{\fWGC}{\mathsf{fwGC}}
\newcommand{\fwGC}{\mathsf{fwGC}}

\newcommand{\qGC}{\mathsf{qGC}}

\newcommand{\tGC}{\mathsf{tGC}}

\newcommand{\mGC}{\mathsf{mGC}}

\newcommand{\C}{\mathsf{C}}
\newcommand{\fM}{\mathsf{fM}}
\newcommand{\M}{\mathsf{M}}

\usetikzlibrary{circuits.logic.US,circuits.logic.IEC,fit}
\newcommand\addvmargin[1]{
  \node[fit=(current bounding box),inner ysep=#1,inner xsep=0]{};
}
\theoremstyle{plain}
\swapnumbers
\newtheorem{theorem}{Theorem}[subsection]
\newtheorem{corollary}[theorem]{Corollary}

\newtheorem{lemma}[theorem]{Lemma}
\newtheorem{proposition}[theorem]{Proposition}

\newtheorem{prop-def}[theorem]{Proposition-definition}

\newtheorem{main-theorem}{Main~Theorem}[section]
\newtheorem{section-theorem}{Theorem}[section]
\newtheorem{section-corollary}{Corollary}[section]

\theoremstyle{definition}
\newtheorem{example}[theorem]{Example}
\newtheorem{remark}[theorem]{Remark}
\newtheorem{definition}[theorem]{Definition}

\usetikzlibrary{decorations.markings}

\def\MarkLt{6pt}
\def\MarkSep{3pt}

\tikzset{
  TwoMarks/.style={
    postaction={decorate,
      decoration={
        markings,
        mark=at position #1 with
          {
              \begin{scope}[xslant=0.2]
              \draw[line width=\MarkSep,white,-] (0pt,-\MarkLt) -- (0pt,\MarkLt) ;
              \draw[-] (-0.5*\MarkSep,-\MarkLt) -- (-0.5*\MarkSep,\MarkLt) ;
              \draw[-] (0.5*\MarkSep,-\MarkLt) -- (0.5*\MarkSep,\MarkLt) ;
              \end{scope}
          }
       }
    }
  },
  TwoMarks/.default={0.5},
  OneMark/.style={
    postaction={decorate,
      decoration={
        markings,
        mark=at position #1 with
          {
              \draw[-,solid] (0,-\MarkLt) -- (0,\MarkLt) ;
          }
       }
    }
  },
  OneMark/.default={0.5},
  NMark/.style={
    postaction={decorate,
      decoration={
        markings,
        mark=at position #1 with
          {
              \draw[-,solid] (-3pt,-\MarkLt) -- (-3pt,\MarkLt) ;
              \draw[-,solid] (3pt,-\MarkLt) -- (3pt,\MarkLt) ;
              
          }
       }
    }
  },
  NMark/.default={0.5}
}

\long\def\symbolfootnote[#1]#2{\begingroup%
\def\thefootnote{\fnsymbol{footnote}}\footnote[#1]{#2}\endgroup}

\author{Oskar Frost}
\address{Mathematics Research Unit, University of Luxembourg, Maison du Nombre, 6 Avenue de la Fonte,
 L-4364 Esch-sur-Alzette, Grand Duchy of Luxembourg }
\email{oskar.frost@uni.lu}
\title{Deformation theory of the wheeled properad of strongly homotopy Lie bialgebras and graph complexes}

\begin{document}
\begin{abstract}
It is well-known that the Lie algebra of homotopy non-trivial degree zero derivations of the properad of strongly homotopy Lie bialgebras $\mathcal{H}olieb$ can be identified with the Grothendieck-Teichmuller Lie algebra $\mathfrak{grt}$. We study in this paper the derivation complex of the wheeled closure $\mathcal{H}olieb^\circlearrowleft$ (and of its degree shifted version $\mathcal{H}olieb_{p,q}^\circlearrowleft,\ \forall p,q\in\mathbb{Z}$) and establishing a quasi-isomorphism to a version of the Kontsevich graph complex. This result leads us to a surprising conclusion that the Lie algebra of homotopy non-trivial derivations of the wheeled properad $\mathcal{H}olieb^{\circlearrowleft}$ can be identified with the direct sum of \textit{two} copies of $\mathfrak{grt}$. As an illustrative example, we describe explicitly how the famous tetrahedron class in $\mathfrak{grt}$ acts as a derivation of $\mathcal{H}olieb^{\circlearrowleft}$ in two homotopy inequivalent ways.

\end{abstract}
\maketitle

\section{Introduction and outline}
\label{sec:introduction}
\subsection{Introduction}
A Lie bialgebra $V$ is a vector space equipped with two operations (bracket and cobracket) $[\, ,]:V\otimes V\rightarrow V$ and $\Delta:V\rightarrow V\otimes V$ such that $(V,[\,,])$ is a Lie algebra, $(V,\Delta)$ is a Lie coalgebra and that the bracket and cobracket satisfy a compability relation.
Lie bialgebras were first introduced by Drinfeld \cite{D} in the context of the theory of Yang Baxter equations and quantum groups. Later it has found applications in many other areas of mathematics and mathematical physics such as the theory of Hopf algebra deformations of universal enveloping algebras \cite{ES}, homological algebra \cite{Sc}, string topology and symplectic field theory \cite{CFL}, Goldman-Turaev theory of free loops in Riemann surfaces with boundaries \cite{G,Tu}, Lagrangian Floer theory of higher genus \cite{Tu}, and the theory of cohomology groups of moduli spaces of algebraic curves with labeled punctures \cite{MW2,AWZ}.
\\

Homotopy Lie bialgebra structures on arbitrary (possibly infinite-dimensional) graded vector spaces are controlled by the homotopy properad $\mathcal{H}olieb$. In this paper we will study this properad in greater generality with the degree of the basic operations shifted. More specifically, $\mathcal{H}olieb_{p,q}$ is the properad controlling homotopy Lie bialgebras with Lie bracket of degree $1-p$ and Lie cobracket of degree $1-q$. The usual properad being realised by $\mathcal{H}olieb_{1,1}$. The deformation theory of $\mathcal{H}olieb_{p,q}$ was studied in \cite{MW1} where it was proven that the cohomology of the derivation complex can be identified with $H(\GC_{p+q})$, the cohomology of the famous Kontsevich graph complex. It was proven by T. Willwacher in \cite{W1} that $H^0(\GC_2)=\mathfrak{grt}$, the Lie algebra of the Grothendieck–Teichmuller group. The latter two results were used in \cite{MW3} to classify all homotopy non-trivial universal quantizations of Lie bialgebras as the set of Drinfeld associators.\\

The Kontsevich graph complexes are dg Lie algebras who come in a family $\GC_k$ for any integer $k$. They are isomorphic (up to degree shift) for $k$ of the same parity, and are generated by equivalence classes of undirected graphs. Hence there are essentially two graph complexes: the even version $\GC_2$ and the odd version $\GC_3$. They admit directed versions, $\dGC_2$ and $\dGC_3$, whose elements are classes of graphs with a fixed direction on edges. These directed versions are quasi-isomorphic to $\GC_2$ and $\GC_3$ respectively \cite{W1}. The directed extensions give us essentially nothing new, but they contain interesting subcomplexes with different cohomologies. We will be particularly interested in some of them. Let $\dGC_k^{s}$ be the \textit{sourced} subcomplex of $\dGC_k$ spanned by classes of graphs with at least one source vertex. Similarly let $\dGC_k^{t}$ be the \textit{targeted} subcomplex of $\dGC_k$ spanned by classes of graphs with at least one target vertex. These complexes are naturally isomorphic to each other by reversing the orientation of every edge in a graph. It has been shown in \cite{Z2} that $H^\bullet(\dGC_{k+1}^s)=H^\bullet(\dGC_{k+1}^t)=H^\bullet(\GC_{k})$, and hence in particular that $H^0(\dGC_{3}^s)=H^0(\dGC_{3}^t)=H^0(\GC_2)\cong\mathfrak{grt}$. Furthermore, $H^l(\GC_3)=0$ for $l\geq -2$, and so from the short exact sequence 
\begin{equation*}
    0 \rightarrow \dGC_3^{s}\rightarrow \dGC_3 \rightarrow \dGC_3/\dGC^s_3 \rightarrow 0
\end{equation*}
we get that $H^{-1}(\dGC_3/\dGC^s_t)\cong\mathfrak{grt}$ (and similary for the targeted complex).
Another notable subcomplex of $\dGC_k$ is the \textit{sourced and targeted} subcomplex $\dGC^{st}_k$ of graphs with at least one source and (or) one target vertex. It has been shown in \cite{Z3} that $H^0(\dGC^{st}_3)=\mathfrak{grt}\oplus\mathfrak{grt}$ (see also \cite{M3}).
Similarly the complex $\dGC_k^{s+t}$ is the subcomplex of $\dGC_k$ spanned by classes of graphs with at least one source or target vertex. Let $\dGC^\circlearrowleft_k:=\dGC^k/\dGC^{s+t}_k$. The cohomology of $\dGC^{s+t}_k$ have been studied in \cite{Z3} and $\dGC_k^\circlearrowleft$ in \cite{M3} respectively. \\

There are two dg Lie algebras which can be associated with a dg prop, in this case the prop $\mathcal{H}olieb_{p,q}$, namely the derivation complex $Der(\mathcal{H}olieb_{p,q})$ and the deformation complex $Def(\mathcal{H}olieb_{p,q}\rightarrow\mathcal{H}olieb_{p,q})$ of the identity map. They are isomorphic as complexes
(up to a degree shift) but they have different Lie algebra structures. The cohomology of these complexes has been computed in \cite{MW1} and have been used by the same authors to classify all non-trivial homotopy classes of universal quantizations of strongly homotopy Lie bialgebras. They can be understood as morphisms of dg props $\mathcal{H}oassb\rightarrow D(\mathcal{H}olieb)$ where $\mathcal{H}oassb$ is a dg prop of homotopy associative bialgebras, and $D$ is the polydifferential endofunctor in the category of props introduced in \cite{MW2}. The existence of such morphisms has been proved earlier by Etingof and Kazhdan in \cite{EK}, and any such a morphism gives us a universal quantization of an arbitrary homotopy Lie bialgebra, a finite-dimensional or infinite dimensional one because the prop $\mathcal{H}olieb$ controls all of them. However, if one is only interested in (deformation quantization theory of) finite-dimensional homotopy Lie bialgebras, then one has to work with a differenit prop $\mathcal{H}olieb^\circlearrowleft$. The wheeled closure of $\mathcal{H}olieb$. which is much larger than $\mathcal{H}olieb$, but contains the latter as a subprop. In other words, there is a monomorphism of props $\mathcal{H}olieb\rightarrow\mathcal{H}olieb^\circlearrowleft$. This monomorphism implies the existence of a universal morphism $F:\mathcal{H}oass\rightarrow D(\mathcal{H}olieb^\circlearrowleft)$, but its deformation theory is controlled by a quite different deformation complex, which in turn contains the deformation complex of the wheeled properad $\mathcal{H}olieb^\circlearrowleft$. The main purpose of our paper is to compute the cohomology of the deformation complexes of $\mathcal{H}olieb^\circlearrowleft$ in terms of Kontsevich graph complexes.\\

The properad $\mathcal{H}olieb_{p,q}$ is generated by corollas with $n$ inputs and $m$ outputs such that $m,n\geq1$, $ m+n \geq 3$, i.e. no "curvature" like generators with $m=0$ and $n=0$ are allowed (see \cite{M2} for full details). However, in the wheeled case one can build from such generators elements of $\mathcal{H}olieb^\circlearrowleft_{p,q}$ with no inputs ($n=0$) or no outputs ($m=0$), or both ($m=n=0$). Such graphs can stay for non-trivial cohomology classes which control deformations of $\mathcal{H}olieb_{p,q}^\circlearrowleft$ into its more general version $\mathcal{H}olieb^{\circlearrowleft*}$ in which curvature terms are allowed (such a phenomenon is impossible in the case of unwheeled props). The derivation complex of $\mathcal{H}oileb^{\circlearrowleft}_{p,q}$ which controls \textit{all} possible deformations (i.e. with curvature terms allowed) of $\mathcal{H}olieb^{\circlearrowleft}$ is denoted by $Der(\mathcal{H}olieb_{p,q}^\circlearrowleft)$. See section \ref{sec:wheelification} for its precise definition.\\

The loop number $b$ of a connected graph $\Gamma$ is the integer $b=e-v+1$ where $e$ is the number of edges of $\Gamma$ and $v$ the number of vertices. The derivation complex decompose naturally into a direct sum of complexes $Der_{b=g}(\mathcal{H}olieb_{p,q}^\circlearrowleft)$ generated by graphs with loop number $g$. The case when $g=0$ is the same as in the case of non-wheeled properads, and so by the result in \cite{MW1}, we can say that $Der^*_{b=0}(\mathcal{H}olieb_{p,q}^\circlearrowleft)$ is one-dimensional and it is generated by the rescaling class. We study $Der^*_{b=1}(\mathcal{H}olieb_{p,q}^\circlearrowleft)$ and $Der^*_{b\geq2}(\mathcal{H}olieb_{p,q}^\circlearrowleft)$ separately, where the latter is the complex of graphs of loop number two and higher.\\

Our first main result in this paper is the proof of the following

\begin{theorem}[Theorem 1]
    $H(Der_{b\geq 2}(\mathcal{H}olieb_{p,q}^\circlearrowleft))=
    H(\dGC_{p+q+1}^\circlearrowleft)\oplus
    H(\dGC_{p+q+1})$.
\end{theorem}
\begin{proposition}
\begin{equation*}
    H(Der_{b=1}(\mathcal{H}olieb_{p,q}^\circlearrowleft))=\bigoplus_{\substack{i\geq 1\\
    i\equiv2(p+q+1)+1\mod 4}} \mathbb{K}[p+q+1-i].
\end{equation*}
\end{proposition}

Next, one can study a deformation theory of $\mathcal{H}olieb_{p,q}^{\circlearrowleft}$ in which the creation of curvature terms is prohibited. It is controlled by a subcomplex $Der^+(\mathcal{H}olieb_{p,q}^\circlearrowleft)$ of the full complex $Der(\mathcal{H}olieb_{p,q}^\circlearrowleft)$ (see section \ref{sec:Der*+} for its precise definition). Its cohomology is quite different from the complex with curvature terms. This complex also decompose naturally into a direct sum of complexes $Der^+_{b=g}(\mathcal{H}olieb_{p,q}^\circlearrowleft)$ generated by graphs with loop number $g$. The case $g=0$ is equivalent to the case above. We study study $Der^+_{b=1}(\mathcal{H}olieb_{p,q}^\circlearrowleft)$ and $Der^+_{b\geq2}(\mathcal{H}olieb_{p,q}^\circlearrowleft)$ separately, where the latter is the complex of graphs of loop number two and higher. Our second main result is the following theorem.

\begin{theorem}[Theorem 2]
    Let $Cone(P)[1]$ be the suspended mapping cone of the projection 
    \begin{equation*}
        P:\dGC_{x}\rightarrow\dGC_{x}/\dGC^s_{x}\oplus\dGC_{x}/\dGC^t_{x},\qquad \text{where } x=p+q+1.
    \end{equation*}
    Then there is a quasi-isomorphism $Cone(P)[1]\rightarrow Der^+_{b\geq2}(\mathcal{H}olieb_{p,q}^\circlearrowleft)$. In particular, $Cone(P)[1]$ fits in the short exact sequence
    \begin{equation*}
        \xymatrix{0 \ar[r] & (\dGC_{x}/\dGC^s_{x}\oplus \dGC_{x}/\dGC^t_{x})[1] \ar[r] & Cone(P)[1] \ar[r] & \dGC_{x} \ar[r] & 0}     .
    \end{equation*}
\end{theorem}
\begin{proposition}
    \begin{equation*}
        H(Der_{b=1}^+(\mathcal{H}olieb_{p,q}^\circlearrowleft))=
    \bigoplus_{\substack{i\geq 1\\ i\equiv 1 \mod 2}}
    \mathcal{C}_i,\quad \text{where }\mathcal{C}_i=
    \begin{cases}
        \mathbb{K}[p+q-i] & \text{if }i\equiv 2k+1 \mod 4\\
        \mathbb{K}[p+q-i]^2 & \text{if }i\equiv 2k+3 \mod 4
    \end{cases}
    \end{equation*}
\end{proposition}
From the induced long exact sequence on cohomology we get our third main theorem as a corollary.
\begin{theorem}[Theorem 3]
    $H^0(Der^+_{b\geq2}(\mathcal{H}olieb_{1,1}^\circlearrowleft))=H^{-1}(\dGC_3/\dGC_3^s)\oplus H^{-1}(\dGC_3/\dGC_3^t)\cong\mathfrak{grt}\oplus\mathfrak{grt}$.
\end{theorem}
There is in fact another proof of this theorem, derived from another short exact sequence using that $H^0(\dGC^{st}_3)\cong\mathfrak{grt}\oplus\mathfrak{grt}$. The second proof is given in section \ref{sec:dGCst}\\

As an illustrative example, we describe explicitly how the famous tetrahedron class in $\mathfrak{grt}$ acts as a derivation of $\mathcal{H}olieb^{\circlearrowleft}_{1,1}$ in two homotopy inequivalent ways.

\subsection{Outline}
In section 2 we define the basic concepts of Lie $n$-bialgebras and introduce its homotopy properad $\mathcal{H}olieb_{n}$ and its degree shifted version $\mathcal{H}olieb_{p,q}$. We also introduce the wheeled properads and derivation complexes of properads. Most notably we describe in detail the derivation complexes $Der(\mathcal{H}olieb_{p,q}^{\circlearrowleft})$ and $Der^+(\mathcal{H}olieb_{p,q}^{\circlearrowleft})$.

In section 3 we give a brief description of the Kontsevich graph complex $\GC_k$ and its directed version $\dGC_k$, and present some notable subcomplexes and results regarding their cohomology.

In section 4 we define two \textit{bi-weighted graph complexes} $\wGC_{p+q+1}$ and $\wGC^+_{p+q+1}$ and show that they are isomorphic to the derivation complexes $Der(\mathcal{H}olieb^{\circlearrowleft}_{p,q})$ and $Der^+(\mathcal{H}olieb^{\circlearrowleft}_{p,q})$ respectively.

In section 5 we define subcomplexes $\qGC_k$ and $\qGC_k^+$ of $\wGC_k$ and $\wGC_k^+$ of graphs whose vertices have four types of decorations, and show that the inclusions are quasi-isomorphisms. We also get a decomposition $\qGC_k=\qGC_k^0\oplus\qGC_k^*$ where $\qGC_k^0\cong\dGC_k^\circlearrowleft$.

In section 6 we reduce $\qGC_k^*$ and $\qGC_k^+$ by quotients, and find quotient complexes $\qGC_k^*$ and $\fM^+_k$ of graphs whose vertices have a single decoration.

In section 7 we finally establish the relations between $\fM_k^*$ and $\fM_k^+$ and the Kontsevich graph complexes.

In section 8, we give an explicit example of two the cohomology classes of $H^0(\dGC_3^{st})\cong\mathfrak{grt}\oplus\mathfrak{grt}$, represented by the tetrahedron class.

\subsection{Notation}
Let $\mathbb{S}_n$ denote the permutation group of the set $\{1,2,...,n\}$. The one dimensional sign representation of $\mathbb{S}_n$ is denoted by $sgn_n$. All vector spaces are assumed to be $\mathbb{Z}$ graded over a field $\mathbb{K}$ of characteristic zero. If $V=\bigoplus_{i\in\mathbb{Z}} V^i$ is a graded vector space, then $V[k]$ denotes the graded vector space where $V[k]^i=V^{i+k}$. For a properad $\mathcal{P}$ we denote by $\mathcal{P}\{k\}$ the properad being uniquely defined by the property: for any graded vector space $V$, a representation of $\mathcal{P}\{k\}$ in $V$ is identical to a representation of $\mathcal{P}$ in $V[k]$.\newline
Let $\Gamma$ be a graph with $v$ vertices and $e$ edges. The \textit{genus} of a graph is the number $e-v+1$. Further, for vertices in directed graphs we consider the following conventions: A vertex is \textit{passing} if it has exactly one incoming edge and one outgoing edge attached. It is a \textit{source} if only outgoing edges are attached to it, and a \textit{target} if there are only incoming edges attached. Finally a vertex is called \textit{generic} if it is at least trivalent, and have at least one outgoing edge and one incoming edge attached.
The \textit{loop number} of a graph $\Gamma$ is the number $b=e-v+1$, where $e$ is the number of edges and $v$ the number of vertices of $\Gamma$.

\subsection*{Acknowledgements} I would like to thank Sergei Merkulov for providing the motivation for this work and assisting with fruitful discussions. I also want to thank Marko Zivkovic for his input.

\section{Deformation complex of $\mathcal{H}olieb_{p,q}^\uparrow$ and $\mathcal{H}olieb_{p,q}^\circlearrowleft$}
\subsection{Lie-bialgebras}
\begin{definition}
    Let $n\in\mathbb{Z}$. A \textit{Lie }$n$\textit{-bialgebra} is a $\mathbb{Z}$ graded vector space $V$ together with two operations $[\,,]:V\otimes V\rightarrow V$ and $\Delta:V\rightarrow V\otimes V$. $[\,,]$ is a Lie-bracket of degree $-n$, $\Delta$ is a Lie coproduct of degree $0$ and the operations are compatible in the following sense:
\begin{equation*}
    \Delta([a,b])=\sum a_1\otimes [a_2,b]+[a,b_1]\otimes b_2 - (-1)^{(|a|+n)(|b|+n)}([b,a_1]\otimes a_2+b_1\otimes[b_2,a]),
\end{equation*}
for any $a,b\in V$ and where the sum is over the terms of the coproduct of $a$ and $b$ respectively. I.e $\Delta(a)=\sum a_1\otimes a_2$ and $\Delta(b)=\sum b_1\otimes b_2$.
\end{definition}
\begin{definition}
Let $\mathcal{L}ieb_n$ denote the properad of Lie $n$-bialgebras. Let $p,q\in\mathbb{Z}$ and define the degree shifted version of $\mathcal{L}ieb_n$ as $\mathcal{L}ieb_{p,q}:=\mathcal{L}ieb_{p+q-2}\{1-p\}$. so that the cobracket has degree $1-p$ and the bracket has degree $1-q$. $\mathcal{L}ieb_{p,q}$ is a quadratic operad, i.e $\mathcal{L}ieb_{p,q}=\mathcal{F}ree\langle E\rangle/\langle \mathcal{R}\rangle$, where $\mathcal{F}ree\langle E\rangle$ is the $\mathbb{S}$-bimodule generated by
\begin{align*}
    E(2,1)=(\mathbf{1}_1\otimes sgn^{\otimes |p|}_2)[p-1]=&
    \mathrm{span}\Big\langle
    \begin{tikzpicture}[node distance=1.2cm,auto,baseline=-0.1cm]
        \node[black] (r) at (0,0) {};
        \node[] (v1) at (-0.45,0.6) {$\scriptstyle1$}
            edge [<-] (r);
        \node[] (v2) at (0.45,0.6) {$\scriptstyle 2$}
            edge [<-] (r);
        \node[] (v3) at (0,-0.7) {$\scriptstyle 1$}
            edge [->] (r);
    \end{tikzpicture}
    = (-1)^{p}
    \begin{tikzpicture}[node distance=1.2cm,auto,baseline=-0.1cm]
        \node[black] (r) at (0,0) {};
        \node[] (v1) at (-0.45,0.6) {$\scriptstyle 2$}
            edge [<-] (r);
        \node[] (v2) at (0.45,0.6) {$\scriptstyle 1$}
            edge [<-] (r);
        \node[] (v3) at (0,-0.7) {$\scriptstyle 1$}
            edge [->] (r);
    \end{tikzpicture}
    \Big\rangle\\
    E(2,1)= (sgn^{\otimes|q|}_2\otimes\mathbf{1}_1)[q-1]=&
    \mathrm{span}\Big\langle
    \begin{tikzpicture}[node distance=1.2cm,auto,baseline=-0.1cm]
        \node[black] (r) at (0,0) {};
        \node[] (v1) at (-0.45,-0.6) {$\scriptstyle1$}
            edge [->] (r);
        \node[] (v2) at (0.45,-0.6) {$\scriptstyle 2$}
            edge [->] (r);
        \node[] (v3) at (0,0.7) {$\scriptstyle 1$}
            edge [<-] (r);
    \end{tikzpicture}
    = (-1)^{q}
    \begin{tikzpicture}[node distance=1.2cm,auto,baseline=-0.1cm]
        \node[black] (r) at (0,0) {};
        \node[] (v1) at (-0.45,-0.6) {$\scriptstyle 2$}
            edge [->] (r);
        \node[] (v2) at (0.45,-0.6) {$\scriptstyle 1$}
            edge [->] (r);
        \node[] (v3) at (0,0.7) {$\scriptstyle 1$}
            edge [<-] (r);
    \end{tikzpicture}
    \Big\rangle
\end{align*} and $\langle\mathcal{R}\rangle$ is the ideal generated by the elements
\begin{equation*}
    \mathcal{R}:\begin{cases}
        \begin{tikzpicture}[node distance=1.2cm,auto,baseline={([yshift=-.5ex]current bounding box.center)}]
        \node[black] (r) at (0,0) {};
        \node[] (v1) at (-0.45,0.6) {$\scriptstyle1$}
            edge [<-] (r);
        \node[] (v2) at (0.45,0.6) {$\scriptstyle 2$}
            edge [<-] (r);
        \node[black] (l) at (0.45,-0.6) {}
            edge [->] (r);
        \node[] (w1) at (0.9,0) {$\scriptstyle 3$}
            edge [<-] (l);
        \node[] (w2) at (0.45,-1.3) {$\scriptstyle 1$}
            edge [->] (l);
    \end{tikzpicture}
    +
    \begin{tikzpicture}[node distance=1.2cm,auto,baseline={([yshift=-.5ex]current bounding box.center)}]
        \node[black] (r) at (0,0) {};
        \node[] (v1) at (-0.45,0.6) {$\scriptstyle3$}
            edge [<-] (r);
        \node[] (v2) at (0.45,0.6) {$\scriptstyle 1$}
            edge [<-] (r);
        \node[black] (l) at (0.45,-0.6) {}
            edge [->] (r);
        \node[] (w1) at (0.9,0) {$\scriptstyle 2$}
            edge [<-] (l);
        \node[] (w2) at (0.45,-1.3) {$\scriptstyle 1$}
            edge [->] (l);
    \end{tikzpicture}
    +
    \begin{tikzpicture}[node distance=1.2cm,auto,baseline={([yshift=-.5ex]current bounding box.center)}]
        \node[black] (r) at (0,0) {};
        \node[] (v1) at (-0.45,0.6) {$\scriptstyle2$}
            edge [<-] (r);
        \node[] (v2) at (0.45,0.6) {$\scriptstyle 3$}
            edge [<-] (r);
        \node[black] (l) at (0.45,-0.6) {}
            edge [->] (r);
        \node[] (w1) at (0.9,0) {$\scriptstyle 1$}
            edge [<-] (l);
        \node[] (w2) at (0.45,-1.3) {$\scriptstyle 1$}
            edge [->] (l);
    \end{tikzpicture}
    ,\quad
    \begin{tikzpicture}[node distance=1.2cm,auto,baseline={([yshift=-.5ex]current bounding box.center)}]
        \node[black] (r) at (0,0) {};
        \node[] (v1) at (-0.45,-0.6) {$\scriptstyle1$}
            edge [->] (r);
        \node[] (v2) at (0.45,-0.6) {$\scriptstyle 2$}
            edge [->] (r);
        \node[black] (l) at (0.45,0.6) {}
            edge [<-] (r);
        \node[] (w1) at (0.9,0) {$\scriptstyle 3$}
            edge [->] (l);
        \node[] (w2) at (0.45,1.3) {$\scriptstyle 1$}
            edge [<-] (l);
    \end{tikzpicture}
    +
    \begin{tikzpicture}[node distance=1.2cm,auto,baseline={([yshift=-.5ex]current bounding box.center)}]
        \node[black] (r) at (0,0) {};
        \node[] (v1) at (-0.45,-0.6) {$\scriptstyle3$}
            edge [->] (r);
        \node[] (v2) at (0.45,-0.6) {$\scriptstyle 1$}
            edge [->] (r);
        \node[black] (l) at (0.45,0.6) {}
            edge [<-] (r);
        \node[] (w1) at (0.9,0) {$\scriptstyle 2$}
            edge [->] (l);
        \node[] (w2) at (0.45,1.3) {$\scriptstyle 1$}
            edge [<-] (l);
    \end{tikzpicture}
    +
    \begin{tikzpicture}[node distance=1.2cm,auto,baseline={([yshift=-.5ex]current bounding box.center)}]
        \node[black] (r) at (0,0) {};
        \node[] (v1) at (-0.45,-0.6) {$\scriptstyle2$}
            edge [->] (r);
        \node[] (v2) at (0.45,-0.6) {$\scriptstyle 3$}
            edge [->] (r);
        \node[black] (l) at (0.45,0.6) {}
            edge [<-] (r);
        \node[] (w1) at (0.9,0) {$\scriptstyle 1$}
            edge [->] (l);
        \node[] (w2) at (0.45,1.3) {$\scriptstyle 1$}
            edge [<-] (l);
    \end{tikzpicture}\\
    \begin{tikzpicture}[node distance=1.2cm,auto,baseline={([yshift=-.5ex]current bounding box.center)}]
        \node[black] (u) at (0,0.3) {};
        \node[black] (l) at (0,-0.3) {}
            edge [->] (u);
        \node[] (u1) at (-0.45,0.9) {$\scriptstyle 1$}
            edge [<-] (u);
        \node[] (u2) at (0.45,0.9) {$\scriptstyle 2$}
            edge [<-] (u);
        \node[] (l1) at (-0.45,-0.9) {$\scriptstyle 1$}
            edge [->] (l);
        \node[] (l2) at (0.45,-0.9) {$\scriptstyle 2$}
            edge [->] (l);
    \end{tikzpicture}
    -
    \begin{tikzpicture}[node distance=1.2cm,auto,baseline={([yshift=-.5ex]current bounding box.center)}]
        \node[black] (u) at (0.25,0.3) {};
        \node[black] (l) at (-0.25,-0.3) {}
            edge [->] (u);
        \node[] (u1) at (0.25,0.9) {$\scriptstyle 2$}
            edge [<-] (u);
        \node[] (u2) at (-0.7,0.3) {$\scriptstyle 1$}
            edge [<-] (l);
        \node[] (l1) at (-0.25,-0.9) {$\scriptstyle 1$}
            edge [->] (l);
        \node[] (l2) at (0.7,-0.3) {$\scriptstyle 2$}
            edge [->] (u);
    \end{tikzpicture}
    -(-1)^q
    \begin{tikzpicture}[node distance=1.2cm,auto,baseline={([yshift=-.5ex]current bounding box.center)}]
        \node[black] (u) at (0.25,0.3) {};
        \node[black] (l) at (-0.25,-0.3) {}
            edge [->] (u);
        \node[] (u1) at (0.25,0.9) {$\scriptstyle 2$}
            edge [<-] (u);
        \node[] (u2) at (-0.7,0.3) {$\scriptstyle 1$}
            edge [<-] (l);
        \node[] (l1) at (-0.25,-0.9) {$\scriptstyle 2$}
            edge [->] (l);
        \node[] (l2) at (0.7,-0.3) {$\scriptstyle 1$}
            edge [->] (u);
    \end{tikzpicture}
    -(-1)^{p+q}
    \begin{tikzpicture}[node distance=1.2cm,auto,baseline={([yshift=-.5ex]current bounding box.center)}]
        \node[black] (u) at (0.25,0.3) {};
        \node[black] (l) at (-0.25,-0.3) {}
            edge [->] (u);
        \node[] (u1) at (0.25,0.9) {$\scriptstyle 1$}
            edge [<-] (u);
        \node[] (u2) at (-0.7,0.3) {$\scriptstyle 2$}
            edge [<-] (l);
        \node[] (l1) at (-0.25,-0.9) {$\scriptstyle 2$}
            edge [->] (l);
        \node[] (l2) at (0.7,-0.3) {$\scriptstyle 1$}
            edge [->] (u);
    \end{tikzpicture}
    -(-1)^p
    \begin{tikzpicture}[node distance=1.2cm,auto,baseline={([yshift=-.5ex]current bounding box.center)}]
        \node[black] (u) at (0.25,0.3) {};
        \node[black] (l) at (-0.25,-0.3) {}
            edge [->] (u);
        \node[] (u1) at (0.25,0.9) {$\scriptstyle 1$}
            edge [<-] (u);
        \node[] (u2) at (-0.7,0.3) {$\scriptstyle 2$}
            edge [<-] (l);
        \node[] (l1) at (-0.25,-0.9) {$\scriptstyle 1$}
            edge [->] (l);
        \node[] (l2) at (0.7,-0.3) {$\scriptstyle 2$}
            edge [->] (u);
    \end{tikzpicture}
    \end{cases}.
\end{equation*}
\end{definition}
The \textit{standard operad} $\mathcal{L}ieb_0$ of Lie bialgebras corresponds to the case $p=q=1$.


Consider the $\mathbb{S}$-bimodule $\{E^\bullet(m,n)\}_{m,n\in\mathbb{N}}$ where
\begin{align*}
    E^\bullet(m,n)&=sgn^{\otimes|p|}_m\otimes sgn^{\otimes|q|}_n[p(m-1)+q(n-1)-1]\\
    &=
    \mathrm{span}\Big\langle
    \begin{tikzpicture}
        [shorten >=1pt,node distance=1.2cm,auto,baseline=-0.1cm]
        \node[black] (c) at (0,0) {};
        \node[] (u1) at (-1.4,0.6) {$\scriptstyle \sigma(1)$}
            edge [<-] (c);
        \node[] (u2) at (-0.8,0.6) {$\scriptstyle \sigma(2)$}
            edge [<-] (c);
        \node[] (u3) at (0,0.5) {$\scriptstyle \cdots$};
        \node[] (u4) at (0.8,0.6) {}
            edge [<-] (c);
        \node[] (u5) at (1.4,0.6) {$\scriptstyle \sigma(m)$}
            edge [<-] (c);
        \node[] (l1) at (-1.4,-0.6) {$\scriptstyle \tau(1)$}
            edge [->] (c);
        \node[] (l2) at (-0.8,-0.6) {$\scriptstyle \tau(2)$}
            edge [->] (c);
        \node[] (l3) at (0,-0.5) {$\scriptstyle \cdots$};
        \node[] (l4) at (0.8,-0.6) {}
            edge [->] (c);
        \node[] (l5) at (1.4,-0.6) {$\scriptstyle \tau(n)$}
            edge [->] (c);
    \end{tikzpicture}
    =(-1)^{p|\sigma|+q|\tau|}
    \begin{tikzpicture}
        [shorten >=1pt,node distance=1.2cm,auto,baseline=-0.1cm]
        \node[black] (c) at (0,0) {};
        \node[] (u1) at (-1.4,0.6) {$\scriptstyle 1$}
            edge [<-] (c);
        \node[] (u2) at (-0.8,0.6) {$\scriptstyle 2$}
            edge [<-] (c);
        \node[] (u3) at (0,0.5) {$\scriptstyle \cdots$};
        \node[] (u4) at (0.8,0.6) {}
            edge [<-] (c);
        \node[] (u5) at (1.4,0.6) {$\scriptstyle m$}
            edge [<-] (c);
        \node[] (l1) at (-1.4,-0.6) {$\scriptstyle 1$}
            edge [->] (c);
        \node[] (l2) at (-0.8,-0.6) {$\scriptstyle 2$}
            edge [->] (c);
        \node[] (l3) at (0,-0.5) {$\scriptstyle \cdots$};
        \node[] (l4) at (0.8,-0.6) {}
            edge [->] (c);
        \node[] (l5) at (1.4,-0.6) {$\scriptstyle n$}
            edge [->] (c);
    \end{tikzpicture}
    \Big\rangle.
\end{align*}
Let $\{E(m,n)\}_{m,n\in\mathbb{N}}$ be the $\mathbb{S}$-submodule defined by
\begin{equation*}
    E(m,n)=\begin{cases}
        E^\bullet(m,n) & \text{if }m+n\geq 3\text{ and }m,n\geq 1\\
        0 & \text{else.}
    \end{cases}
\end{equation*}
\label{eqn:holiebdiff)}
\begin{definition}
    Let $(\mathcal{H}olieb_{p,q},d)$ be the free dg properad generated by the $\mathbb{S}$-bimodule $\{E(m,n)\}_{m,n\in\mathbb{N}}$. Define the differential on generators by splitting the vertex into two and summing over all partitions of the in and out legs. That is
\begin{equation}\label{eq:holiebdiff)}
d
\resizebox{14mm}{!}{\begin{xy}
 <0mm,0mm>*{\circ};<0mm,0mm>*{}**@{},
 <-0.6mm,0.44mm>*{};<-8mm,5mm>*{}**@{-},
 <-0.4mm,0.7mm>*{};<-4.5mm,5mm>*{}**@{-},
 <0mm,0mm>*{};<-1mm,5mm>*{\ldots}**@{},
 <0.4mm,0.7mm>*{};<4.5mm,5mm>*{}**@{-},
 <0.6mm,0.44mm>*{};<8mm,5mm>*{}**@{-},
   <0mm,0mm>*{};<-8.5mm,5.5mm>*{^1}**@{},
   <0mm,0mm>*{};<-5mm,5.5mm>*{^2}**@{},
   <0mm,0mm>*{};<4.5mm,5.5mm>*{^{m\hspace{-0.5mm}-\hspace{-0.5mm}1}}**@{},
   <0mm,0mm>*{};<9.0mm,5.5mm>*{^m}**@{},
 <-0.6mm,-0.44
 mm>*{};<-8mm,-5mm>*{}**@{-},
 <-0.4mm,-0.7mm>*{};<-4.5mm,-5mm>*{}**@{-},
 <0mm,0mm>*{};<-1mm,-5mm>*{\ldots}**@{},
 <0.4mm,-0.7mm>*{};<4.5mm,-5mm>*{}**@{-},
 <0.6mm,-0.44mm>*{};<8mm,-5mm>*{}**@{-},
   <0mm,0mm>*{};<-8.5mm,-6.9mm>*{^1}**@{},
   <0mm,0mm>*{};<-5mm,-6.9mm>*{^2}**@{},
   <0mm,0mm>*{};<4.5mm,-6.9mm>*{^{n\hspace{-0.5mm}-\hspace{-0.5mm}1}}**@{},
   <0mm,0mm>*{};<9.0mm,-6.9mm>*{^n}**@{},
 \end{xy}}
\ \ = \ \
 \sum_{\substack{[1,\ldots,m]=I_1\sqcup I_2\\
 {|I_1|\geq 0, |I_2|\geq 1}}}
 \sum_{\substack{[1,\ldots,n]=J_1\sqcup J_2 \\
 {|J_1|\geq 1, |J_2|\geq 1} }
}\hspace{0mm}
\pm
\resizebox{22mm}{!}{ \begin{xy}
 <0mm,0mm>*{\circ};<0mm,0mm>*{}**@{},
 <-0.6mm,0.44mm>*{};<-8mm,5mm>*{}**@{-},
 <-0.4mm,0.7mm>*{};<-4.5mm,5mm>*{}**@{-},
 <0mm,0mm>*{};<0mm,5mm>*{\ldots}**@{},
 <0.4mm,0.7mm>*{};<4.5mm,5mm>*{}**@{-},
 <0.6mm,0.44mm>*{};<12.4mm,4.8mm>*{}**@{-},
     <0mm,0mm>*{};<-2mm,7mm>*{\overbrace{\ \ \ \ \ \ \ \ \ \ \ \ }}**@{},
     <0mm,0mm>*{};<-2mm,9mm>*{^{I_1}}**@{},
 <-0.6mm,-0.44mm>*{};<-8mm,-5mm>*{}**@{-},
 <-0.4mm,-0.7mm>*{};<-4.5mm,-5mm>*{}**@{-},
 <0mm,0mm>*{};<-1mm,-5mm>*{\ldots}**@{},
 <0.4mm,-0.7mm>*{};<4.5mm,-5mm>*{}**@{-},
 <0.6mm,-0.44mm>*{};<8mm,-5mm>*{}**@{-},
      <0mm,0mm>*{};<0mm,-7mm>*{\underbrace{\ \ \ \ \ \ \ \ \ \ \ \ \ \ \
      }}**@{},
      <0mm,0mm>*{};<0mm,-10.6mm>*{_{J_1}}**@{},
 <13mm,5mm>*{};<13mm,5mm>*{\circ}**@{},
 <12.6mm,5.44mm>*{};<5mm,10mm>*{}**@{-},
 <12.6mm,5.7mm>*{};<8.5mm,10mm>*{}**@{-},
 <13mm,5mm>*{};<13mm,10mm>*{\ldots}**@{},
 <13.4mm,5.7mm>*{};<16.5mm,10mm>*{}**@{-},
 <13.6mm,5.44mm>*{};<20mm,10mm>*{}**@{-},
      <13mm,5mm>*{};<13mm,12mm>*{\overbrace{\ \ \ \ \ \ \ \ \ \ \ \ \ \ }}**@{},
      <13mm,5mm>*{};<13mm,14mm>*{^{I_2}}**@{},
 <12.4mm,4.3mm>*{};<8mm,0mm>*{}**@{-},
 <12.6mm,4.3mm>*{};<12mm,0mm>*{\ldots}**@{},
 <13.4mm,4.5mm>*{};<16.5mm,0mm>*{}**@{-},
 <13.6mm,4.8mm>*{};<20mm,0mm>*{}**@{-},
     <13mm,5mm>*{};<14.3mm,-2mm>*{\underbrace{\ \ \ \ \ \ \ \ \ \ \ }}**@{},
     <13mm,5mm>*{};<14.3mm,-4.5mm>*{_{J_2}}**@{},
 \end{xy}}.
\end{equation}
For details of the signs we refer to \cite{MaVo}.
We define $(\mathcal{H}olieb^\bullet_{p,q},d)$ to be the free dg properad generated by the $\mathbb{S}$-bimodule $\{E^\bullet(m,n)\}_{m,n\in\mathbb{N}}$ with analogously defined differential. Also define the dg properad $(\mathcal{H}olieb^+_{p,q},d)$ generated by the same generators as $\mathcal{H}olieb_{p,q}$ with one additional generator $
    \begin{tikzpicture}[node distance=1.2cm,auto,baseline=-0.1cm]
    \node[black] (m) at (0,0) {};
    \node[] (d) at (0,-0.55) {}
        edge [->] (m);
    \node[] (u) at (0,0.55) {}
        edge [<-] (m);
    \end{tikzpicture}$ of arity $(1,1)$.
\end{definition}
The properad $(\mathcal{H}olieb_{p,q},d)$ is a minimal model of $\mathcal{L}ieb_{p,q}$, shown in \cite{K,V}. Further it is a subproperad of both $\mathcal{H}olieb_{p,q}^+$ and $\mathcal{H}olieb_{p,q}^\bullet$. Denote the inclusions as $\pi^+:\mathcal{H}olieb_{p,q}\hookrightarrow \mathcal{H}olieb^+_{p,q}$ and $\pi^\bullet:\mathcal{H}olieb_{p,q}\hookrightarrow \mathcal{H}olieb^\bullet_{p,q}$ respectively.
\begin{remark}
    The elements of $\mathcal{H}olieb_{p,q}$ can be seen as \textit{oriented graphs} (graphs without any closed paths) with out- and in-hairs attached to the vertices such that each vertex have at least one outgoing edge or out-hair, at least one incoming edge or in-hair and at least trivalent with respect to edges and hairs.
\end{remark}
\subsection{Deformations of properads}
\begin{definition}
    Let $(A,d_A),(B,d_B)$ be dg properads and $f:A\rightarrow B$ a morphism of dg properads. That is $d_A\circ f=f\circ d_B$, for all $x,y\in A$ $f(x\circ_Ay)=f(x)\circ_Bf(y)$, and $|f|=0$. A \textit{derivation} of $f$ is a morphism of $\mathbb{S}$-bimodules $D:A\rightarrow B$ such that
    \begin{equation*}
        D(x\circ_A y)=D(x)\circ_Bf(y)-(-1)^{|x|}f(x)\circ_BD(y).
    \end{equation*}
    We denote by $Der(A,B)$ the set of derivations with respect to $f$.
    Note that this is a $\mathbb{K}$ graded vector space. Define the space of \textit{deformations} with respect to $f$ as to be the space of derivations shifted by $-1$. That is $Def(A,B)=Der(A,B)[-1]$.
    We will in this paper mostly study the deformations.
    We define a differential $d$ on $Def(A,B)$ by 
    $d(D):=d_B\circ D + f\circ D + (-1)^{|f|} D\circ f$ for all $D\in Def(A,B)$.
\end{definition}
\begin{remark}
    When $A$ is free, then any derivation $D$ is uniquely determined by its values on the generators of $A$. In particular $Def(\mathcal{F}ree(E),B)\cong Hom(E,B)$ and $Der(\mathcal{F}ree(E),B)\cong Hom(E,B)[-1]$, where $Hom$ stands for the linear space of morphisms of $\mathbb{S}$-bimodules.
\end{remark}
\begin{example}
    Let $Der(\mathcal{H}olieb^{\uparrow}_{p,q}$) be the derivation complex with respect to $\pi^\bullet:\mathcal{H}olieb^{\bullet}_{p,q}\rightarrow \mathcal{H}olieb_{p,q}$ and $Der^+(\mathcal{H}olieb^{\uparrow}_{p,q}$) be the derivation complex with respect to $\pi^+:\mathcal{H}olieb_{p,q}^{+}\rightarrow \mathcal{H}olieb_{p,q}$. Both properads $\mathcal{H}olieb^{\bullet}_{p,q}$ and $\mathcal{H}olieb^{+}_{p,q}$ are free and the deformation complexes are thus isomorphic to the space of linear morphisms from the generators of $\mathcal{H}olieb_{p,q}^\bullet$ and $\mathcal{H}olieb_{p,q}^+$ respectively to $\mathcal{H}olieb_{p,q}$.
    Hence the derivation complexes can be described as
    \begin{align*}
        Der(\mathcal{H}olieb^{\uparrow}_{p,q})&\cong\prod_{m,n\geq 0}(\mathcal{H}olieb_{p,q}^\uparrow(m,n)\otimes sgn_m^{\otimes |p|}\otimes sgn_n^{\otimes |q|})^{\mathbb{S}_m\times\mathbb{S}_n}[1+p(1-m)+q(1-n)]\\
        Der^+(\mathcal{H}olieb^{\uparrow}_{p,q})&\cong\prod_{m,n\geq 1}(\mathcal{H}olieb_{p,q}^\uparrow(m,n)\otimes sgn_m^{\otimes |p|}\otimes sgn_n^{\otimes |q|})^{\mathbb{S}_m\times\mathbb{S}_n}[1+p(1-m)+q(1-n)].
    \end{align*}
    where $\mathcal{H}olieb^\uparrow_{p,q}(m,n)$ is the set of elements represented by oriented graphs having $m$ outputs and $n$ inputs. 
    The differential $d$ of the derivation complex is given by the vertex splitting differential $d^{spl}$ from $\mathcal{H}olieb_{p,q}$ with the additional terms of attaching $(m,n)$ corollas to every hair for all integers $m,n$:
\begin{equation}
\label{eq:defholiebdiff}
 d\Gamma =
 d^{spl}\Gamma
  \pm
  \sum_{m,n}
\begin{tikzpicture}
        [shorten >=1pt,node distance=1.2cm,auto,baseline=-0.1cm]
        \node[black] (c) at (0,0) {};
        \node[] (u1) at (-1.4,0.6) {$\scriptstyle 1$}
            edge [-] (c);
        \node[] (u2) at (-0.8,0.6) {$\scriptstyle 2$}
            edge [-] (c);
        \node[] (u3) at (0,0.5) {$\scriptstyle \cdots$};
        \node[] (u4) at (0.8,0.6) {$\scriptstyle m$}
            edge [-] (c);
        \node[] (u5) at (1.4,0.6) {$\Gamma$}
            edge [-] (c);
        \node[] (l1) at (-1.4,-0.6) {$\scriptstyle 1$}
            edge [-] (c);
        \node[] (l2) at (-0.8,-0.6) {$\scriptstyle 2$}
            edge [-] (c);
        \node[] (l3) at (0,-0.5) {$\scriptstyle \cdots$};
        \node[] (l4) at (0.8,-0.6) {}
            edge [-] (c);
        \node[] (l5) at (1.4,-0.6) {$\scriptstyle n$}
            edge [-] (c);
    \end{tikzpicture}
    \mp
    \sum_{m,n}
    \begin{tikzpicture}
        [shorten >=1pt,node distance=1.2cm,auto,baseline=-0.1cm]
        \node[black] (c) at (0,0) {};
        \node[] (u1) at (-1.4,0.6) {$\scriptstyle1$}
            edge [-] (c);
        \node[] (u2) at (-0.8,0.6) {$\scriptstyle2$}
            edge [-] (c);
        \node[] (u3) at (0,0.5) {$\scriptstyle \cdots$};
        \node[] (u4) at (0.8,0.6) {}
            edge [-] (c);
        \node[] (u5) at (1.4,0.6) {$\scriptstyle m$}
            edge [-] (c);
        \node[] (l1) at (-1.4,-0.6) {$\scriptstyle 1$}
            edge [-] (c);
        \node[] (l2) at (-0.8,-0.6) {$\scriptstyle 2$}
            edge [-] (c);
        \node[] (l3) at (0,-0.5) {$\scriptstyle \cdots$};
        \node[] (l4) at (0.8,-0.6) {$\scriptstyle n$}
            edge [-] (c);
        \node[] (l5) at (1.4,-0.6) {$\Gamma$}
            edge [-] (c);
    \end{tikzpicture}
 \end{equation}
The sign rule for this formula can be found in \cite{MW1}, where the cohomology of these complexes are described in terms of Kontsevich graph complexes.
\end{example}
\begin{remark}
    Note that $Der(\mathcal{H}olieb^\uparrow_{p,q})=Der^+(\mathcal{H}olieb^{\uparrow}_{p,q})$ since $\mathcal{H}olieb_{p,q}^\uparrow(0,0)=\mathcal{H}olieb_{p,q}^\uparrow(1,0)=\mathcal{H}olieb_{p,q}^\uparrow(0,1)=0 $ due to any composition of operations in $\mathcal{H}olieb_{p,q}$ must have at least one output and one input.
\end{remark}
Let $Der(\widehat{\mathcal{H}olieb}^\uparrow_{p,q})$ denote the derivation complex over the genus completed properads $\pi^\bullet:\widehat{\mathcal{H}olieb}_{p,q}\rightarrow\widehat{\mathcal{H}olieb}_{p,q}^\bullet$.
The genus completion of $\mathcal{H}olieb_{p,q}$ has a complete topology, and the derivations are considered as continuous. In the derivation complex, this means that we allow infinite series of graphs as elements. It is easy to see that the loop number $b=e-v+1$ of a graph is invariant under the differential. Let $b_0Der(\widehat{\mathcal{H}olieb}^\uparrow_{p,q})$ denote the subcomplex of $Der(\widehat{\mathcal{H}olieb}^\uparrow_{c,d})$ generated by graphs with loop number zero.
\begin{theorem}\label{thm:onevertex}
    The cohomology of $b_0Der(\widehat{\mathcal{H}olieb}^\uparrow_{p,q})$ is generated by the sum of single vertex graphs
    \begin{equation*}
        \sum_{
        \substack{
            m,n\geq 1\\
            m+n\geq 3}}
            \ (m+n-2) \ 
    \begin{tikzpicture}[node distance=1.2cm,auto,baseline=-0.1cm]
        \node[black] (r) at (0,0) {};
        \node[invisible] (u1) at (-0.4,0.5) {}
            edge [<-] (r);
        \node[] (udot1) at (0,0.4) {$\cdots$};
        \node[invisible] (u2) at (0.4,0.5) {}
            edge [<-] (r);
        \node[invisible] (d1) at (-0.4,-0.5) {}
            edge [->] (r);
        \node[] (ddot1) at (0,-0.4) {$\cdots$};
        \node[invisible] (d2) at (0.4,-0.5) {}
            edge [->] (r);
        \node[] at (0,0.7) {$\overbrace{\ \ \ \ \ \ \ \ }$};
        \node[] at (0,-0.7) {$\underbrace{\ \ \ \ \ \ \ \ }$};
        \node[] at (0,1) {$\scriptstyle m$};
        \node[] at (0,-1) {$\scriptstyle n$};
            \end{tikzpicture}.
    \end{equation*}
\end{theorem}
\begin{proof}
    See \cite{MW1}.
\end{proof}

\subsection{Wheelification}\label{sec:wheelification}
Recall that a properad is an $\mathbb{S}$-module $\{P(m,n)\}_{m,n\in\mathbb{N}}$ together with two compositions, namely horizontal composition
\begin{equation*}
    \otimes:P(m_1,n_1)\otimes...\otimes P(m_k,n_k)\rightarrow P(m_1+...+m_k,n_1+...+n_k)
\end{equation*}
and vertical composition \cite{V}:
\begin{equation*}
    \circ: P(m,n)\otimes P(n,k)\rightarrow P(m,k).
\end{equation*}
A \textit{wheeled properad} is a properad equipped with additional contraction maps
\begin{equation*}
    \xi^i_j:P(m,n)\rightarrow P(m-1,n-1)\quad\text{ where }\quad 1\leq m,n,\, 1\leq i\leq m \text{ and } 1\leq j\leq n
\end{equation*}
which satisfy axioms (see \cite{M2,MMS} for full details). A wheeled properad is best understood from the basic example of the endomorphism properad $End_V=\{Hom(V^m,V^n)\}_{m,n\in\mathbb{N}}$ for any finite dimensional vector space $V$. There the maps $\xi_{i,j}:Hom(V^m,V^n)\rightarrow Hom(V^{m-1},V^{n-1})$ are induced by the standard trace map $Hom(V,V)\rightarrow\mathbb{K}$.

There is a pair of adjoint functors between the category of properads and the category of wheeled properads, the first being the forgetful functor. The second is called the \textit{wheelification functor}. If $\mathcal{P}$ is a free properad, then elements of $P$ consist of freely composing operations from the generating set such that no closed paths occur in the resulting operation when represented as a graph. The wheelification $\mathcal{P}^{\circlearrowleft}$ of $\mathcal{P}$ consists of freely composing generating operations where the latter restriction is removed. That is composition of operations can be represented by graphs with possibly closed paths. In particular, if $\mathcal{F}$ is a free properad, then elements of $\mathcal{F}^\circlearrowleft$ can be viewed as graphs with hairs attached to the vertices so that each vertex have in- and outputs of some operation. In contrast, such a graph in $\mathcal{F}$ needs to be oriented.

\subsection{The wheeled deformation complexes $Der(\mathcal{H}olieb_{p,q}^\circlearrowleft)$ and $Der^+(\mathcal{H}olieb_{p,q}^\circlearrowleft)$}\label{sec:Der*+}
\begin{definition}
    Let $Der(\mathcal{H}olieb_{p,q}^\circlearrowleft)$ be the genus completed derivation complex with respect to $\pi^\bullet:\mathcal{H}olieb_{p,q}^{\circlearrowleft\bullet}\rightarrow\mathcal{H}olieb_{p,q}$. Similarly let $Der^+(\mathcal{H}olieb_{p,q}^\circlearrowleft)$ be the genus completed derivation complex with respect to $\pi^+:\mathcal{H}olieb_{p,q}^{\circlearrowleft+}\rightarrow\mathcal{H}olieb_{p,q}^\circlearrowleft$. Since $\mathcal{H}olieb_{p,q}^{\circlearrowleft\bullet}$ and $\mathcal{H}olieb_{p,q}^{\circlearrowleft+}$ are free properads, the derivation complexes can be described as
    \begin{align*}
        Der(\mathcal{H}olieb^{\circlearrowleft}_{p,q})&\cong\prod_{m,n\geq 0}(\mathcal{H}olieb_{p,q}^\circlearrowleft(m,n)\otimes sgn_m^{\otimes |p|}\otimes sgn_n^{\otimes |q|})^{\mathbb{S}_m\times\mathbb{S}_n}[1+p(1-m)+q(1-n)]\\
        Der^+(\mathcal{H}olieb^{\circlearrowleft}_{p,q})&\cong\prod_{m,n\geq 1}(\mathcal{H}olieb_{p,q}^\circlearrowleft(m,n)\otimes sgn_m^{\otimes |p|}\otimes sgn_n^{\otimes |q|})^{\mathbb{S}_m\times\mathbb{S}_n}[1+p(1-m)+q(1-n)].
    \end{align*}
    where $\mathcal{H}olieb_{p,q}^\circlearrowleft(m,n)$ is the set of elements represented by graphs with $m$ outputs and $n$ inputs.
\end{definition}
\begin{remark}
    From the decomposition above we see that the derivation complex $Der(\mathcal{H}olieb_{p,q}^\circlearrowleft)$ split as
    \begin{equation*}
        Der(\mathcal{H}olieb_{p,q}^\circlearrowleft)=Der^0(\mathcal{H}olieb_{p,q}^\circlearrowleft)\oplus Der^*(\mathcal{H}olieb_{p,q}^\circlearrowleft)
    \end{equation*}
    where 
    \begin{align*}
        Der^0(\mathcal{H}olieb_{p,q}^\circlearrowleft) =& (\mathcal{H}olieb_{p,q}^\circlearrowleft(0,0)\otimes sgn_0^{\otimes |p|}\otimes sgn_0^{\otimes |q|})^{\mathbb{S}_0\times\mathbb{S}_0}[1+p+q]\\
        Der^*(\mathcal{H}olieb_{p,q}^\circlearrowleft) =& \prod_{\substack{m,n\geq 0 \\ m+n\geq1}}(\mathcal{H}olieb_{p,q}^\circlearrowleft(m,n)\otimes sgn_m^{\otimes |p|}\otimes sgn_n^{\otimes |q|})^{\mathbb{S}_m\times\mathbb{S}_n}[1+p(1-m)+q(1-n)].
    \end{align*}
    Recall that the loop number is invariant under the differential. The complex $Der^0(\mathcal{H}olieb_{p,q}^\circlearrowleft)$ consists of graphs with no in- or outputs, and is generated by graphs where each vertex is at least trivalent and is neither a source nor a target. It hence consists of graphs with loop number at least two or higher.
    Let $b_0Der(\mathcal{H}olieb_{p,q}^\circlearrowleft)$ be the  component with loop number zero of $Der(\mathcal{H}olieb_{p,q}^\circlearrowleft)$. We note that $b_0Der(\mathcal{H}olieb_{p,q}^\circlearrowleft)=b_0Der^*(\mathcal{H}olieb_{p,q}^\circlearrowleft)=b_0Der^+(\mathcal{H}olieb_{p,q}^\circlearrowleft)$ and so the cohomology of the latter two complexes is generated by the same sum of one-vertex graphs as in theorem \ref{thm:onevertex}.
\end{remark}

\section{Graph complexes}

\subsection{The Kontsevich graph complex}
Let $e,v\in\mathbb{N}$ and let $\overline{\mathrm{V}}_v\overline{\mathrm{E}}_e\mathrm{cgra}$ be the set of connected directed graphs with $e$ edges and $v$ vertices labelled from $1$ to $e$ and from $1$ to $v$ respectively. Tadpoles and multiple edges are allowed.
Let $k\in\mathbb{Z}$ and let
\begin{equation*}
    \overline{\mathrm{V}}_v\overline{\mathrm{E}}_e\mathrm{GC}_k=\langle\overline{\mathrm{V}}_v\overline{\mathrm{E}}_e\mathrm{cgra}\rangle[-(v-1)k-(1-k)e]    
\end{equation*}
be the graded vector space over $\mathbb{K}$ of formal power series of graphs. The vector space is concentrated in degree $(v-1)k+(1-k)e$.
There is a natural right action of the group $\mathbb{S}_v\times\mathbb{S}_e\times\mathbb{S}_2^{\times e}$ on $\overline{\mathrm{V}}_v\overline{\mathrm{E}}_e\mathrm{GC}_k$ where $\mathbb{S}_v$ permutes vertices, $\mathbb{S}_e$ permutes edges and $\mathbb{S}_2^{\times e}$ acts on edges reversing their direction.
\begin{definition}\label{def:GC}
The \textit{full and connected Kontsevich graph complex} $(\mathsf{cfGC}_k,d)$ is the chain complex where
\begin{equation*}
        \mathsf{cfGC}_k:=
        \begin{cases}
        \prod_{e,v} \Big(\mathrm{\Bar{V}}_v\mathrm{\Bar{E}}_e\mathrm{GC}_k \otimes \mathrm{sgn}_e \Big)_{\mathbb{S}_v\times \mathbb{S}_e \times\mathbb{S}_2^{\times e}} & \text{ for }k \text{ even,}\\
        \prod_{e,v} \Big(\mathrm{\Bar{V}}_v\mathrm{\Bar{E}}_e\mathrm{GC}_k \otimes \mathrm{sgn}_v \otimes \mathrm{sgn}_2^{\otimes e}\Big)_{\mathbb{S}_v\times \mathbb{S}_e \times\mathbb{S}_2^{\times e}} & \text{ for }k \text{ odd.}
        \end{cases}
\end{equation*}
The subscript denotes taking the space of coinvariants under the group action.
The differential $d$ is of degree 1 and is defined on graphs as
\begin{equation*}
    d(\Gamma):=\delta(\Gamma)-\delta'(\Gamma)-\delta''(\Gamma)=\sum_{x\in V(\Gamma)}\delta_x(\Gamma)-\delta_x'(\Gamma)-\delta_x''(\Gamma)
\end{equation*}
where $V(\Gamma)$ is the set of vertices of $\Gamma$, $\delta_x(\Gamma)$ denotes the \textit{splitting} of the vertex $x$ which is a sum of graphs where $x$ has been replaced by $\begin{tikzpicture}[baseline={([yshift=-.5ex]current bounding box.center)}]
        \node[black] (l) at (0,0) {};
        \node[black] (r) at (0.9,0) {}
            edge [<-] (l);
    \end{tikzpicture}$ and the sum is over all possible ways of reattaching the edges originally connected to $x$ to the two new vertices. The $\delta_x'(\Gamma)$ stands for the graph where an outgoing edge is attached to $x$ with a univalent vertex in the other end, and $\delta_x''(\Gamma)$ stands for the graph where an incoming edge is attached to $x$ in a similar way.
The signs of the resulting graphs is determined so that the new edge is labeled $e+1$, the source vertex of the edge $e+1$ is labeled with the original vertex $x$ and the target vertex is labeled with $v+1$. Note that no univalent vertices are created under the action of the differential since any such graphs in $\delta$ and $\delta'+\delta''$ cancel each other.
\end{definition}
Elements of $\mathsf{cfGC}_k$ can be viewed as equivalence classes of undirected graphs up to a sign depending on labellings on vertices or edges. When representing a graph, we normally pick a representative and label edges, vertices and include directions of edges.
\begin{example}\label{ex:vertexsplit}
    Consider a vertex $x$ with one outgoing edge and two incoming edges in some representative graph $\Gamma$ in $\fcGC_k$. Pictorially we draw this as
    \begin{equation*}
        \begin{tikzpicture}[baseline={([yshift=-.5ex]current bounding box.center)}]
        \node[black,label=west:$\scriptstyle x$] (c) at (0,0) {};
        \node[black,label=west:$\scriptstyle i_2$] (b1) at (-0.35,-0.5) {}
            edge [->] (c);
        \node[black,label=east:$\scriptstyle i_3$] (b2) at (0.35,-0.5) {}
            edge [->] (c);
        \node[black,label=west:$\scriptstyle i_1$] (t1) at (0,0.6) {}
            edge [<-] (c);
    \end{tikzpicture}
    \end{equation*}
    where eventual vertices and edges not adjacent to $x$ have been omitted. Then $\delta_x$ acts on $\Gamma$ as
\begin{align*}
    \delta_x\Bigg(
    \begin{tikzpicture}[baseline={([yshift=-.5ex]current bounding box.center)}]
        \node[black,label=west:$\scriptstyle x$] (c) at (0,0) {};
        \node[black,label=west:$\scriptstyle i_2$] (b1) at (-0.35,-0.5) {}
            edge [->] (c);
        \node[black,label=east:$\scriptstyle i_3$] (b2) at (0.35,-0.5) {}
            edge [->] (c);
        \node[black,label=west:$\scriptstyle i_1$] (t1) at (0,0.6) {}
            edge [<-] (c);
    \end{tikzpicture}\Bigg)
    & \ =
    \begin{tikzpicture}[auto,baseline={([yshift=-.5ex]current bounding box.center)}]
        \node[black,label=west:$\scriptstyle x$] (l) at (0,0) {};
        \node[black,label=east:$\scriptstyle v+1$] (r) at (0.9,0) {}
            edge [<-] node [swap] {$\scriptstyle e+1$} (l);
        \node[black,label=west:$\scriptstyle i_2$] (b1) at (-0.35,-0.5) {}
            edge [->] (l);
        \node[black,label=east:$\scriptstyle i_3$] (b2) at (0.35,-0.5) {}
            edge [->] (l);
        \node[black,label=west:$\scriptstyle i_1$] (t1) at (0,0.6) {}
            edge [<-] (l);
    \end{tikzpicture}
    \ + \ 
    \begin{tikzpicture}[auto,baseline={([yshift=-.5ex]current bounding box.center)}]
        \node[black,label=west:$\scriptstyle x$] (l) at (0,0) {};
        \node[black,label=east:$\scriptstyle v+1$] (r) at (0.9,0) {}
            edge [<-] node [swap] {$\scriptstyle e+1$} (l);
        \node[black,label=west:$\scriptstyle i_2$] (b1) at (-0.35,-0.5) {}
            edge [->] (l);
        \node[black,label=east:$\scriptstyle i_3$] (b2) at (0.35,-0.5) {}
            edge [->] (l);
        \node[black,label=east:$\scriptstyle i_1$] (t1) at (0.9,0.6) {}
            edge [<-] (r);
    \end{tikzpicture}
    \ + \ 
    \begin{tikzpicture}[auto,baseline={([yshift=-.5ex]current bounding box.center)}]
        \node[black,label=west:$\scriptstyle x$] (l) at (0,0) {};
        \node[black,label=east:$\scriptstyle v+1$] (r) at (0.9,0) {}
            edge [<-] node [swap] {$\scriptstyle e+1$} (l);
        \node[black,label=west:$\scriptstyle i_3$] (b1) at (0,-0.5) {}
            edge [->] (l);
        \node[black,label=east:$\scriptstyle i_2$] (b2) at (0.9,-0.5) {}
            edge [->] (r);
        \node[black,label=west:$\scriptstyle i_1$] (t1) at (0,0.6) {}
            edge [<-] (l);
    \end{tikzpicture}
    \ + \
    \begin{tikzpicture}[auto,baseline={([yshift=-.5ex]current bounding box.center)}]
        \node[black,label=west:$\scriptstyle x$] (l) at (0,0) {};
        \node[black,label=east:$\scriptstyle v+1$] (r) at (0.9,0) {}
            edge [<-] node [swap] {$\scriptstyle e+1$} (l);
        \node[black,label=west:$\scriptstyle i_2$] (b1) at (0,-0.5) {}
            edge [->] (l);
        \node[black,label=east:$\scriptstyle i_3$] (b2) at (0.9,-0.5) {}
            edge [->] (r);
        \node[black,label=west:$\scriptstyle i_1$] (t1) at (0,0.6) {}
            edge [<-] (l);
    \end{tikzpicture}
    \\ 
    &\ + \ 
    \begin{tikzpicture}[auto,baseline={([yshift=-.5ex]current bounding box.center)}]
        \node[black,label=west:$\scriptstyle x$] (l) at (0,0) {};
        \node[black,label=east:$\scriptstyle v+1$] (r) at (0.9,0) {}
            edge [<-] node [swap] {$\scriptstyle e+1$} (l);
        \node[black,,label=west:$\scriptstyle i_3$] (b1) at (0,-0.5) {}
            edge [->] (l);
        \node[black,label=east:$\scriptstyle i_2$] (b2) at (0.9,-0.5) {}
            edge [->] (r);
        \node[black,label=east:$\scriptstyle i_1$] (t1) at (0.9,0.6) {}
            edge [<-] (r);
    \end{tikzpicture}
    \ + \
    \begin{tikzpicture}[auto,baseline={([yshift=-.5ex]current bounding box.center)}]
        \node[black,label=west:$\scriptstyle x$] (l) at (0,0) {};
        \node[black,label=east:$\scriptstyle v+1$] (r) at (0.9,0) {}
            edge [<-] node [swap] {$\scriptstyle e+1$} (l);
        \node[black,label=east:$\scriptstyle i_3$] (b1) at (0.9,-0.5) {}
            edge [->] (r);
        \node[black,label=west:$\scriptstyle i_2$] (b2) at (0,-0.5) {}
            edge [->] (l);
        \node[black,label=east:$\scriptstyle i_1$] (t1) at (0.9,0.6) {}
            edge [<-] (r);
    \end{tikzpicture}
    \ + \ 
    \begin{tikzpicture}[auto,baseline={([yshift=-.5ex]current bounding box.center)}]
        \node[black,label=west:$\scriptstyle x$] (l) at (0,0) {};
        \node[black,label=east:$\scriptstyle v+1$] (r) at (0.9,0) {}
            edge [<-] node [swap] {$\scriptstyle e+1$} (l);
        \node[black,label=west:$\scriptstyle i_2$] (b1) at (0.55,-0.5) {}
            edge [->] (r);
        \node[black,label=east:$\scriptstyle i_3$] (b2) at (1.25,-0.5) {}
            edge [->] (r);
        \node[black,label=west:$\scriptstyle i_1$] (t1) at (0,0.6) {}
            edge [<-] (l);
    \end{tikzpicture}
    \ + \ 
    \begin{tikzpicture}[auto,baseline={([yshift=-.5ex]current bounding box.center)}]
        \node[black,label=west:$\scriptstyle x$] (l) at (0,0) {};
        \node[black,label=east:$\scriptstyle v+1$] (r) at (0.9,0) {}
            edge [<-] node [swap] {$\scriptstyle e+1$} (l);
        \node[black,label=west:$\scriptstyle i_2$] (b1) at (0.55,-0.5) {}
            edge [->] (r);
        \node[black,label=east:$\scriptstyle i_3$] (b2) at (1.25,-0.5) {}
            edge [->] (r);
        \node[black,label=east:$\scriptstyle i_1$] (t1) at (0.9,0.6) {}
            edge [<-] (r);
    \end{tikzpicture}.
\end{align*}
    Similarly $\delta_x'$ and $\delta_x''$ act on $\Gamma$ as
    \begin{equation*}
    \delta_x'\Bigg(
    \begin{tikzpicture}[baseline={([yshift=-.5ex]current bounding box.center)}]
        \node[black,label=west:$\scriptstyle x$] (c) at (0,0) {};
        \node[black,label=west:$\scriptstyle i_2$] (b1) at (-0.35,-0.5) {}
            edge [->] (c);
        \node[black,label=east:$\scriptstyle i_3$] (b2) at (0.35,-0.5) {}
            edge [->] (c);
        \node[black,label=west:$\scriptstyle i_1$] (t1) at (0,0.6) {}
            edge [<-] (c);
    \end{tikzpicture}\Bigg)
    \ = \
    \begin{tikzpicture}[auto,baseline={([yshift=-.5ex]current bounding box.center)}]
        \node[black,label=west:$\scriptstyle x$] (l) at (0,0) {};
        \node[black,label=east:$\scriptstyle v+1$] (r) at (0.9,0) {}
            edge [<-] node [swap] {$\scriptstyle e+1$} (l);
        \node[black,label=west:$\scriptstyle i_2$] (b1) at (-0.35,-0.5) {}
            edge [->] (l);
        \node[black,label=east:$\scriptstyle i_3$] (b2) at (0.35,-0.5) {}
            edge [->] (l);
        \node[black,label=west:$\scriptstyle i_1$] (t1) at (0,0.6) {}
            edge [<-] (l);
    \end{tikzpicture}\ ,
    \quad\quad
    \delta_x''\Bigg(
    \begin{tikzpicture}[baseline={([yshift=-.5ex]current bounding box.center)}]
        \node[black,label=west:$\scriptstyle x$] (c) at (0,0) {};
        \node[black,,label=west:$\scriptstyle i_2$] (b1) at (-0.35,-0.5) {}
            edge [->] (c);
        \node[black,label=east:$\scriptstyle i_3$] (b2) at (0.35,-0.5) {}
            edge [->] (c);
        \node[black,label=west:$\scriptstyle i_1$] (t1) at (0,0.6) {}
            edge [<-] (c);
    \end{tikzpicture} \Bigg)
    \ = \
    \begin{tikzpicture}[auto,baseline={([yshift=-.5ex]current bounding box.center)}]
        \node[black,label=west:$\scriptstyle x$] (l) at (0,0) {};
        \node[black,label=east:$\scriptstyle v+1$] (r) at (0.9,0) {}
            edge [<-] node [swap] {$\scriptstyle e+1$} (l);
        \node[black,label=west:$\scriptstyle i_2$] (b1) at (0.55,-0.5) {}
            edge [->] (r);
        \node[black,label=east:$\scriptstyle i_3$] (b2) at (1.25,-0.5) {}
            edge [->] (r);
        \node[black,label=west:$\scriptstyle i_1$] (t1) at (0.9,0.6) {}
            edge [<-] (r);
    \end{tikzpicture}.
\end{equation*}
\end{example}
\begin{remark}
    The graph complex $(\mathsf{fGC}_k,d)$ of not necessarily connected graphs can be described in terms of the complex of connected graphs, so no information is lost by considering the smaller complex. More specifically $\mathsf{cfGC}_k=S^+(\mathsf{fGC}_k[-k])[k]$, where $S^+(V)$ denotes the (completed) symmetric product space of the (dg) vector space $V$ \cite{W1}.
\end{remark}
\begin{remark}
    From the definition of $\mathsf{cfGC}_k$, we note that complexes of the same parity are isomorphic up to a degree shift. The only two crucial complexes to study are thus $\mathsf{cfGC}_2$ and $\mathsf{cfGC}_3$. As we will soon see, we know much more about the first one compared to the latter.
\end{remark}
The term \textit{full} refers to that there is no restriction on which types of graphs we consider as generators. Let $\mathsf{b}_2\GC_k$ denote the subcomplex of graphs only consisting of bivalent vertices, and $\GC_k$ the subcomplex of graphs whose vertices are of valency three or higher and contain no tadpoles. We call $\GC_k$ the \textit{Kontsevich graph complex}. The inclusion
\begin{equation*}
    \mathsf{b}_2\GC_k\oplus\GC_k\hookrightarrow\mathsf{cfGC}_k
\end{equation*}
is a quasi-isomorphism \cite{W2}.
The complex $\mathsf{b}_2\GC_k$ consists of loop-graphs and its cohomology is fully described as
\begin{equation*}
    H(\mathsf{b}_2\GC_k)=\bigoplus_{\substack{i\geq 1 \\ i\equiv 2k+1 \mod 4}}\mathbb{K}[k-i]
\end{equation*}
where $\mathbb{K}[k-i]$ denotes the loop-graph containing $i$ edges. The cohomology of $\GC_k$ is only partially understood in negative degrees and degree zero for $k=2$. The following remarkable results was shown by T. Willwacher in \cite{W1}.
\begin{theorem}
    There is a Lie-algebra structure on $H^0(\GC_2)$ which is isomorhpic to the Grothendieck-Teichmuller Lie algebra $\mathfrak{grt}_0$.
    \begin{equation*}
        H^0(\GC_2)\cong\mathfrak{grt}_0
    \end{equation*}
    Furthermore the cohomology in negative degrees vanish.
    \begin{equation*}
        H^{<0}(\GC_2)=0.
    \end{equation*}
\end{theorem}
In the dual complex $\mathsf{gc}_k$ where the differential is defined by contracting edges, the zeroth homology contain the the wheel classes $\{\omega_{2n+1}\}_{n\in\mathbb{N}}$:
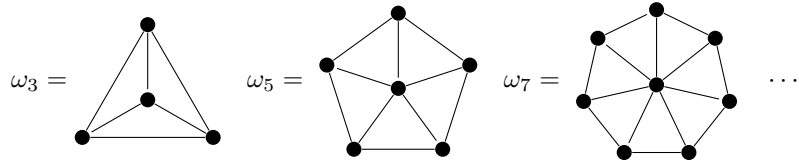
\begin{equation*}
    \omega_3 = \begin{tikzpicture}[shorten >=1pt,node distance=1.2cm,auto,>=angle 90,scale=1,baseline={([yshift=-.5ex]current bounding box.center)}]
    \node[black] (t0) at (0,1) {};
    \node[black] (t1) at (0.87,-0.5) {};
    \node[black] (t2) at (-0.87,-0.5) {};
    \node[black] (c) at (0,0) {};
    \draw (t0) -- (t1) -- (t2) -- (t0) -- (c);
    \draw (t1) -- (c) -- (t2);
    \end{tikzpicture}
    \quad \omega_5 = \ \begin{tikzpicture}[shorten >=1pt,node distance=1.2cm,auto,>=angle 90,scale=1,baseline={([yshift=-.5ex]current bounding box.center)}]
    \node[black] (t0) at (0,1) {};
    \node[black] (t1) at (0.95,0.31) {};
    \node[black] (t2) at (0.59,-0.81) {};
    \node[black] (t3) at (-0.59,-0.81) {};
    \node[black] (t4) at (-0.95,0.31) {};
    \node[black] (c) at (0,0) {};
    \draw (t0) -- (t1) -- (t2) -- (t3) -- (t4) -- (t0) -- (c);
    \draw (t1) -- (c) -- (t2);
    \draw (t3) -- (c) -- (t4);
    \end{tikzpicture}
    \quad \omega_7 = \ \begin{tikzpicture}[shorten >=1pt,node distance=1.2cm,auto,>=angle 90,scale=1,baseline={([yshift=-.5ex]current bounding box.center)}]
    \node[black] (c) at (0,0) {};
    \node[black] (t0) at (0,1) {};
    \node[black] (t1) at (0.78,0.62) {};
    \node[black] (t2) at (0.97,-0.22) {};
    \node[black] (t3) at (0.43,-0.9) {};
    \node[black] (t4) at (-0.43,-0.9) {};
    \node[black] (t5) at (-0.97,-0.22) {};
    \node[black] (t6) at (-0.78,0.62) {};
    \draw (c) -- (t0) -- (t1) -- (t2) -- (t3) -- (t4) -- (t5) -- (t6) -- (t0) ;
    \draw (t1) -- (c) -- (t2);
    \draw (t3) -- (c) -- (t4);
    \draw (t5) -- (c) -- (t6);
    \end{tikzpicture} \quad \cdots
\end{equation*}
It is conjectured that these are these are the only classes of homology zero. In $\GC_2$, the first wheel classes are represented by the following graphs:
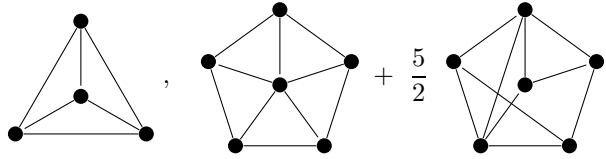
\begin{equation*}
    \begin{tikzpicture}[shorten >=1pt,node distance=1.2cm,auto,>=angle 90,scale=1,baseline={([yshift=-.5ex]current bounding box.center)}]
    \node[black] (t0) at (0,1) {};
    \node[black] (t1) at (0.87,-0.5) {};
    \node[black] (t2) at (-0.87,-0.5) {};
    \node[black] (c) at (0,0) {};
    \draw (t0) -- (t1) -- (t2) -- (t0) -- (c);
    \draw (t1) -- (c) -- (t2);
    \end{tikzpicture}
    \ , \quad 
    \begin{tikzpicture}[shorten >=1pt,node distance=1.2cm,auto,>=angle 90,scale=1,baseline={([yshift=-.5ex]current bounding box.center)}]
    \node[black] (t0) at (0,1) {};
    \node[black] (t1) at (0.95,0.31) {};
    \node[black] (t2) at (0.59,-0.81) {};
    \node[black] (t3) at (-0.59,-0.81) {};
    \node[black] (t4) at (-0.95,0.31) {};
    \node[black] (c) at (0,0) {};
    \draw (t0) -- (t1) -- (t2) -- (t3) -- (t4) -- (t0) -- (c);
    \draw (t1) -- (c) -- (t2);
    \draw (t3) -- (c) -- (t4);
    \end{tikzpicture} \ + \ \frac{5}{2} \ \
    \begin{tikzpicture}[shorten >=1pt,node distance=1.2cm,auto,>=angle 90,scale=1,baseline={([yshift=-.5ex]current bounding box.center)}]
    \node[black] (t0) at (0,1) {};
    \node[black] (t1) at (0.95,0.31) {};
    \node[black] (t2) at (0.59,-0.81) {};
    \node[black] (t3) at (-0.59,-0.81) {};
    \node[black] (t4) at (-0.95,0.31) {};
    \node[black] (c) at (0,0) {};
    \draw (t0) -- (t1) -- (t2) -- (t3) -- (t4) -- (t0);
    \draw (t0) -- (c) -- (t1);
    \draw (t2) -- (t4);
    \draw (t0) -- (t3) -- (c);
    \end{tikzpicture}
\end{equation*}
The cohomology of $\GC_3$ is not as known as $\GC_2$, but we know that it is concentrated in negative degrees. This can be derived from more general statement. Note that the loop number $b=e-v+1$ of a graph is invariant under the differential. Let $\mathsf{b}_b\GC_{k}$ be the subcomplex of $\GC_k$ of graphs with loop number $b$. The graphs of $\GC_k$ have a loop number that is at least 2 or higher.
\begin{proposition}
    Let $b\geq 2$. Then $H^l(\mathsf{b}_b\GC_{k})=0$ for $(3-k)b-3< l$.
\end{proposition}
\begin{proof}
    The degree of a graph $\Gamma\in\mathsf{b}_b\GC_{k}$ can be rewritten as 
    \begin{equation*}
        |\Gamma|=(v-1)k-(k-1)e=(v-e-1)k+e=-bk+e.
    \end{equation*}
    Now $3v\leq 2e$ since the vertices of $\Gamma$ are at least trivalent. Then we get the inequality
    \begin{equation*}
        2b=2e-2v+2\geq v+2 \implies v\leq 2b-2.
    \end{equation*}
    Finally the degree of $|\Gamma|\leq (3-k)b-3$ since
    \begin{equation*}
        |\Gamma|=-bk+e=-bk+(b+v-1)=(1-k)b+v-1\leq (1-k)b+2b-3=(3-k)b-3.
    \end{equation*}
    Since $b$ and $k$ are independent from $\Gamma$, there are no graphs of such degrees in $\mathsf{b}_b\GC_k$.
\end{proof}
\begin{corollary}
    $H^l(\GC_3)=0$ for $l>-3$.
\end{corollary}
\subsection{The directed Kontsevich graph complex}
\begin{definition}
    The \textit{full and connected directed Kontsevich graph complex} $(\mathsf{fdGC},d)$ is the chain complex defined by
    \begin{equation*}
        \mathsf{cfdGC}_k:=
        \begin{cases}
        \prod_{e,v} \Big(\mathrm{\Bar{V}}_v\mathrm{\Bar{E}}_e\mathrm{GC}_k \otimes \mathrm{sgn}_e \Big)_{\mathbb{S}_v\times \mathbb{S}_e} & \text{ for }k \text{ even,}\\
        \prod_{e,v} \Big(\mathrm{\Bar{V}}_v\mathrm{\Bar{E}}_e\mathrm{GC}_k \otimes \mathrm{sgn}_v \Big)_{\mathbb{S}_v\times \mathbb{S}_e} & \text{ for }k \text{ odd.}
        \end{cases}
\end{equation*}
    The elements are now equivalence classes of directed graphs. The differential is defined using the same formula $d(\Gamma)=\delta(\Gamma)-\delta'(\Gamma)-\delta''(\Gamma)$ 
    in definition \ref{def:GC}.
\end{definition}
Let $\dGC_k$ denote the subcomplex of $\mathsf{cfdGC}_k$ consisting of graphs with no univalent vertices nor passing vertices (i.e bivalent vertices with exactly one incoming edge and one outgoing edge) and at least one vertex of valency three or higher. Also let $\dGC_k^2$ denote the subcomplex of graphs consisting only of bivalent vertices. The inclusion $\dGC_k^2\oplus\dGC_k\hookrightarrow\mathsf{cfdGC}_k$ is a quasi-isomorphism. We call $\dGC_k$ the \textit{directed Kontsevich graph complex}.
Let $f:\GC_k\rightarrow\dGC$ mapping an undirected graph $\Gamma$ to $f(\Gamma)$, where $f(\Gamma)$ is the sum of all possible ways of adding directions to the edges of $\Gamma$. T. Willwacher showed in \cite{W1} the following result.
\begin{proposition}
    The map $f:\GC_k\rightarrow\dGC_k$ is a chain map and furthermore a quasi-isomorphism.
\end{proposition}
Hence from a cohomology point of view, we can choose which of the two complexes to study.
\subsection{Subcomplexes of $\dGC_k$}
There are several subcomplexes of $\dGC_k$, whose cohomology have been studied and been related to other graph complexes. Here is an overview of the most important ones.
\begin{itemize}
    \item The \textit{oriented graph complex} $\OGC_k$. This is the subcomplex of graphs that do not contain any closed paths of directed edges. M. Zivkovic found an explicit chain-map $\dGC_{k}\rightarrow\OGC_{k+1}$, which he also showed to be a quasi-isomorphism \cite{Z1}. 
    \item The \textit{sourced graph complex} $\dGC^s_k$. It is the subcomplex of graphs that contain at least one source vertex. The inclusion $\OGC_k\hookrightarrow\dGC^s_k$ is a quasi-isomorphism \cite{Z2}.
    \item The \textit{targeted graph complex} $\dGC^t_k$. It is the subcomplex of graphs that contain at least one target vertex. It is naturally isomorphic to $\dGC^s_k$ by the map that reverses the direction of all edges of a graph. Similarly the inclusion $\OGC_k\hookrightarrow\dGC^t_k$ is a quasi-isomorphism \cite{Z2}.
    \item The \textit{sourced or targeted graph complex} $\dGC_k^{s+t}$. It is the subcomplex of graphs with at least one source or one target vertex. We have that $H^l(\dGC_3^{s+t})=0$ for $l\leq 1$ \cite{Z3}.
    \item The \textit{sourced and targeted graph complex} $\dGC_k^{st}$. This is the subcomplex of graphs that contain at least one source and one target vertex. There is a short exact sequence
    \begin{equation*}
        \xymatrix{
        0 \ar[r] & \dGC^{st}_k \ar[r] & \dGC^s_k\oplus\dGC^t_k \ar[r] & \dGC^{s+t}_k \ar[r] & 0\\
            & \Gamma \ar@{|->}[r] & (\Gamma,\Gamma)\\
            & & (\Gamma_1,\Gamma_2) \ar@{|->}[r] & \Gamma_1-\Gamma_2 }
    \end{equation*}
    Since $H^l(\dGC_3^{s+t})=0$ for $l\leq 1$ one sees that $H^0(\dGC^{st}_3)=H^0(\dGC^s_3)\oplus H(\dGC^t_3)$ \cite{Z3}. In particular we get the remarkable result:
    \begin{corollary}\label{cor:st-grt}
                $H^0(\dGC^{st}_3)=H^0(\GC_{2})\oplus H^0(\GC_2)=\mathfrak{grt}\oplus\mathfrak{grt}$.
    \end{corollary}
    \item The \textit{completely wheeled graph complex} $\dGC^\circlearrowleft_k$. This complex is defined as the quotient complex $\dGC_k/\dGC^{s+t}_k$. Hence it consists of graphs where all vertices are at least trivalent, and each vertex have at least one incoming and one outgoing edge. We get the short exact sequence
    \begin{equation*}
        \xymatrix{0 \ar[r] & \dGC^{s+t}_k \ar[r] & \dGC_k \ar[r] & \dGC^{\circlearrowleft}_k \ar[r] & 0}
    \end{equation*}
    Note in particular that due to $H^l(\dGC^{s+t}_3)=0$ for $l\leq 1$ and $H^l(\dGC_3)=0$ for $l\geq -2$ we get the following result.
    \begin{corollary}\label{cor:CohomDGC3}
        $H^l(\dGC^\circlearrowleft_3)=0 \ \text{for} \ -2\leq l \leq 1$.
    \end{corollary}
\end{itemize}

\section{The bi-weighted graph complex $\fWGC_k$}
When studying the complex $Der(\mathcal{H}olieb_{p,q}^{*\circlearrowleft})$, one note that the sign rules of the differential are rather complicated when expressed in terms of generating corollas using the defining formula \ref{eq:defholiebdiff}. In this section we introduce the \textit{bi-weighted graph complex} $\fWGC_k$, and show that it is isomorphic to $Der(\mathcal{H}olieb_{p,q}^{\circlearrowleft})$. The sign rules of $\fWGC_k$ are the same as for the Kontsevich graph complexes, i.e based on ordering of edges or vertices, and are generally easier to work with.
\subsection{Definition of the bi-weighted graph complex}
Let $\Gamma$ be a directed graph and $x$ a vertex of $\Gamma$. Let $|x|_{out}$ and $|x|_{in}$ denote the number of outgoing and incoming edges respectively from $x$. A \textit{bi-weight} of a vertex $x$ is a pair of non-negative integers $(w_x^{out},w_x^{in})$ satisfying
\begin{enumerate}
    \item $w_x^{out}+|x|_{out}\geq 1$,
    \item $w_x^{in}+|x|_{in}\geq 1$,
    \item $w_x^{out}+w_x^{in}+|x|_{out}+|x|_{in}\geq 3$.
\end{enumerate}
We refer to $w^{out}_x$ as the \textit{out-weight} and $w^{in}_x$ the \textit{in-weight} of $x$ respectively. A graphs whose vertices all carries bi-weights is called a \textit{bi-weighted} graph.
\begin{example}
    We include the bi-weights of a vertex when drawing a bi-weighted graph as
    $\begin{tikzpicture}[baseline={([yshift=-.5ex]current bounding box.center)}]
        \node[biw] (v) at (0,0) {$\scriptstyle{w_x^{out}}$\nodepart{lower}$\scriptstyle{w_x^{in}}$};
    \end{tikzpicture}$. A bi-weighted graphs might then look like
    
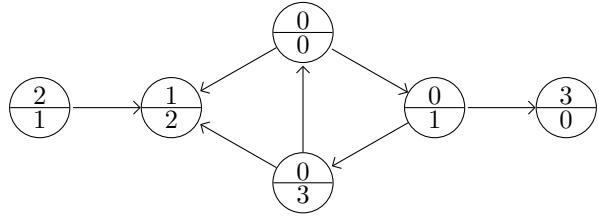
\begin{equation*}
    \begin{tikzpicture}[shorten >=1pt,>=angle 90,baseline={([yshift=-.5ex]current bounding box.center)}]
    \node[biw] (v1) at (-3.5,0) {$2$\nodepart{lower}$1$};
    \node[biw] (v2) at (-1.75,0) {$1$\nodepart{lower}$2$}
        edge [<-] (v1);
    \node[biw] (v3) at (0,1) {$0$\nodepart{lower}$0$}
        edge [->] (v2);
    \node[biw] (v4) at (0,-1) {$0$\nodepart{lower}$3$}
        edge [->] (v2)
        edge [->] (v3);
    \node[biw] (v5) at (1.75,0) {$0$\nodepart{lower}$1$}
        edge [<-] (v3)
        edge [->] (v4);
    \node[biw] (v6) at (3.5,0) {$3$\nodepart{lower}$0$}
        edge [<-] (v5);
\end{tikzpicture}
\end{equation*}.
\end{example}
Let $\overline{\mathrm{V}}_v\overline{\mathrm{E}}_e\mathrm{wcgra}$ be the set of connected bi-weighted graphs with $e$ edges and $v$ vertices.
Let $k\in\mathbb{Z}$ and let
\begin{equation*}
    \overline{\mathrm{V}}_v\overline{\mathrm{E}}_e\mathrm{WGC}_k:=\langle\overline{\mathrm{V}}_v\overline{\mathrm{E}}_e\mathrm{wcgra}\rangle[-(v-1)k-(1-k)e]
\end{equation*}
be the graded vector space over $\mathbb{K}$ of formal power series of bi-weighted graphs concentrated in degree $(v-1)k+(1-k)e$. We have an action of $\mathbb{S}_v\times\mathbb{S}_e$ on the vector space by permuting vertices and edges respectively.

\begin{definition}
The \textit{bi-weighted graph complex} $(\mathsf{f}\WGC_k,d)$ is the chain complex where
\begin{equation*}
        \mathsf{f}\WGC_k:=
        \begin{cases}
            \prod_{v,e} \Big(\mathrm{\Bar{V}}_v\mathrm{\Bar{E}}_e\mathrm{WGC}_d \otimes \mathrm{sgn}_e \Big)_{S_v\times S_e} & \text{ for }k \text{ even},\\
            \prod_{v,e} \Big(\mathrm{\Bar{V}}_v\mathrm{\Bar{E}}_e\mathrm{WGC}_d \otimes \mathrm{sgn}_v \Big)_{S_v\times S_e} & \text{ for }k \text{ odd}.
        \end{cases}
\end{equation*}
The differential $d$ is defined on a graph $\Gamma$ as
\begin{equation*}
    d(\Gamma):=\delta(\Gamma)-\delta'(\Gamma)-\delta''(\Gamma)=\sum_{x\in V(\Gamma)}\delta_x(\Gamma)-\delta_x'(\Gamma)-\delta_x''(\Gamma)
\end{equation*}
where $V(\Gamma)$ is the set of vertices of $\Gamma$, $\delta_x(\Gamma)$ denotes the splitting of the vertex $x$ similar to the vertex splitting in $\dGC_k$ with the addition that we sum over all possible ways of redistributing the bi-weight of $x$ to the two new vertices. The $\delta_x'(\Gamma)$ is the sum of graphs where we decrease the out-weight of $x$ by one and add an outgoing edge attached to a new univalent vertex. The summation is over all possible bi-weights on the new vertex.
The $\delta_x''(\Gamma)$ is defined similarly but where the in-weight is decreased and the edge is incoming to $x$. The signs of the graphs under the action of the differential are the same as for $\delta, \delta'$ and $\delta''$ in $\dGC_k$.
\end{definition}

 We pictorially represent the action of $d_x$ on a vertex $x$ as
\begin{equation*}
    d_x\Big(\ 
    \begin{tikzpicture}[shorten >=1pt,>=angle 90,baseline={([yshift=-.5ex]current bounding box.center)}]
        \node[biw] (v0) at (0,0) {$m$\nodepart{lower}$n$};
        \node[invisible] (v1) at (-0.4,0.8) {}
            edge [<-] (v0);
        \node[] at (0,0.6) {$\scriptstyle\cdots$};
        \node[invisible] (v3) at (0.4,0.8) {}
            edge [<-] (v0);
        \node[invisible] (v4) at (-0.4,-0.8) {}
            edge [->] (v0);
        \node[] at (0,-0.6) {$\scriptstyle\cdots$};
        \node[invisible] (v6) at (0.4,-0.8) {}
            edge [->] (v0);
    \end{tikzpicture}\ \Big) \ 
    = \ \sum_{\substack{\scriptstyle{m=m_1+m_2}\\ \scriptstyle{n=n_1+n_2}}}
    \begin{tikzpicture}[shorten >=1pt,>=angle 90,baseline={([yshift=-.5ex]current bounding box.center)}]
        \node[ellipse,
            draw = black,
            minimum width = 3cm, 
            minimum height = 1.8cm,
            dotted] (e) at (0,0) {};
        \node[biw] (vLeft) at (-0.8,0) {$m_1$\nodepart{lower}$n_1$};
        \node[biw] (vRight) at (0.8,0) {$m_2$\nodepart{lower}$n_2$}
            edge [<-] (vLeft);
        \node[invisible] (v1) at (-0.6,1.4) {}
            edge [<-] (e);
        \node[] at (0,1.2) {$\cdots$};
        \node[invisible] (v3) at (0.6,1.4) {}
            edge [<-] (e);
        \node[invisible] (v4) at (-0.6,-1.4) {}
            edge [->] (e);
        \node[] at (0,-1.2) {$\cdots$};
        \node[invisible] (v6) at (0.6,-1.4) {}
            edge [->] (e);
    \end{tikzpicture}
    \ -\ \sum_{\substack{i\geq 1,j\geq 0\\
    i+j\geq 2}}\ 
    \begin{tikzpicture}[shorten >=1pt,>=angle 90,baseline={(4ex,-0.5ex)}]
        \node[biw] (v0) at (0,0) {$\scriptstyle m-1$\nodepart{lower}$n$};
        \node[biw] (new) at (1.5,0.7) {$i$\nodepart{lower}$j$}
            edge [<-] (v0);
        \node[invisible] (v1) at (-0.4,0.8) {}
            edge [<-] (v0);
        \node[] at (0,0.6) {$\scriptstyle\cdots$};
        \node[invisible] (v3) at (0.4,0.8) {}
            edge [<-] (v0);
        \node[invisible] (v4) at (-0.4,-0.8) {}
            edge [->] (v0);
        \node[] at (0,-0.6) {$\scriptstyle\cdots$};
        \node[invisible] (v6) at (0.4,-0.8) {}
            edge [->] (v0);
    \end{tikzpicture}
    \ -\ \sum_{\substack{i\geq 0,j\geq 1\\
    i+j\geq 2}}\ 
    \begin{tikzpicture}[shorten >=1pt,>=angle 90,baseline={(4ex,-0.5ex)}]
        \node[biw] (v0) at (0,0) {$\scriptstyle m$\nodepart{lower}$\scriptstyle{n-1}$};
        \node[biw] (new) at (1.5,-0.7) {$i$\nodepart{lower}$j$}
            edge [->] (v0);
        \node[invisible] (v1) at (-0.4,0.8) {}
            edge [<-] (v0);
        \node[] at (0,0.6) {$\scriptstyle\cdots$};
        \node[invisible] (v3) at (0.4,0.8) {}
            edge [<-] (v0);
        \node[invisible] (v4) at (-0.4,-0.8) {}
            edge [->] (v0);
        \node[] at (0,-0.6) {$\scriptstyle\cdots$};
        \node[invisible] (v6) at (0.4,-0.8) {}
            edge [->] (v0);
    \end{tikzpicture}
\end{equation*}
We tacitly assume in this formula that any term (if any) with negative in- or out-weight is set to zero. Similarly any graph with vertices of a bi-weight not satisfying condition 3 above is also set to zero. Some examples of vertices with invalid bi-weights are
\begin{equation*}
\begin{tikzpicture}[shorten >=1pt,>=angle 90,baseline=-0.1cm]
    \node[biw] (r) at (0,0) {$1$\nodepart{lower} $0$};
    \node[invisible] (v1) at (-0.35,0.8) {}
        edge [<-] (r);
    \node[invisible] (v2) at (0.35,0.8) {}
        edge [<-] (r);
\end{tikzpicture}
\qquad\quad
\begin{tikzpicture}[shorten >=1pt,>=angle 90,baseline=-0.1cm]
    \node[biw] (r) at (0,0) {$0$\nodepart{lower} $0$};
    \node[invisible] (v2) at (0,0.85) {}
        edge [<-] (r);
    \node[invisible] (v1) at (0,-0.85) {}
        edge [->] (r);
\end{tikzpicture}
\qquad\quad
\begin{tikzpicture}[shorten >=1pt,>=angle 90,baseline=-0.1cm]
    \node[biw] (r) at (0,0) {$1$\nodepart{lower} $1$};
\end{tikzpicture}
\qquad\quad
\begin{tikzpicture}[shorten >=1pt,>=angle 90,baseline=-0.1cm]
    \node[biw] (r) at (0,0) {$0$\nodepart{lower} $4$};
\end{tikzpicture}
\qquad\quad
\begin{tikzpicture}[shorten >=1pt,>=angle 90,baseline=-0.1cm]
    \node[biw] (r) at (0,0) {$1$\nodepart{lower} $0$};
    \node[invisible] (v1) at (0,-0.85) {}
        edge [->] (r);
\end{tikzpicture}
\qquad\quad
\begin{tikzpicture}[shorten >=1pt,>=angle 90,baseline=-0.1cm]
    \node[biw] (r) at (0,0) {$-1$\nodepart{lower} $2$};
    \node[invisible] (v1) at (-0.35,0.8) {}
        edge [<-] (r);
    \node[invisible] (v2) at (0.35,0.8) {}
        edge [<-] (r);
    \node[invisible] (v1) at (-0.35,-0.8) {}
        edge [->] (r);
    \node[invisible] (v2) at (0.35,-0.8) {}
        edge [->] (r);
\end{tikzpicture}
\end{equation*}
\begin{remark}
    Contrary to $\dGC_k$, the creation of new univalent vertices of graphs in $\mathsf{f}\WGC_k$ do not in general cancel under the action of the differential. If $d^{uni}_x$ is the part of the differential which increase the number of univalent vertices, then $d_x^{uni}$ acting on a vertex can look like
    \begin{equation*}
        d_x^{uni}\Big(\ 
        \begin{tikzpicture}[shorten >=1pt,>=angle 90,baseline={([yshift=-.5ex]current bounding box.center)}]
            \node[biw] (v0) at (0,0) {$3$\nodepart{lower}$0$};
            \node[invisible] (v1) at (-0.4,0.8) {}
                edge [<-] (v0);
            \node[invisible] (v3) at (0.4,0.8) {}
                edge [<-] (v0);
            \node[invisible] (v4) at (-0.4,-0.8) {}
                edge [->] (v0);
            \node[invisible] (v6) at (0.4,-0.8) {}
                edge [->] (v0);
        \end{tikzpicture}\ \Big) \
        =\ 
        \begin{tikzpicture}[shorten >=1pt,>=angle 90,baseline={(4ex,-0.5ex)}]
            \node[biw] (v0) at (0,0) {$ 0$\nodepart{lower}$0$};
            \node[biw] (new) at (1.5,0.7) {$3$\nodepart{lower}$0$}
                edge [<-] (v0);
            \node[invisible] (v1) at (-0.4,0.8) {}
                edge [<-] (v0);
            \node[invisible] (v3) at (0.4,0.8) {}
                edge [<-] (v0);
            \node[invisible] (v4) at (-0.4,-0.8) {}
                edge [->] (v0);
            \node[invisible] (v6) at (0.4,-0.8) {}
                edge [->] (v0);
        \end{tikzpicture}
        \ + \ 
        \begin{tikzpicture}[shorten >=1pt,>=angle 90,baseline={(4ex,-0.5ex)}]
            \node[biw] (v0) at (0,0) {$ 1$\nodepart{lower}$0$};
            \node[biw] (new) at (1.5,0.7) {$2$\nodepart{lower}$0$}
                edge [<-] (v0);
            \node[invisible] (v1) at (-0.4,0.8) {}
                edge [<-] (v0);
            \node[invisible] (v3) at (0.4,0.8) {}
                edge [<-] (v0);
            \node[invisible] (v4) at (-0.4,-0.8) {}
                edge [->] (v0);
            \node[invisible] (v6) at (0.4,-0.8) {}
                edge [->] (v0);
        \end{tikzpicture}
        \ -\ \sum_{\substack{i\geq 1,j\geq 0\\
        i+j\geq 2}}\
        \begin{tikzpicture}[shorten >=1pt,>=angle 90,baseline={(4ex,-0.5ex)}]
            \node[biw] (v0) at (0,0) {$ 2$\nodepart{lower}$0$};
            \node[biw] (new) at (1.5,0.7) {$i$\nodepart{lower}$j$}
                edge [<-] (v0);
            \node[invisible] (v1) at (-0.4,0.8) {}
                edge [<-] (v0);
            \node[invisible] (v3) at (0.4,0.8) {}
                edge [<-] (v0);
            \node[invisible] (v4) at (-0.4,-0.8) {}
                edge [->] (v0);
            \node[invisible] (v6) at (0.4,-0.8) {}
                edge [->] (v0);
        \end{tikzpicture}
    \end{equation*}
\end{remark}
\subsection{A special kind of bi-weight}
\begin{definition}
    Let $r\geq 0$ be an integer. The symbol $\infty_r$ when used as an in-weight or out-weight denotes the sum of graphs
\begin{equation*}
    \begin{tikzpicture}[baseline={([yshift=-.5ex]current bounding box.center)}]
        \node[biw] (v0) at (0,0) {$\infty_r$\nodepart{lower}$n$};
        \node[invisible] (v1) at (-0.4,0.8) {}
            edge [<-] (v0);
        \node[] at (0,0.6) {$\scriptstyle\cdots$};
        \node[invisible] (v3) at (0.4,0.8) {}
            edge [<-] (v0);
        \node[invisible] (v4) at (-0.4,-0.8) {}
            edge [->] (v0);
        \node[] at (0,-0.6) {$\scriptstyle\cdots$};
        \node[invisible] (v6) at (0.4,-0.8) {}
            edge [->] (v0);
    \end{tikzpicture}  
\ = \ \sum_{i\geq r}\ 
    \begin{tikzpicture}[baseline={([yshift=-.5ex]current bounding box.center)}]
        \node[biw] (v0) at (0,0) {$i$\nodepart{lower}$n$};
        \node[invisible] (v1) at (-0.4,0.8) {}
            edge [<-] (v0);
        \node[] at (0,0.6) {$\scriptstyle\cdots$};
        \node[invisible] (v3) at (0.4,0.8) {}
            edge [<-] (v0);
        \node[invisible] (v4) at (-0.4,-0.8) {}
            edge [->] (v0);
        \node[] at (0,-0.6) {$\scriptstyle\cdots$};
        \node[invisible] (v6) at (0.4,-0.8) {}
            edge [->] (v0);
    \end{tikzpicture} 
\ \qquad,\qquad\
    \begin{tikzpicture}[baseline={([yshift=-.5ex]current bounding box.center)}]
        \node[biw] (v0) at (0,0) {$m$\nodepart{lower}$\infty_r$};
        \node[invisible] (v1) at (-0.4,0.8) {}
            edge [<-] (v0);
        \node[] at (0,0.6) {$\scriptstyle\cdots$};
        \node[invisible] (v3) at (0.4,0.8) {}
            edge [<-] (v0);
        \node[invisible] (v4) at (-0.4,-0.8) {}
            edge [->] (v0);
        \node[] at (0,-0.6) {$\scriptstyle\cdots$};
        \node[invisible] (v6) at (0.4,-0.8) {}
            edge [->] (v0);
    \end{tikzpicture}  
\ = \ \sum_{i\geq r} \ 
    \begin{tikzpicture}[baseline={([yshift=-.5ex]current bounding box.center)}]
        \node[biw] (v0) at (0,0) {$m$\nodepart{lower}$i$};
        \node[invisible] (v1) at (-0.4,0.8) {}
            edge [<-] (v0);
        \node[] at (0,0.6) {$\scriptstyle\cdots$};
        \node[invisible] (v3) at (0.4,0.8) {}
            edge [<-] (v0);
        \node[invisible] (v4) at (-0.4,-0.8) {}
            edge [->] (v0);
        \node[] at (0,-0.6) {$\scriptstyle\cdots$};
        \node[invisible] (v6) at (0.4,-0.8) {}
            edge [->] (v0);
    \end{tikzpicture}
\end{equation*}
For graphs with two or more of these symbols decorating vertices, the sum is distributed as in the example below:
\begin{equation*}
\begin{tikzpicture}[shorten >=1pt,node distance=1.2cm,auto,>=angle 90]
    \node[biw] (g1v1) {$4$\nodepart{lower} $\infty_1$};
    \node[biw] (g1v2) [right of=g1v1] {$\infty_0$\nodepart{lower} $0$}
        edge [<-] (g1v1);
    \node[biw] (g1v3) [right of=g1v2] {$\infty_2$\nodepart{lower} $0$}
        edge [<-] (g1v2);
    \node[biw] (g1v4) [below of=g1v1] {$\infty_1$\nodepart{lower} $\infty_1$}
        edge [->] (g1v1);
    \node[biw] (g1v5) [right of=g1v4] {$0$\nodepart{lower} $\infty_0$}
        edge [<-] (g1v4)
        edge [->] (g1v2);
    \node[biw] (g1v6) [below of=g1v4] {$\infty_1$\nodepart{lower} $0$}
        edge [<-] (g1v4);
    \node[biw] (g1v7) [below of=g1v5] {$2$\nodepart{lower} $\infty_1$}
        edge [->] (g1v5);
    \node (likhet) at (3.3,-1.5) {$=\sum\limits_{\substack{
    i_l\geq 0\\j_m\geq 1\\k_n\geq 2}}$};
    \node[biw] (g2v1) at (4.5,0){$4$\nodepart{lower} $j_1$};
    \node[biw] (g2v2) [right of=g2v1] {$i_1$\nodepart{lower} $1$}
        edge [<-] (g2v1);
    \node[biw] (g2v3) [right of=g2v2] {$k_1$\nodepart{lower} $0$}
        edge [<-] (g2v2);
    \node[biw] (g2v4) [below of=g2v1] {$j_3$\nodepart{lower} $j_4$}
        edge [->] (g2v1);
    \node[biw] (g2v5) [right of=g2v4] {$0$\nodepart{lower} $i_2$}
        edge [<-] (g2v4)
        edge [->] (g2v2);
    \node[biw] (g2v6) [below of=g2v4] {$j_5$\nodepart{lower} $0$}
        edge [<-] (g2v4);
    \node[biw] (g2v7) [below of=g2v5] {$2$\nodepart{lower} $j_6$}
        edge [->] (g2v5);
\end{tikzpicture}
\end{equation*}
Any term of such a sum containing a vertex of invalid bi-weight is set to zero.
\end{definition}
Using this convention the differential is described as
\begin{equation*}
    d_x\Big(\ 
    \begin{tikzpicture}[shorten >=1pt,>=angle 90,baseline={([yshift=-.5ex]current bounding box.center)}]
        \node[biw] (v0) at (0,0) {$m$\nodepart{lower}$n$};
        \node[invisible] (v1) at (-0.4,0.8) {}
            edge [<-] (v0);
        \node[] at (0,0.6) {$\scriptstyle\cdots$};
        \node[invisible] (v3) at (0.4,0.8) {}
            edge [<-] (v0);
        \node[invisible] (v4) at (-0.4,-0.8) {}
            edge [->] (v0);
        \node[] at (0,-0.6) {$\scriptstyle\cdots$};
        \node[invisible] (v6) at (0.4,-0.8) {}
            edge [->] (v0);
    \end{tikzpicture}\ \Big) \ 
    = \ \sum_{\substack{\scriptstyle{m=m_1+m_2}\\ \scriptstyle{n=n_1+n_2}}}
    \begin{tikzpicture}[shorten >=1pt,>=angle 90,baseline={([yshift=-.5ex]current bounding box.center)}]
        \node[ellipse,
            draw = black,
            minimum width = 3cm, 
            minimum height = 1.8cm,
            dotted] (e) at (0,0) {};
        \node[biw] (vLeft) at (-0.8,0) {$m_1$\nodepart{lower}$n_1$};
        \node[biw] (vRight) at (0.8,0) {$m_2$\nodepart{lower}$n_2$}
            edge [<-] (vLeft);
        \node[invisible] (v1) at (-0.6,1.4) {}
            edge [<-] (e);
        \node[] at (0,1.2) {$\cdots$};
        \node[invisible] (v3) at (0.6,1.4) {}
            edge [<-] (e);
        \node[invisible] (v4) at (-0.6,-1.4) {}
            edge [->] (e);
        \node[] at (0,-1.2) {$\cdots$};
        \node[invisible] (v6) at (0.6,-1.4) {}
            edge [->] (e);
    \end{tikzpicture}
    \ -\
    \begin{tikzpicture}[shorten >=1pt,>=angle 90,baseline={(4ex,-0.5ex)}]
        \node[biw] (v0) at (0,0) {$\scriptstyle m-1$\nodepart{lower}$n$};
        \node[biw] (new) at (1.5,0.7) {$\infty_1$\nodepart{lower}$\infty_0$}
            edge [<-] (v0);
        \node[invisible] (v1) at (-0.4,0.8) {}
            edge [<-] (v0);
        \node[] at (0,0.6) {$\scriptstyle\cdots$};
        \node[invisible] (v3) at (0.4,0.8) {}
            edge [<-] (v0);
        \node[invisible] (v4) at (-0.4,-0.8) {}
            edge [->] (v0);
        \node[] at (0,-0.6) {$\scriptstyle\cdots$};
        \node[invisible] (v6) at (0.4,-0.8) {}
            edge [->] (v0);
    \end{tikzpicture}
    \ -\ 
    \begin{tikzpicture}[shorten >=1pt,>=angle 90,baseline={(4ex,-0.5ex)}]
        \node[biw] (v0) at (0,0) {$\scriptstyle m$\nodepart{lower}$\scriptstyle{n-1}$};
        \node[biw] (new) at (1.5,-0.7) {$\infty_0$\nodepart{lower}$\infty_1$}
            edge [->] (v0);
        \node[invisible] (v1) at (-0.4,0.8) {}
            edge [<-] (v0);
        \node[] at (0,0.6) {$\scriptstyle\cdots$};
        \node[invisible] (v3) at (0.4,0.8) {}
            edge [<-] (v0);
        \node[invisible] (v4) at (-0.4,-0.8) {}
            edge [->] (v0);
        \node[] at (0,-0.6) {$\scriptstyle\cdots$};
        \node[invisible] (v6) at (0.4,-0.8) {}
            edge [->] (v0);
    \end{tikzpicture}
\end{equation*}
\begin{remark}
    Outgoing univalent vertices cancel under the differential when the out-weight of a vertex is $0$ or $\infty_1$, and similarly for incoming univalent vertices.
\end{remark}
\subsection{The bi-weighted graph complex and the deformation complex}
\begin{definition}
    Let $\fWGC^+_k$ be the subcomplex of $\fWGC_k$ generated by graphs having at least one vertex with out-weight greater than zero, and at least one vertex with in-weight greater than zero. These two vertices are allowed to be the same vertex.
\end{definition}
The complex $\mathsf{f}\WGC_k$ is constructed to mimic the derivation complex $Der(\mathcal{H}olieb_{p,q}^{\circlearrowleft})$. We interpret the bi-weight of a vertex as the number of outgoing and incoming hairs that are attached to it. In the derivation complex the hairs are symmetrized/skew-symmetrized, which is not the case in the bi-weighted graph complex. The sign conventions for the differentials also differ between the complexes. In the deformation complex the sign convention is complicated, while the sign convention of the graph complex is a bit simpler.
We define the map
\begin{equation*}
    F:Der(\mathcal{H}olieb_{p,q}^{\circlearrowleft})\rightarrow\mathsf{f}\WGC_{p+q+1}    
\end{equation*} where a graph $\Gamma$ with unlabeled hairs (up to symmetry/skew-symmetry) is mapped to the bi-weighted graph $F(\Gamma)$ of the same shape and where the bi-weights of vertices correspond to the number of in- and out-hairs of the vertices in $\Gamma$.
\begin{proposition}\label{prop:HoliebWGC}
    The map $F:Der(\mathcal{H}olieb_{p,q}^{\circlearrowleft})\rightarrow\mathsf{f}\WGC_{p+q+1}$ is a chain map of degree 0 such that
    \begin{enumerate}
        \item the map $F$ is an isomorphism of complexes,
        \item the map $F$ restricts to an isomorphism $F^+:Der^+(\mathcal{H}olieb_{p,q}^{\circlearrowleft})\rightarrow\fwGC_{p+q+1}^+$.
    \end{enumerate}
\end{proposition}
\begin{proof}
    Proving that $F$ is an isomorphism is just the untwisting of definitions of both complexes. We do however need to show that $F$ is of degree 0. Recall that the derivation complex decomposes as
    \begin{equation*}
        Der(\mathcal{H}olieb_{p,q}^{\circlearrowleft})=\prod_{m,n\geq0}\big(\mathcal{H}olieb_{p,q}^{\circlearrowleft}\otimes sgn^{\otimes|p|}_m\otimes sgn^{\otimes|q|}_n\big)^{\mathbb{S}_m\times\mathbb{S}_n}[1+p(m-1)+q(n-1)]
    \end{equation*} where $\mathcal{H}olieb_{p,q}^{\circlearrowleft}$ is the vector space of graphs with $m$ out-hairs and $n$ in-hairs. Let $\Gamma$ be a graph in the derivation complex. For each vertex $x$ of $\Gamma$, let $|x|_{out} $ denote the number of outgoing half-edges and $|x|_{in}$ the number of incoming half-edges. Then
    \begin{align*}
        |\Gamma|&=\sum_{x\in V(\Gamma)}\big(1-(|x|_{out}-1)-(|x|_{in}-1) \big)-1-p(1-m)-q(1-n)\\
        &=\sum_{x\in V(\Gamma)}\big(1+p+q\big) -(1+p+q)- \sum_{x\in V(\Gamma)}\big(p|x|_{out}+m|x|_{in}\big)+pm+qn\\
        &=|V(\Gamma)-1|(p+q+1) - |E(\Gamma)|(p+q) = |F(\Gamma)|
    \end{align*}
    One remarks that $F^+$ is a bijection by noting that bi-weights are symbolizing hairs and $Der^+(\mathcal{H}olieb_{c,d}^{\circlearrowleft})$ can be seen as generated by graphs with at least one out- and in-hair attached.
\end{proof}
\subsection{Decomposition over decorations and loop numbers}\label{sec:BiWCohom}
Let $\fwGC_k^0$ be the subcomplex of $\fwGC_k$ generated by graphs whose vertices are only decorated by $\frac{0}{0}$, and $\fwGC_k^*$ its complement. Then the complex $\fwGC_k$ split as
\begin{equation*}
    \fwGC_k=\fwGC_k^0\oplus\fwGC_k^*.
\end{equation*}
\begin{proposition}\label{prop:0complex}
    The complex $\fwGC_k^0$ is isomorphic to the complex $\dGC^\circlearrowleft_k=\dGC_k/\dGC^{s+t}_k$.
\end{proposition}
\begin{proof}
    By direct inspection of the graphs where all vertices can be decorated by $\frac{0}{0}$, one easily sees that they need to be at least trivalent and have at least one incoming and one outgoing vertex. This corresponds to the graphs in $\dGC^\circlearrowleft_k$. One also notes that the differentials act in the same manner.
\end{proof}
Recall that the loop number of a graph is preserved under the differential. Consider the decompositions
\begin{align*}
    \fwGC_k&=\mathsf{b}_0\WGC_k\oplus\wGC_k\\
    \fwGC^*_k&=\mathsf{b}_0\WGC_k^*\oplus\wGC^*_k\\
    \fwGC^+_k&=\mathsf{b}_0\WGC_k^+\oplus\WGC_k^+
\end{align*}
where $\mathsf{b}_0\WGC_k, \mathsf{b}_0\WGC_k^*$ and $\mathsf{b}_0\WGC_k^+$ are the subcomplexes of graphs with loop number zero and $\WGC_k,\ \WGC_k^*$ and $\WGC_k^+$ the subcomplex of graphs with loop number one and higher.
We note that graphs with loop number zero cannot be completely bald, and so $\mathsf{b}_0\WGC_k=\mathsf{b}_0\WGC_k^*$.
Further note that there are no closed loops in a graph with loop number zero, and so they contain at least one source and one target vertex. These vertices must have positive in-weight and out-weight respectively, and so  $\mathsf{b}_0\wGC_k=\mathsf{b}_0\wGC_k^+$.
\begin{proposition}
    The cohomology of the complex of graphs with loop number zero $\mathsf{b}_0\wGC_k$ is generated by the series
    \begin{equation*}
        \sum_{\substack{i,j\geq1\\i+j\geq 3}}(i+j-2)\ 
        \tikz[baseline=-2.07cm]
        \node[biw] at (0,-2) {$i$\nodepart{lower}$j$};\ .
    \end{equation*}
\end{proposition}
\begin{proof}
    These graphs correspond to the part of the deformation complex of graphs with loop number zero, whose cohomology was computed to be this graphs counterpart in \ref{thm:onevertex}.
\end{proof}
\section{Special in-vertices and special out-vertices}
In this section we will define three subcomplexes $\qGC_k\subset\wGC_k$, $\qGC^*_k\subset\wGC^*_k$ and $\qGC^+_k\subset\wGC_k^+$ consisting of graphs whose vertices are decorated by four types of decorations $\frac{\infty_1}{\infty_1}$, $\frac{\infty_1}{0}$, $\frac{0}{\infty_1}$ and $\frac{0}{0}$. The main goal is to
show that these inclusions are quasi-isomorphisms. We do this by considering two consecutive filtrations on
the complexes over special-in and special-out vertices respectively. We then show that the associated spectral sequences agree on some page.

\subsection{Convergence of filtrations and their associated spectral sequences}
In this paper, we will consider many arguments where we consider a filtration of a chain complex and then study the associated spectral sequences. All the spectral sequences we construct in this way will converge. To see this, we do a similar trick of shifting the degrees of the complexes as seen in \cite{W1}. Let the new degree of a graph be $k(v-1)-(k-1)e+(k-\frac{1}{2})(e-v)=\frac{1}{2}(v+e)-k$. The cohomology of both complexes agree up to degree shifts. Further any filtration in the old grading corresponds to a filtration with the new grading. We see that the number of underlying directed graphs of the bi-weighted graphs contained in each degree is finite. Any filtration we do will be over the number of vertices of certain types, and so the filtration will be bounded and hence converges to the desired cohomology. The cohomology of the original complex is then acquired from the shifted version.
\subsection{Filtration over special in-vertices}
\begin{definition}
    Let $\Gamma$ be a bi-weighted graph. A
    vertex $x$ of $\Gamma$ is a \textit{special in-vertex} if
    \begin{enumerate}
        \item[i)] either $x$ is a univalent vertex with one outgoing edge and out-weight zero, i.e on the form $\begin{tikzpicture}[shorten >=1pt,node distance=1.2cm,auto,>=angle 90,baseline=-0.1cm]
            \node[biw] (a) at (0,0) {$0$\nodepart{lower}$n$};
            \node[] (d) at (0,1) {}
                edge [<-] (a);
        \end{tikzpicture}$
        \item[ii)] or $x$ becomes a univalent vertex of type i) after recursive removal of all special-in vertices of type i) from $\Gamma$ (see figure \ref{fig:SpecIn}).
    \end{enumerate}
    \begin{figure}[h]
    \centering
\begin{tikzpicture}[shorten >=1pt,>=angle 90,baseline={([yshift={-\ht\strutbox}]current bounding box.north)},outer sep=0pt,inner sep=0pt]
    \node[biw] (r1) at (0,0) {$2$\nodepart{lower}$3$};
    \node[biw] (g1v1) at (-0.75,-1) {$0$\nodepart{lower}$3$}
        edge [->] (r1);
    \node[biw] (g1v2) at (-1.5,-2) {$0$\nodepart{lower}$1$}
        edge [->] (g1v1);
    \node[biw] (g1v3) at (-2.25,-3) {$0$\nodepart{lower}$4$}
        edge [->] (g1v2);
    \node[biw] (g1v4) at (0.75,-1) {$0$\nodepart{lower}$0$}
        edge [->] (r1);
    \node[biw] (g1v5) at (1.5,-2) {$0$\nodepart{lower}$2$}
        edge [->] (g1v4);
    \node[biw] (g1v6) at (2.25,-3) {$0$\nodepart{lower}$2$}
        edge [->] (g1v5);
    \node[biw] (g1v7) at (0,-2) {$0$\nodepart{lower}$2$}
        edge [->] (g1v4);
    \node[biw] (g1v8) at (-0.75,-3) {$0$\nodepart{lower}$2$}
        edge [->] (g1v7);
    \node[biw] (g1v9) at (0.75,-3) {$0$\nodepart{lower}$6$}
        edge [->] (g1v7);
\end{tikzpicture}
\ \ 
\begin{tikzpicture}[shorten >=1pt,>=angle 90,baseline={([yshift={-\ht\strutbox}]current bounding box.north)},outer sep=0pt,inner sep=0pt]
    \node[biw] (r1) at (0,0) {$2$\nodepart{lower}$3$};
    \node[biw] (g1v1) at (-0.75,-1) {$0$\nodepart{lower}$3$}
        edge [->] (r1);
    \node[biw] (g1v2) at (-1.5,-2) {$0$\nodepart{lower}$1$}
        edge [->] (g1v1);
    \node[biw] (g1v4) at (0.75,-1) {$0$\nodepart{lower}$0$}
        edge [->] (r1);
    \node[biw] (g1v5) at (1.5,-2) {$0$\nodepart{lower}$2$}
        edge [->] (g1v4);
    \node[biw] (g1v7) at (0,-2) {$0$\nodepart{lower}$2$}
        edge [->] (g1v4);
\end{tikzpicture}
\ \ 
\begin{tikzpicture}[shorten >=1pt,>=angle 90,baseline={([yshift={-\ht\strutbox}]current bounding box.north)},outer sep=0pt,inner sep=0pt]
    \node[biw] (r1) at (0,0) {$2$\nodepart{lower}$3$};
    \node[biw] (g1v1) at (-0.75,-1) {$0$\nodepart{lower}$3$}
        edge [->] (r1);
    \node[biw] (g1v4) at (0.75,-1) {$0$\nodepart{lower}$0$}
        edge [->] (r1);
\end{tikzpicture}
\ \ 
\begin{tikzpicture}[baseline={([yshift={-\ht\strutbox}]current bounding box.north)},outer sep=0pt,inner sep=0pt]
    \node[biw] (r1) at (0,0) {$2$\nodepart{lower}$3$};
\end{tikzpicture}
\caption{Example of recursive removal of special-in vertices of the rightmost graph. All vertices except the top one are special in-vertices.}
\label{fig:SpecIn}
\end{figure}
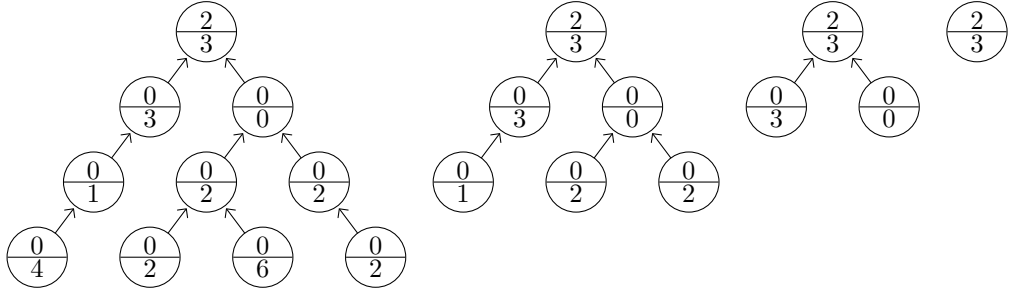
    Any vertex that is not a special in-vertex is called an \textit{in-core vertex}. The special in-vertices of a graph form trees with a flow towards some \textit{in-core vertex} and where the trees have no out-hairs.
    Given an arbitrary graph $\Gamma$ in $\WGC_k$, we define the associated \textit{in-core graph} $\gamma$ as the one spanned by in-core vertices with their in-weight forgotten. Then every vertex $x$ in $\gamma$ has three integer parameters associated to it, $|x|_{out}$, $|x|_{in}$ and $w_x^{out}$. Note that every graph contain at least one in-core vertex, and so every graph has an associated in-core graph.
\end{definition}
Consider a filtration of $\wGC_k$ over the number of in-core vertices. Let $\{S^{in}_{r}\wGC_k\}_{r\geq 0}$ be the associated spectral sequence.
\begin{proposition}\label{prop:specialIn}
The page one complex $S^{in}_1\wGC_k$ is generated by directed graphs $\Gamma$ whose vertices $V(\Gamma)$ are independently decorated by one out-weight $m$ and with two possible symbols $0$ or $\infty_1$ as in-decorations, subject to the following conditions:
\begin{enumerate}
    \item If $x\in V(\Gamma)$ is a source, then
        \begin{equation*}
        x=\begin{tikzpicture}[shorten >=1pt,node distance=1.2cm,auto,>=angle 90,baseline=-0.1cm]
            \node[biw] (a) at (0,0) {$m$\nodepart{lower}$\infty_1$};
            \node[] (up1) at (-0.4,0.8) {}
                edge [<-] (a);
            \node[] (up2) at (0,0.6) {$\scriptstyle\cdots$};
            \node[] (up2) at (0,0.9) {$\scriptstyle \geq 1$};
            \node[] (up3) at (0.4,0.8) {}
                edge [<-] (a);
        \end{tikzpicture}
        \text{with } m+|x|_{out}\geq 2\text{ and }|x|_{out} \geq 1
    \end{equation*}
    \item If $x\in V(\Gamma)$ is a target with precisely one in-edge, then
        \begin{equation*}
        x=\begin{tikzpicture}[shorten >=1pt,node distance=1.2cm,auto,>=angle 90,baseline=-0.1cm]
            \node[biw] (a) at (0,0) {$m$\nodepart{lower}$\infty_1$};
            \node[] (d) at (0,-1) {}
                edge [->] (a);
        \end{tikzpicture}
        \text{ with }m\geq 1,\text{ or }
        x=\begin{tikzpicture}[shorten >=1pt,node distance=1.2cm,auto,>=angle 90,baseline=-0.1cm]
            \node[biw] (a) at (0,0) {$m$\nodepart{lower}$0$};
            \node[] (d) at (0,-1) {}
                edge [->] (a);
        \end{tikzpicture}
        \text{ with }m\geq 2
    \end{equation*}
    \item If $x\in V(\Gamma)$ is a target with at least two in-edges, then
    \begin{equation*}
        x=\begin{tikzpicture}[shorten >=1pt,node distance=1.2cm,auto,>=angle 90,baseline=-0.1cm]
            \node[biw] (a) at (0,0) {$m$\nodepart{lower}$\infty_1$};
            \node[] (up1) at (-0.4,-0.8) {}
                edge [->] (a);
            \node[] (up2) at (0,-0.6) {$\scriptstyle\cdots$};
            \node[] (up2) at (0,-0.9) {$\scriptstyle \geq 2$};
            \node[] (up3) at (0.4,-0.8) {}
                edge [->] (a);
        \end{tikzpicture}
        \text{with } m\geq 1,\text{ or }
        x=\begin{tikzpicture}[shorten >=1pt,node distance=1.2cm,auto,>=angle 90,baseline=-0.1cm]
            \node[biw] (a) at (0,0) {$m$\nodepart{lower}$0$};
            \node[] (up1) at (-0.4,-0.8) {}
                edge [->] (a);
            \node[] (up2) at (0,-0.6) {$\scriptstyle\cdots$};
            \node[] (up2) at (0,-0.9) {$\scriptstyle \geq 2$};
            \node[] (up3) at (0.4,-0.8) {}
                edge [->] (a);
        \end{tikzpicture}
        \text{with } m\geq 1
    \end{equation*}
    \item If $x\in V(\Gamma)$ is passing (one in-edge and one out-edge), then
    \begin{equation*}
        x=\begin{tikzpicture}[shorten >=1pt,node distance=1.2cm,auto,>=angle 90,baseline=-0.1cm]
            \node[biw] (a) at (0,0) {$m$\nodepart{lower}$\infty_1$};
            \node[] (up) at (0,1) {}
                edge [<-] (a);
            \node[] (down) at (0,-1) {}
                edge [->] (a);
        \end{tikzpicture}
        \text{ with }m\geq 0,\text{ or }
        x=\begin{tikzpicture}[shorten >=1pt,node distance=1.2cm,auto,>=angle 90,baseline=-0.1cm]
            \node[biw] (a) at (0,0) {$m$\nodepart{lower}$0$};
            \node[] (up) at (0,1) {}
                edge [<-] (a);
            \node[] (down) at (0,-1) {}
                edge [->] (a);
        \end{tikzpicture}
        \text{ with }m\geq 1
    \end{equation*}
    \item If $x\in V(\Gamma)$ is of none of the types above (i.e $x$ is at least trivalent and has at least one in-edge and at least one out-edge), then
    \begin{equation*}
        x=\begin{tikzpicture}[shorten >=1pt,node distance=1.2cm,auto,>=angle 90,baseline=-0.1cm]
            \node[biw] (a) at (0,0) {$m$\nodepart{lower}$\infty_1$};
            \node[] (down1) at (-0.4,-0.8) {}
                edge [->] (a);
            \node[] (down2) at (0,-0.6) {$\scriptstyle\cdots$};
            \node[] (down3) at (0.4,-0.8) {}
                edge [->] (a);
            \node[] (up1) at (-0.4,0.8) {}
                edge [<-] (a);
            \node[] (up2) at (0,0.6) {$\scriptstyle\cdots$};
            \node[] (up3) at (0.4,0.8) {}
                edge [<-] (a);
        \end{tikzpicture}
        \text{with } m\geq 0,\text{ or }
        x=\begin{tikzpicture}[shorten >=1pt,node distance=1.2cm,auto,>=angle 90,baseline=-0.1cm]
            \node[biw] (a) at (0,0) {$m$\nodepart{lower}$0$};
            \node[] (down1) at (-0.4,-0.8) {}
                edge [->] (a);
            \node[] (down2) at (0,-0.6) {$\scriptstyle\cdots$};
            \node[] (down3) at (0.4,-0.8) {}
                edge [->] (a);
            \node[] (up1) at (-0.4,0.8) {}
                edge [<-] (a);
            \node[] (up2) at (0,0.6) {$\scriptstyle\cdots$};
            \node[] (up3) at (0.4,0.8) {}
                edge [<-] (a);
        \end{tikzpicture}
        \text{with } m\geq 0
    \end{equation*}
\end{enumerate}
The differential acts on a graph $\Gamma\in S^{in}_1\wGC_k$ with vertices of the types (1)-(5) above as $d(\Gamma)=\sum_{x\in V(\Gamma)}d_x(\Gamma)$. The map $d_x$ act on vertices with in-weight $\infty_1$ and $0$ respectively as
\begin{align*}
    d_x\Big(\ 
    \begin{tikzpicture}[shorten >=1pt,>=angle 90,baseline={([yshift=-.5ex]current bounding box.center)}]
        \node[biw] (v0) at (0,0) {$m$\nodepart{lower}$\infty_1$};
        \node[invisible] (v1) at (-0.4,0.8) {}
            edge [<-] (v0);
        \node[] at (0,0.6) {$\scriptstyle\cdots$};
        \node[invisible] (v3) at (0.4,0.8) {}
            edge [<-] (v0);
        \node[invisible] (v4) at (-0.4,-0.8) {}
            edge [->] (v0);
        \node[] at (0,-0.6) {$\scriptstyle\cdots$};
        \node[invisible] (v6) at (0.4,-0.8) {}
            edge [->] (v0);
    \end{tikzpicture}\ \Big) \ 
    &= \ \sum_{\scriptstyle{m=m_1+m_2}}
    \Bigg(
    \begin{tikzpicture}[shorten >=1pt,>=angle 90,baseline={([yshift=-.5ex]current bounding box.center)}]
        \node[ellipse,
            draw = black,
            minimum width = 3cm, 
            minimum height = 1.8cm,
            dotted] (e) at (0,0) {};
        \node[biw] (vLeft) at (-0.8,0) {$m_1$\nodepart{lower}$\infty_1$};
        \node[biw] (vRight) at (0.8,0) {$m_2$\nodepart{lower}$\infty_1$}
            edge [<-] (vLeft);
        \node[invisible] (v1) at (-0.6,1.4) {}
            edge [<-] (e);
        \node[] at (0,1.2) {$\cdots$};
        \node[invisible] (v3) at (0.6,1.4) {}
            edge [<-] (e);
        \node[invisible] (v4) at (-0.6,-1.4) {}
            edge [->] (e);
        \node[] at (0,-1.2) {$\cdots$};
        \node[invisible] (v6) at (0.6,-1.4) {}
            edge [->] (e);
    \end{tikzpicture}
    \ + \
    \begin{tikzpicture}[shorten >=1pt,>=angle 90,baseline={([yshift=-.5ex]current bounding box.center)}]
        \node[ellipse,
            draw = black,
            minimum width = 3cm, 
            minimum height = 1.8cm,
            dotted] (e) at (0,0) {};
        \node[biw] (vLeft) at (-0.8,0) {$m_1$\nodepart{lower}$0$};
        \node[biw] (vRight) at (0.8,0) {$m_2$\nodepart{lower}$\infty_1$}
            edge [<-] (vLeft);
        \node[invisible] (v1) at (-0.6,1.4) {}
            edge [<-] (e);
        \node[] at (0,1.2) {$\cdots$};
        \node[invisible] (v3) at (0.6,1.4) {}
            edge [<-] (e);
        \node[invisible] (v4) at (-0.6,-1.4) {}
            edge [->] (e);
        \node[] at (0,-1.2) {$\cdots$};
        \node[invisible] (v6) at (0.6,-1.4) {}
            edge [->] (e);
    \end{tikzpicture}
    \ + \ 
    \begin{tikzpicture}[shorten >=1pt,>=angle 90,baseline={([yshift=-.5ex]current bounding box.center)}]
        \node[ellipse,
            draw = black,
            minimum width = 3cm, 
            minimum height = 1.8cm,
            dotted] (e) at (0,0) {};
        \node[biw] (vLeft) at (-0.8,0) {$m_1$\nodepart{lower}$\infty_1$};
        \node[biw] (vRight) at (0.8,0) {$m_2$\nodepart{lower}$0$}
            edge [<-] (vLeft);
        \node[invisible] (v1) at (-0.6,1.4) {}
            edge [<-] (e);
        \node[] at (0,1.2) {$\cdots$};
        \node[invisible] (v3) at (0.6,1.4) {}
            edge [<-] (e);
        \node[invisible] (v4) at (-0.6,-1.4) {}
            edge [->] (e);
        \node[] at (0,-1.2) {$\cdots$};
        \node[invisible] (v6) at (0.6,-1.4) {}
            edge [->] (e);
    \end{tikzpicture}
    \Bigg)\\
    &\qquad \ -\
    \begin{tikzpicture}[shorten >=1pt,>=angle 90,baseline={(4ex,-0.5ex)}]
        \node[biw] (v0) at (0,0) {$\scriptstyle m-1$\nodepart{lower}$\infty_1$};
        \node[biw] (new) at (1.5,0.7) {$\infty_1$\nodepart{lower}$\infty_0$}
            edge [<-] (v0);
        \node[invisible] (v1) at (-0.4,0.8) {}
            edge [<-] (v0);
        \node[] at (0,0.6) {$\scriptstyle\cdots$};
        \node[invisible] (v3) at (0.4,0.8) {}
            edge [<-] (v0);
        \node[invisible] (v4) at (-0.4,-0.8) {}
            edge [->] (v0);
        \node[] at (0,-0.6) {$\scriptstyle\cdots$};
        \node[invisible] (v6) at (0.4,-0.8) {}
            edge [->] (v0);
    \end{tikzpicture}
    \ -\ 
    \begin{tikzpicture}[shorten >=1pt,>=angle 90,baseline={(4ex,-0.5ex)}]
        \node[biw] (v0) at (0,0) {$m$\nodepart{lower}$\infty_1$};
        \node[biw] (new) at (1.5,-0.7) {$\infty_1$\nodepart{lower}$\infty_1$}
            edge [->] (v0);
        \node[invisible] (v1) at (-0.4,0.8) {}
            edge [<-] (v0);
        \node[] at (0,0.6) {$\scriptstyle\cdots$};
        \node[invisible] (v3) at (0.4,0.8) {}
            edge [<-] (v0);
        \node[invisible] (v4) at (-0.4,-0.8) {}
            edge [->] (v0);
        \node[] at (0,-0.6) {$\scriptstyle\cdots$};
        \node[invisible] (v6) at (0.4,-0.8) {}
            edge [->] (v0);
    \end{tikzpicture}
    \ -\ 
    \begin{tikzpicture}[shorten >=1pt,>=angle 90,baseline={(4ex,-0.5ex)}]
        \node[biw] (v0) at (0,0) {$m$\nodepart{lower}$0$};
        \node[biw] (new) at (1.5,-0.7) {$\infty_1$\nodepart{lower}$\infty_1$}
            edge [->] (v0);
        \node[invisible] (v1) at (-0.4,0.8) {}
            edge [<-] (v0);
        \node[] at (0,0.6) {$\scriptstyle\cdots$};
        \node[invisible] (v3) at (0.4,0.8) {}
            edge [<-] (v0);
        \node[invisible] (v4) at (-0.4,-0.8) {}
            edge [->] (v0);
        \node[] at (0,-0.6) {$\scriptstyle\cdots$};
        \node[invisible] (v6) at (0.4,-0.8) {}
            edge [->] (v0);
    \end{tikzpicture}\\
    d_x\Big(\ 
    \begin{tikzpicture}[shorten >=1pt,>=angle 90,baseline={([yshift=-.5ex]current bounding box.center)}]
        \node[biw] (v0) at (0,0) {$m$\nodepart{lower}$0$};
        \node[invisible] (v1) at (-0.4,0.8) {}
            edge [<-] (v0);
        \node[] at (0,0.6) {$\scriptstyle\cdots$};
        \node[invisible] (v3) at (0.4,0.8) {}
            edge [<-] (v0);
        \node[invisible] (v4) at (-0.4,-0.8) {}
            edge [->] (v0);
        \node[] at (0,-0.6) {$\scriptstyle\cdots$};
        \node[invisible] (v6) at (0.4,-0.8) {}
            edge [->] (v0);
    \end{tikzpicture}\ \Big) \ 
    &= \ \sum_{\scriptstyle{m=m_1+m_2}}
    \begin{tikzpicture}[shorten >=1pt,>=angle 90,baseline={([yshift=-.5ex]current bounding box.center)}]
        \node[ellipse,
            draw = black,
            minimum width = 3cm, 
            minimum height = 1.8cm,
            dotted] (e) at (0,0) {};
        \node[biw] (vLeft) at (-0.8,0) {$m_1$\nodepart{lower}$0$};
        \node[biw] (vRight) at (0.8,0) {$m_2$\nodepart{lower}$0$}
            edge [<-] (vLeft);
        \node[invisible] (v1) at (-0.6,1.4) {}
            edge [<-] (e);
        \node[] at (0,1.2) {$\cdots$};
        \node[invisible] (v3) at (0.6,1.4) {}
            edge [<-] (e);
        \node[invisible] (v4) at (-0.6,-1.4) {}
            edge [->] (e);
        \node[] at (0,-1.2) {$\cdots$};
        \node[invisible] (v6) at (0.6,-1.4) {}
            edge [->] (e);
    \end{tikzpicture}
    \ -\
    \begin{tikzpicture}[shorten >=1pt,>=angle 90,baseline={(4ex,-0.5ex)}]
        \node[biw] (v0) at (0,0) {$\scriptstyle m-1$\nodepart{lower}$0$};
        \node[biw] (new) at (1.5,0.7) {$\infty_1$\nodepart{lower}$\infty_0$}
            edge [<-] (v0);
        \node[invisible] (v1) at (-0.4,0.8) {}
            edge [<-] (v0);
        \node[] at (0,0.6) {$\scriptstyle\cdots$};
        \node[invisible] (v3) at (0.4,0.8) {}
            edge [<-] (v0);
        \node[invisible] (v4) at (-0.4,-0.8) {}
            edge [->] (v0);
        \node[] at (0,-0.6) {$\scriptstyle\cdots$};
        \node[invisible] (v6) at (0.4,-0.8) {}
            edge [->] (v0);
    \end{tikzpicture}
\end{align*}
Any term on the right hand side containing at least one vertex not of the type $(1)-(5)$ is set to zero.
\end{proposition}
The rule of the differential follows from the bi-weight types of the vertices mentioned above. So we only need to show that the cohomology of $S^{in}_0\wGC_k$ is generated by graphs with such bi-weights. The differential acts on the initial page $S^{in}_{0}\wGC_d$ by only creating special-in vertices and leaves the connected in-core graph unchanged. Hence the complex decomposes into a direct sum parameterized by the set of all possible in-core graphs
\begin{equation*}
    S^{in}_0\wGC_k\cong \bigoplus_{\gamma}\mathsf{inCore}(\gamma)
\end{equation*}
where $\mathsf{inCore}(\gamma)$ is the subcomplex of $S^{in}_0\WGC_k$ of all graphs whose associated in-core graph is $\gamma$. The complex $\mathsf{inCore}(\gamma)$ decomposes further into a tensor product of complexes
\begin{equation*}
    \mathsf{inCore}(\gamma)\cong\Big(\bigotimes_{x\in V(\gamma)}\mathcal{T}^{in}_x\Big)^{\mathrm{Aut}(\gamma)}
\end{equation*}
with one complex $\mathcal{T}^{in}_x$ for each $x$ in $V(\gamma)$ and where $\mathrm{Aut}(\gamma)$ is the group of automorphisms of $\gamma$ acting by permuting the complexes of the tensor complex to preserve signs.
Each complex $\mathcal{T}^{in}_x$ consists of trees of special in-vertices attached to an in-core vertex $x$ and the differential acts by only creating special in-vertices on $x$ and in the trees. The complexes $\mathcal{T}^{in}_x$ depends on $x$ only via the number of outgoing and incoming edges attached to $x$ in the in-core graph as well as on the out-weight $w_x^{out}$. The in-weight of $x$ is not fixed. The complexes $\mathcal{T}^{in}_x$ which have the same values of the parameters $|x|_{out}$, $|x|_{in}$ and $w_x^{out}$ are isomorphic to each other, so we will often write $\mathcal{T}^{in}_v\cong \mathcal{T}^{in}_{|x|_{out},|x|_{in},w_x^{out}}$. So we have to study a family of complexes $\mathcal{T}^{in}_{a,b,c}$ parameterized by integers $a,b,c\geq 0$ such that $a+c\geq 1$ and $(a,b,c)\neq(1,0,0)$. The first condition guarantees that the in-core vertex has at least one out-edge or out-weight, and the second condition correspond to the invalid configuration where the in-core vertex would be a special-in vertex. The in-weight of the core vertex in $\mathcal{T}^{in}_{a,b,c}$ is any number $w_x^{in}$ satisfying the condition $w_x^{in}+b+\#\text{(in-edges from special in-vertices)}\geq 1$ and $w_x^{in}+a+b+c+\#\text{(in-edges from special in-vertices)}\geq 3$. Proposition \ref{prop:specialIn} now follows from the following lemma.
\begin{lemma}\label{lemma:intreecohom}
    The cohomology of $\mathcal{T}^{in}_{a,b,c}$ is generated by one or two classes containing the in-core vertex decorated by some bi-weights depending on the parameters $a,b,c$. More precisely
    \begin{align*}
        \text{For }a\geq 1\quad
        H(\mathcal{T}^{in}_{a,0,c})&=
        \Big\langle
            \tikz[baseline=-2.07cm]
            \node[biw] at (0,-2) {$c$\nodepart{lower}$\infty_1$};
        \Big\rangle & (c\geq 0 \text{ and } a+c\geq 2)\\
        H(\mathcal{T}^{in}_{0,1,c})&=
        \begin{cases}
            \Big\langle
            \tikz[baseline=-2.07cm]
            \node[biw] at (0,-2) {$c$\nodepart{lower}$\infty_1$};
            \Big\rangle &\text{ if }c=1\\
            \Big\langle
            \tikz[baseline=-2.07cm]
            \node[biw] at (0,-2) {$c$\nodepart{lower}$\infty_1$};
            \, ,\
            \tikz[baseline=-2.07cm]
            \node[biw] at (0,-2) {$c$\nodepart{lower}$0$};
        \Big\rangle &\text{ if }c\geq 2
        \end{cases}\\
        H(\mathcal{T}^{in}_{1,1,c})&=
            \begin{cases}
                \Big\langle
                \tikz[baseline=-2.07cm]
                \node[biw] at (0,-2) {$c$\nodepart{lower}$\infty_1$};
                \Big\rangle &\text{ if }c=0\\
                \Big\langle
                \tikz[baseline=-2.07cm]
                \node[biw] at (0,-2) {$c$\nodepart{lower}$\infty_1$};
                \, ,\
                \tikz[baseline=-2.07cm]
                \node[biw] at (0,-2) {$c$\nodepart{lower}$0$};
            \Big\rangle &\text{ if }c\geq 1
            \end{cases}\\
        \text{For }b\geq 2\quad H(\mathcal{T}^{in}_{0,b,c})&=
            \Big\langle
            \tikz[baseline=-2.07cm]
            \node[biw] at (0,-2) {$c$\nodepart{lower}$\infty_1$};
            \, ,\
            \tikz[baseline=-2.07cm]
            \node[biw] at (0,-2) {$c$\nodepart{lower}$0$};
            \Big\rangle & (c\geq 1)
            \\
        \text{For }a,b\geq 1\text{ and }a+b\geq 3\quad H(\mathcal{T}^{in}_{a,b,c})&=
            \Big\langle
            \tikz[baseline=-2.07cm]
            \node[biw] at (0,-2) {$c$\nodepart{lower}$\infty_1$};
            \, ,\
            \tikz[baseline=-2.07cm]
            \node[biw] at (0,-2) {$c$\nodepart{lower}$0$};
        \Big\rangle &(c\geq 0)\\
    \end{align*}
\end{lemma}
    The differential in $\mathcal{T}^{in}_{a,b,c}$ does not decrease the number of univalent special in-vertices in a graph. Consider a filtration on $\mathcal{T}^{in}_{a,b,c}$ by the number of univalent vertices (considering the graph consisting of only the root vertex $v$ as having one univalent vertex).
    Let $gr(\mathcal{T}^{in}_{a,b,c})$ be the associated graded complex. It decomposes as $gr(\mathcal{T}^{in}_{a,b,c})=\bigoplus_{N\geq 1}u_N\mathcal{T}^{in}_{a,b,c}$ where $u_N\mathcal{T}^{in}_{a,b,c}$ is spanned by trees with precisely $N$ univalent special in-vertices. Lemma \ref{lemma:intreecohom} follows from the following results:
\begin{lemma}\label{lemma:u1Tin}
    The cohomology $H(u_1\mathcal{T}^{in}_{a,b,c})$ is generated by the same elements given in lemma \ref{lemma:intreecohom}.        
\end{lemma}
\begin{lemma}\label{lemma:uNTin}
        The complex $u_N\mathcal{T}^{in}_{a,b,c}$ is acyclic for $N\geq 2$.
\end{lemma}
    We start by computing the cohomology of $u_1\mathcal{T}^{in}_{a,b,c}$. The graphs in this complex are on the form
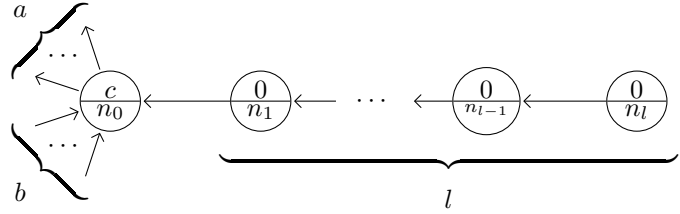
\begin{equation*}
    \begin{tikzpicture}[shorten >=1pt,node distance=1.2cm,auto,>=angle 90,baseline=-0.1cm]
    \node[biw] (core) at (0,0) {$c$\nodepart{lower}$n_0$};
    \node[biw] (p1) at (2,0) {$0$\nodepart{lower}$n_1$}
        edge [->] (core);
    \node[invisible] (lgap) at (3,0) {}
        edge [->] (p1);
    \node[] (cdot) at (3.5,0) {$\cdots$};
    \node[invisible] (rgap) at (4,0) {};
    \node[biw] (p2) at (5,0) {$0$\nodepart{lower}$\scriptstyle n_{l-1}$}
        edge [->] (rgap);
    \node[biw] (p3) at (7,0) {$0$\nodepart{lower}$n_{l}$}
        edge [->] (p2);
    \node[invisible] (out1) at (-1,0.33) {}
        edge [<-] (core);
    \node[] (dotout) at (-0.6,0.6) {$\cdots$};
    \node[invisible] (out2) at (-0.33,1) {}
        edge [<-] (core);
    \node[invisible] (in1) at (-1,-0.33) {}
        edge [->] (core);
    \node[] (dotin) at (-0.6,-0.6) {$\cdots$};
    \node[invisible] (in2) at (-0.33,-1) {}
        edge [->] (core);
    \node[] at (4.5,-0.8) {$\underbrace{\qquad \qquad \qquad \qquad \qquad \qquad \qquad \qquad \ \ \ \ }$};
    \node[] at (4.5,-1.3) {$l$};
    \node[rotate=45] at (-0.86,0.86) {$\overbrace{\qquad \qquad}$};
    \node[] at (-1.2,1.2) {$a$};
    \node[rotate=135] at (-0.86,-0.86) {$\overbrace{\qquad \qquad}$};
    \node[] at (-1.2,-1.2) {$b$};
    \end{tikzpicture}
\end{equation*}
with $l\geq 0$. If $l=0$, then $n_0\geq 1$ when $(a,b,c)$ is either of the three cases $(0,1,1)$, $(1,0,0)$ or $(a,0,c)$ for $a\geq 1$ and $c\geq 0$. In any other case for $(a,b,c)$ we have that $n_0\geq 0$.
If $l\geq 1$, then $n_0\geq 0$, $n_i\geq 1$ for $1\leq i\leq l-1$ and $n_l\geq 2$.
If $l=0$, the induced differential acts on the root vertex as
\begin{equation*}
    d\Big(
        \begin{tikzpicture}[shorten >=1pt,node distance=1.2cm,auto,>=angle 90,baseline=-0.1cm]
            \node[biw] (a) at (0,0) {$c$\nodepart{lower}$n_0$};
            \node[] (up1) at (-0.4,0.8) {}
                edge [-] (a);
            \node[] (up2) at (0,0.6) {$\scriptstyle\cdots$};
            \node[] (up3) at (0.4,0.8) {}
                edge [-] (a);
        \end{tikzpicture}
    \Big)
    =
    \sum_{\substack{n_0=n_{0}'+n_{0}''\\ n_{0}'\geq 0,\ n_{0}''\geq 2}}\ 
        \begin{tikzpicture}[shorten >=1pt,node distance=1.2cm,auto,>=angle 90,baseline=-0.1cm]
            \node[biw] (up) at (0,0.6) {$c$\nodepart{lower}$n_{0}'$};
            \node[biw] (down) at (0,-0.6) {$0$\nodepart{lower}$n_{0}''$}
                edge [->] (up);
            \node[] (up1) at (-0.4,1.4) {}
                edge [-] (up);
            \node[] (up2) at (0,1.2) {$\scriptstyle\cdots$};
            \node[] (up3) at (0.4,1.4) {}
                edge [-] (up);
        \end{tikzpicture}
    \ - \ 
        \begin{tikzpicture}[shorten >=1pt,node distance=1.2cm,auto,>=angle 90,baseline=-0.1cm]
            \node[biw] (up) at (0,0.6) {$c$\nodepart{lower}$\scriptstyle n_0-1$};
            \node[biw] (down) at (0,-0.6) {$0$\nodepart{lower}$\infty_2$}
                edge [->] (up);
            \node[] (up1) at (-0.4,1.4) {}
                edge [-] (up);
            \node[] (up2) at (0,1.2) {$\scriptstyle\cdots$};
            \node[] (up3) at (0.4,1.4) {}
                edge [-] (up);
        \end{tikzpicture}
\end{equation*}
and when $l\geq 1$ it acts on the root vertex as
\begin{equation*}
    d\Big(
        \begin{tikzpicture}[shorten >=1pt,node distance=1.2cm,auto,>=angle 90,baseline=-0.1cm]
            \node[biw] (a) at (0,0) {$c$\nodepart{lower}$n_0$};
            \node[] (up1) at (-0.4,0.8) {}
                edge [-] (a);
            \node[] (up2) at (0,0.6) {$\scriptstyle\cdots$};
            \node[] (up3) at (0.4,0.8) {}
                edge [-] (a);
            \node[] (down) at (0,-1) {}
                edge [->] (a);
        \end{tikzpicture}
    \Big)
    =
    \sum_{\substack{n_0=n_{0}'+n_{0}''\\ n_{0}'\geq 0,\ n_{0}''\geq 1}}\ 
        \begin{tikzpicture}[shorten >=1pt,node distance=1.2cm,auto,>=angle 90,baseline=-0.1cm]
            \node[biw] (up) at (0,0.6) {$0$\nodepart{lower}$n_{0}'$};
            \node[biw] (down) at (0,-0.6) {$0$\nodepart{lower}$n_{0}''$}
                edge [->] (up);
            \node[] (up1) at (-0.4,1.4) {}
                edge [-] (up);
            \node[] (up2) at (0,1.2) {$\scriptstyle\cdots$};
            \node[] (up3) at (0.4,1.4) {}
                edge [-] (up);
            \node[] (in) at (0,-1.6) {}
                edge [->] (down);
        \end{tikzpicture}.
\end{equation*}
Also for $l\geq 1$ there are passing vertices and a univalent vertex in the graph. The differential acts on these vertices as
\begin{equation*}
    d\Big(
        \begin{tikzpicture}[shorten >=1pt,node distance=1.2cm,auto,>=angle 90,baseline=-0.1cm]
            \node[biw] (a) at (0,0) {$0$\nodepart{lower}$n_i$};
            \node[] (up) at (0,1) {}
                edge [<-] (a);
            \node[] (down) at (0,-1) {}
                edge [->] (a);
        \end{tikzpicture}
    \Big)
    =
    \sum_{\substack{n_i=n_i'+n_{i}''\\ n_i',n_{i}''\geq 1}}\ 
        \begin{tikzpicture}[shorten >=1pt,node distance=1.2cm,auto,>=angle 90,baseline=-0.1cm]
            \node[biw] (up) at (0,0.6) {$0$\nodepart{lower}$n_i'$};
            \node[biw] (down) at (0,-0.6) {$0$\nodepart{lower}$n_{i}''$}
                edge [->] (up);
            \node[] (out) at (0,1.6) {}
                edge [<-] (up);
            \node[] (in) at (0,-1.6) {}
                edge [->] (down);
        \end{tikzpicture}
        \ ,\quad\quad
    d\Big(
        \begin{tikzpicture}[shorten >=1pt,node distance=1.2cm,auto,>=angle 90,baseline=-0.1cm]
            \node[biw] (a) at (0,0) {$0$\nodepart{lower}$n_l$};
            \node[] (up) at (0,1) {}
                edge [<-] (a);
        \end{tikzpicture}
    \Big)
    =
    \sum_{\substack{n_l=n_l'+n_{l}''\\ n_l'\geq 1,\ n_{l}''\geq 2}}\ 
        \begin{tikzpicture}[shorten >=1pt,node distance=1.2cm,auto,>=angle 90,baseline=-0.1cm]
            \node[biw] (up) at (0,0.6) {$0$\nodepart{lower}$n_l'$};
            \node[biw] (down) at (0,-0.6) {$0$\nodepart{lower}$n_{l}''$}
                edge [->] (up);
            \node[] (out) at (0,1.6) {}
                edge [<-] (up);
        \end{tikzpicture}
    \ - \ 
        \begin{tikzpicture}[shorten >=1pt,node distance=1.2cm,auto,>=angle 90,baseline=-0.1cm]
            \node[biw] (up) at (0,0.6) {$0$\nodepart{lower}$\scriptstyle n_l-1$};
            \node[biw] (down) at (0,-0.6) {$0$\nodepart{lower}$\infty_2$}
                edge [->] (up);
            \node[] (out) at (0,1.6) {}
                edge [<-] (up);
        \end{tikzpicture}.
\end{equation*}
We can already identify the first cohomology classes generated by the graph consisting of one core vertex on the form $\tikz[baseline=-2.07cm]\node[biw] at (0,-2) {$c$\nodepart{lower}$0$};$ (whenever this in-weight is possible).
The remaining graphs to study are then on the form above with $n_0\geq 1$ if $l=0$, and $n_0\geq 0$, $n_1,...,n_{l-1}\geq1$ and $n_l\geq 2$ if $l\geq 1$. Let $(\mathcal{I},d)$ be the complex generated by these graphs (i.e the complex where we exclude the graphs $\tikz[baseline=-2.07cm]\node[biw] at (0,-2) {$c$\nodepart{lower}$0$};$ from $u_1\mathcal{T}^{in}_{a,b,c}$ when present). It is easy to verify that that $\tikz[baseline=-2.07cm]\node[biw] at (0,-2) {$c$\nodepart{lower}$\infty_1$};$ is a cycle in $(\mathcal{I},d)$. Lemma \ref{lemma:u1Tin} follows if we can show that the cohomology of $\mathcal{I}$ is one-dimensional. There is a filtration of $\mathcal{I}$ over the total in-weight of a graph, as the total in-weight can not decrease under the differential. Let $\{T_r\mathcal{I}\}_{r\geq 0}$ be the associated spectral sequence. The following two results proves lemma \ref{lemma:u1Tin}.
\begin{lemma}
The cohomology of $(T_0\mathcal{I},d)$ is generated by the graph consisting of a single in-core vertex with in-weight $1$. That is $H(\mathcal{I},\delta)=\Big\langle\tikz[baseline=-2.07cm]\node[biw] at (0,-2) {$c$\nodepart{lower}$1$};\Big\rangle$.
\end{lemma}
\begin{proof}
We construct an isomorphic complex using the bar and cobar-construction. Let $V=V_{-1}\oplus V_0$ be a graded vector space with $V_0=\mathbb{K}$ and $V_{-1}=\mathbb{K}a$. We consider the augmented dga-algebra structure on $V$ where the product is defined by $\mu(a,a)=0$, and the differential $d$ is zero.
The bar complex $B(V)=(T^c(\overline{V}),d_B)$ satisfies $d_B=0$. Recall that $T^c(\overline{V})=\mathbb{K}\oplus sVa\oplus (sVa)^{\otimes 2} \oplus ...$ . All elements are of zero degree and a we denote the generator of $(sVa)^{\otimes n}$ by $sa\otimes\cdots\otimes sa=[sa]^n$. Also recall the coproduct is defined as $\Delta([sa]^n)=\sum_{i=1}^{n-1}[sa]^i\otimes[sa]^{n-i}$.
Next consider the cobar complex $\Omega(B(V))=(T(s^{-1}\overline{T^c(s\overline{V})}),d_\Omega)$. The degree of the elements are ranging from $0$ to $-\infty$. A general element of degree $-k$ is on the form $s^{-1}[sa]^{p_1}\otimes \cdots \otimes s^{-1}[sa]^{p_k}$. The differential $d_\Omega$ is now defined as $d_\Omega(s^{-1}[sa]^{p_1}\otimes \cdots \otimes s^{-1}[sa]^{p_k})=\sum_{i=1}^{k}(-1)^{1-i} s^{-1}[sa]^{p_1}\otimes\cdots\otimes \Delta(s^{-1}[sa]^{p_i}) \otimes\cdots\otimes s^{-1}[sa]^{p_k}$.
If we reverse the grading of $\Omega(B(V))$ and consider the reduced complex, it is easy to check that $\overline{\Omega(B(V))}$ is isomorphic $(T_0\mathcal{I},d)$ as a complex. In this isomorphism the element $s^{-1}[sa]^{p_1}\otimes \cdots \otimes s^{-1}[sa]^{p_k}$ maps to the graph with $k$ vertices, the in-core vertex having in-weight $p_1-1$ and the $i:th$ vertex having in-weight $p_i$ for $i\geq 2$. 
Now $\Omega(B(V))$ is quasi-isomorphic to $V$ (for example see \cite{LV}). $H(V)$ is generated by $1_\mathbb{K}$ and $a$. Hence we note that $[s^{-1}sa]$ is the only cohomology class of $\overline{\Omega(B(V))}$, which correspond in $(T_0\mathcal{I},d)$ to the single vertex graph with in-weight one.
\end{proof}
\begin{corollary}\label{cor:Helpcomplex}
    The cohomology group of $(\mathcal{I},d)$ is generated by the graph
    \begin{equation*}
        \tikz[baseline=-2.07cm]
        \node[biw] at (0,-2) {$c$\nodepart{lower}$\infty_1$};
        =\sum_{i=1}^\infty\ 
        \tikz[baseline=-2.07cm]
        \node[biw] at (0,-2) {$c$\nodepart{lower}$i$};\ .
    \end{equation*}
\end{corollary}
\begin{proof}
    By the above argument, we already know that $H(\mathcal{I},d)$ is one-dimensional and that its generating class given by $\tikz[baseline=-2.07cm]
    \node[biw] at (0,-2) {$c$\nodepart{lower}$1$};$ plus higher in-weight terms. It is straight-forward to see that the sum $\tikz[baseline=-2.07cm]
    \node[biw] at (0,-2) {$c$\nodepart{lower}$\infty_1$};$ is a cycle which satisfy this property.
\end{proof}
\begin{proof}[Proof of lemma \ref{lemma:uNTin}]
We need to show that $u_N\mathcal{T}^{in}_{a,b,c}$ is acyclic for $N\geq 2$. We say that a vertex $x$ is a \textit{branch vertex} if there are at least two paths starting at two different univalent special in-vertices and ending at $x$ (see figure \ref{fig:branch}). The number of branch vertices can not decrease under the differential, and so consider a filtration over the number of branch vertices. In the associated graded complex $gr(u_N\mathcal{T}^{in}_{a,b,c})$ the non branch vertices of a graph are attached to branch vertices as strings of passing vertices, and the differential acts by prolonging these strings. The differential can not increase the in-weight of a branch vertex. We consider a filtration over the total sum of the in-weights of the branch vertices. The differential of the associated graded complex $gr(gr(u_N\mathcal{T}^{in}_{a,b,c}))$ now act only act on the non branch vertices. The branch vertices and the number of non branch vertices that are attached to a branch vertex are invariant under the differential. By contracting the strings of non branch vertices in a graph $\Gamma$ into $N$ hairs, we get a \textit{branch graph} $\Gamma_{br}$ whose vertices are branch vertices. Then the complex splits as
\begin{equation*}
    gr(gr(u_N\mathcal{T}^{in}_{a,b,c}))=\bigoplus_{\Gamma_{br}}\mathsf{branchGraph}(\Gamma_{br})
\end{equation*}
summed over the set of all branch graphs $\Gamma_{br}$. These complexes decompose as
\begin{equation*}
    \mathsf{branchGraph}(\Gamma_{br})=\big(\bigotimes_{h\in H(\Gamma_{br})}\mathcal{I}\big)^{\mathrm{Aut}(\Gamma_{br})}
\end{equation*}
where $H(\Gamma_{br})$ is the set of hairs in $\Gamma_{br}$ and $\mathcal{I}$ is the complex from lemma \ref{cor:Helpcomplex} and $\mathrm{Aut}(\Gamma_{br})$ is the group of symmetries of $\Gamma_{br}$ acting with the appropriate signs. Since this group is finite, Maschke's theorem gives
\begin{equation*}
    H\Big(\big(\bigotimes_{h\in H(\Gamma_{br})}\mathcal{I}\big)^{\mathrm{Aut}(\Gamma_{br})}\Big)
    \cong\Big(H\big(\bigotimes_{h\in H(\Gamma_{br})}\mathcal{I}\big)\Big)^{\mathrm{Aut}(\Gamma_{br})}
    \cong\big(\bigotimes_{h\in H(\Gamma_{br})}H(\mathcal{I})\big)^{\mathrm{Aut}(\Gamma_{br})}
\end{equation*}
Now $H(\mathcal{I})$ is generated by one element by corollary \ref{cor:Helpcomplex}. In a branch graph there is at least one vertex with two or more hairs. Hence the cohomology classes are zero in $\mathsf{branchGraph}(\Gamma_{br})$ due to symmetries, finishing the proof.
\end{proof}
\begin{figure}[h]
    \begin{equation*}
    \begin{tikzpicture}[shorten >=1pt,node distance=1.2cm,auto,>=angle 90,baseline=-0.1cm]
        \node[biw] (root) at (0,0) {$3$\nodepart{lower}$4$};
        \node[] (up1) at (-0.4,0.8) {}
            edge [-] (root);
        \node[] (up2) at (0,0.6) {$\scriptstyle\cdots$};
        \node[] (up3) at (0.4,0.8) {}
            edge [-] (root);
        \node[biw] (b1) at (-0.7,-1.1) {$0$\nodepart{lower}$7$}
            edge [->] (root);
        \node[biw] (b11) at (-2.1,-1.3) {$0$\nodepart{lower}$5$}
            edge [->] (b1);
        \node[biw] (b21) at (-2,-2.3) {$0$\nodepart{lower}$0$}
            edge [->] (b1);
        \node[biw] (b22) at (-3.4,-2.5) {$0$\nodepart{lower}$2$}
            edge [->] (b21);
        \node[biw] (b23) at (-2,-3.7) {$0$\nodepart{lower}$5$}
            edge [->] (b21);
        \node[biw] (b31) at (-0.7,-2.4) {$0$\nodepart{lower}$3$}
            edge [->] (b1);
        \node[biw] (c1) at (0.7,-1.1) {$0$\nodepart{lower}$0$}
            edge [->] (root);
        \node[biw] (c2) at (0.7,-2.4) {$0$\nodepart{lower}$1$}
            edge [->] (c1);
        \node[biw] (c3) at (2.1,-1.3) {$0$\nodepart{lower}$4$}
            edge [->] (c1);
        \node[biw] (c4) at (0.7,-3.7) {$0$\nodepart{lower}$1$}
            edge [->] (c2);
        \node[biw] (c5) at (0.7,-5) {$0$\nodepart{lower}$2$}
            edge [->] (c4);
    \end{tikzpicture}
    \quad\quad
    \begin{tikzpicture}[shorten >=1pt,node distance=1.2cm,auto,>=angle 90,baseline=-0.1cm]
        \node[biw] (root) at (0,0) {$3$\nodepart{lower}$4$};
        \node[] (up1) at (-0.4,0.8) {}
            edge [-] (root);
        \node[] (up2) at (0,0.6) {$\scriptstyle\cdots$};
        \node[] (up3) at (0.4,0.8) {}
            edge [-] (root);
        \node[biw] (b1) at (-0.7,-1.1) {$0$\nodepart{lower}$7$}
            edge [->] (root);
        \node[] (b11) at (-2.1,-1.3) {}
            edge [-] (b1);
        \node[biw] (b21) at (-2,-2.3) {$0$\nodepart{lower}$0$}
            edge [->] (b1);
        \node[] (b22) at (-3.4,-2.5) {}
            edge [-] (b21);
        \node[] (b23) at (-2,-3.7) {}
            edge [-] (b21);
        \node[] (b31) at (-0.7,-2.4) {}
            edge [-] (b1);
        \node[biw] (c1) at (0.7,-1.1) {$0$\nodepart{lower}$0$}
            edge [->] (root);
        \node[] (c2) at (0.7,-2.4) {}
            edge [-] (c1);
        \node[] (c3) at (2.1,-1.3) {}
            edge [-] (c1);
    \end{tikzpicture}
\end{equation*}
    \caption{Picture of a graph and its corresponding branch graph.}
    \label{fig:branch}
\end{figure}
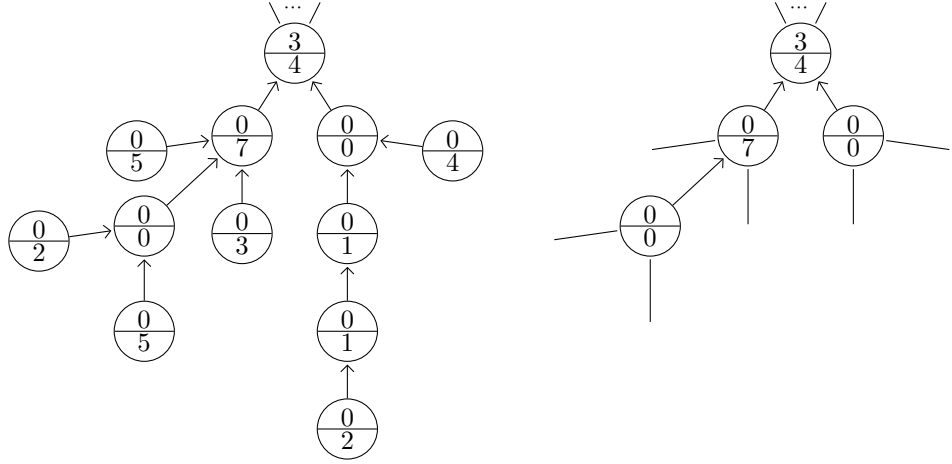
Lastly we consider the complexes $\wGC^*_k$ and $\wGC^+_k$. Let $\{S_r^{in}\wGC_k^*\}_{r\geq 0}$ and $\{S_r^{in}\wGC_k^+\}_{r\geq 0}$ be the spectral sequences associated to the filtration over the number of in-core vertices.
\begin{prop}\label{prop:OutIn1}
    The page one complex $S^{in}_1\wGC^*_k$ is a subcomplex of $S^{in}_1\wGC_k$ generated by directed graphs whose vertices are independently decorated by the two types 
    $\Bigl\{\begin{tikzpicture}[baseline={([yshift=-.5ex]current bounding box.center)},scale=0.7]
        \node[biw] at (0,0) {$m$\nodepart{lower}$\infty_1$};
    \end{tikzpicture}\Bigr\}_{m\geq 0}$ and 
    $\Bigl\{\begin{tikzpicture}[baseline={([yshift=-.5ex]current bounding box.center)},scale=0.7]
        \node[biw] at (0,0) {$m$\nodepart{lower}$0$};
    \end{tikzpicture}\Bigr\}_{m\geq 0}$ subject to the conditions of proposition \ref{prop:specialIn} as well as the additional condition that either
    \begin{itemize}
        \item at least one vertex is decorated with the bi-weight
        $\begin{tikzpicture}[baseline={([yshift=-.5ex]current bounding box.center)},scale=0.7]
        \node[biw] at (0,0) {$m$\nodepart{lower}$\infty_1$};
    \end{tikzpicture}$ for some $m\geq 0$, or
    \item at least one vertex is decorated with the bi-weight
    $\begin{tikzpicture}[baseline={([yshift=-.5ex]current bounding box.center)},scale=0.7]
        \node[biw] at (0,0) {$m$\nodepart{lower}$0$};
    \end{tikzpicture}$ 
    for some $m\geq 1$.
    \end{itemize}
    Furthermore, the page one complex $S^{in}_1\wGC^+_k$ is a subcomplex of $S^{in}_1\wGC^*_k$ where each graph additionally satisfies that either
    \begin{itemize}
        \item one vertex is decorated with the bi-weight
        $\begin{tikzpicture}[baseline={([yshift=-.5ex]current bounding box.center)},scale=0.7]
        \node[biw] at (0,0) {$m$\nodepart{lower}$\infty_1$};
    \end{tikzpicture}$ for some $m\geq 1$, or
    \item two vertices are decorated with the bi-weights 
    $\begin{tikzpicture}[baseline={([yshift=-.5ex]current bounding box.center)},scale=0.7]
        \node[biw] at (0,0) {$m$\nodepart{lower}$0$};
    \end{tikzpicture}$ 
     and 
    $\begin{tikzpicture}[baseline={([yshift=-.5ex]current bounding box.center)},scale=0.7]
        \node[biw] at (0,0) {$0$\nodepart{lower}$\infty_1$};
    \end{tikzpicture}$ respectively for some $m\geq 1$.
    \end{itemize}
\end{prop}
\begin{proof}
    Similar to the case of $S_0^{in}\wGC_k$, both $S_0^{in}\wGC_k^*$ and $S_0^{in}\wGC_k^+$ decompose over in-core graphs $\gamma$ as
    \begin{align*}
        S_0^{in}\wGC_k^*&\cong \bigoplus_{\gamma}\mathsf{inCore}^*(\gamma)\\
        S_0^{in}\wGC_k^+&\cong \bigoplus_{\gamma}\mathsf{inCore}^+(\gamma)
    \end{align*}
    where $\mathsf{inCore}^*(\gamma)$ and $\mathsf{inCore}^+(\gamma)$ are the complex generated by graphs having $\gamma$ as their in-core graph. These complex do not however split into a tensor over the tree complexes $\mathcal{T}_x^{in}$, since some of the tensors will not represent graphs in these complexes. For example, a graph having all vertices decorated by $\frac{0}{0}$ or a graph having only zero out-weights or zero in-weights respectively.
    The tree complexes $\mathcal{T}_x^{in}$ split into (at most) four complexes as $\mathcal{T}_x^{in}=\mathcal{T}_x^{in}(0)\oplus\mathcal{T}_x^{in}(out)\oplus\mathcal{T}_x^{in}(in)\oplus\mathcal{T}_x^{in}(out\wedge in)$, where $\mathcal{T}_x^{in}(0)$ is the complex where all bi-weights are zero, $\mathcal{T}_x^{in}(out)$ is the complex where all in bi-weights are zero and at least one vertex is decorated with positive out bi-weight, $\mathcal{T}_x^{in}(in)$ is the complex where all out bi-weights are zero and at least one vertex is decorated with positive in bi-weight, and $\mathcal{T}_x^{in}(out\wedge in)$ is the complex where at least one vertex is decorated with positive out bi-weight and and at least one vertex with positive in bi-weight. Depending on the vertex $x$, some of the three first complexes might be zero in the decomposition of $\mathcal{T}_x^{out}$. It is easy to verify that the differential preserve the decomposition. We then get that
    \begin{align*}
        \mathsf{inCore}^*(\gamma)=\bigoplus_I\big(\mathcal{T}_{x_1}^{out}(I_1)\otimes \mathcal{T}_{x_2}^{out}(I_2)\otimes...\otimes\mathcal{T}_{x_k}^{out}(I_k)\big)\\
        \mathsf{inCore}^+(\gamma)=\bigoplus_J\big(\mathcal{T}_{x_1}^{out}(J_1)\otimes \mathcal{T}_{x_2}^{out}(J_2)\otimes...\otimes\mathcal{T}_{x_k}^{out}(J_k)\big)
    \end{align*}
    where the first sum runs over signatures $I=(I_1,I_2,...,I_{k})\in\{0,\ in,\ out,\ in\wedge out\}^{k}$ such that there is an $i$ such that $I_i\neq0$. Similarly the second sum runs over signatures $J=(J_1,J_2,...,J_k)\in\{0,\ in,\ out,\ in\wedge out\}^{k}$ such that either there is an $i$ such that $J_i=out\wedge in$, or there are $i,j$ such that $J_i=in$ and $J_j=out$. Lemma \ref{lemma:intreecohom} gives us that 
    \begin{align*}
        H(\mathcal{T}_x^{out}(out\wedge in))\ =\ &\Big\langle \begin{tikzpicture}[baseline={([yshift=-.5ex]current bounding box.center)},scale=0.7]
        \node[biw] at (0,0) {$m$\nodepart{lower}$\infty_1$};
    \end{tikzpicture}\ : \ m\geq 1 \Big\rangle \\
    H(\mathcal{T}_x^{out}(out))\ = \ &\Big\langle \begin{tikzpicture}[baseline={([yshift=-.5ex]current bounding box.center)},scale=0.7]
        \node[biw] at (0,0) {$m$\nodepart{lower}$0$};
    \end{tikzpicture}\ : \ m\geq 1 \Big\rangle \\
    H(\mathcal{T}_x^{out}(in))\ =\ &\Big\langle \begin{tikzpicture}[baseline={([yshift=-.5ex]current bounding box.center)},scale=0.7]
        \node[biw] at (0,0) {$0$\nodepart{lower}$\infty_1$};
    \end{tikzpicture} \Big\rangle \\
    H(\mathcal{T}_x^{out}(0))\ =\ &\Big\langle \begin{tikzpicture}[baseline={([yshift=-.5ex]current bounding box.center)},scale=0.7]
        \node[biw] at (0,0) {$0$\nodepart{lower}$0$};
    \end{tikzpicture} \Big\rangle
    \end{align*}
    The proposition immediately follows by comparing the condition for the signatures $I$ and $J$ with the proposition statement.
\end{proof}

\subsection{Filtrations over special out-vertices}
We define special out-vertices by analogy to the special in-vertices introduced above.
\begin{definition}
A vertex $x$ in a graph $\Gamma$ is called a \textit{special out-vertex} if
\begin{enumerate}
    \item[i)] either $x$ is a univalent vertex with one incoming edge and in-weight zero, i.e on the form $\begin{tikzpicture}[shorten >=1pt,node distance=1.2cm,auto,>=angle 90,baseline=-0.1cm]
            \node[biw] (a) at (0,0) {$m$\nodepart{lower}$0$};
            \node[] (d) at (0,-1) {}
                edge [->] (a);
        \end{tikzpicture}$
    \item[ii)] or $v$ becomes a univalent vertex of type i) after recursive removal of all special-out vertices of type i) from $\Gamma$.
\end{enumerate}
\end{definition}
Vertices which are not special out-vertices are called \textit{out-core vertices} or just \textit{core vertices}.
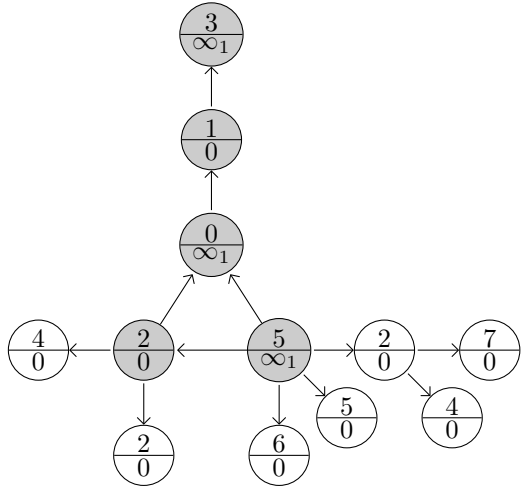
\begin{figure}[h]
    \begin{equation*}
    \begin{tikzpicture}[shorten >=1pt,node distance=1.2cm,auto,>=angle 90,baseline=-0.1cm]
        \node[biw,fill=mygray] (up) at (0,0) {$0$\nodepart{lower}$\infty_1$};
        \node[biw,fill=mygray] (left) at (-0.9,-1.4) {$2$\nodepart{lower}$0$}
            edge [->] (up);
        \node[biw,fill=mygray] (right) at (0.9,-1.4) {$5$\nodepart{lower}$\infty_1$}
            edge [->] (up)
            edge [->] (left);
        \node[biw,fill=mygray] (up1) at (0,1.4) {$1$\nodepart{lower}$0$}
            edge [<-] (up);
        \node[biw,fill=mygray] (up2) at (0,2.8) {$3$\nodepart{lower}$\infty_1$}
            edge [<-] (up1);
        \node[biw] (left1) at (-2.3,-1.4) {$4$\nodepart{lower}$0$}
            edge [<-] (left);
        \node[biw] (left2) at (-0.9,-2.8) {$2$\nodepart{lower}$0$}
            edge [<-] (left);
        \node[biw] (right1) at (2.3,-1.4) {$2$\nodepart{lower}$0$}
            edge [<-] (right);
        \node[biw] (right2) at (1.8,-2.3) {$5$\nodepart{lower}$0$}
            edge [<-] (right);
        \node[biw] (right3) at (0.9,-2.8) {$6$\nodepart{lower}$0$}
            edge [<-] (right);
        \node[biw] at (3.7,-1.4) {$7$\nodepart{lower}$0$}
            edge [<-] (right1);
        \node[biw] at (3.2,-2.3) {$4$\nodepart{lower}$0$}
            edge [<-] (right1);    
    \end{tikzpicture}
\end{equation*}
    \caption{Example of a graph in $S_1^{in}\wGC_k$ where out-core vertices are colored gray and the special out-vertices are colored white.}
    \label{fig:SpecOut}
\end{figure}
\newline
Note that there the are no special-in vertices in any graph of $S^{in}_1\wGC_k$, but there are graphs with special out-vertices (see figure \ref{fig:SpecOut}). Given an arbitrary graph $\Gamma$ in $S^{in}_1\wGC_k$, we define the associated \textit{core graph} $\gamma$ as the graph spanned by core vertices with their out-weight forgotten. Similar to before, we consider a filtration of $S^{in}_1\wGC_k$ over the number of core-vertices and let $\{S^{out}_r\wGC_k\}_{r\geq 0}$ be the associated spectral sequence. The differential acts by only creating special out-vertices and leaves the connected core-graph unchanged.
\begin{proposition}\label{prop:specialOut}
    The page one complex $S^{out}_1\wGC_d$ is generated by graphs $\Gamma$ whose vertices $V(\Gamma)$ are independently decorated with four possible bi-weights
$\begin{tikzpicture}[baseline={([yshift=-.5ex]current bounding box.center)}]
    \node[biw] at (0,0) {$\infty_1$\nodepart{lower}$\infty_1$};
\end{tikzpicture}$, 
$\begin{tikzpicture}[baseline={([yshift=-.5ex]current bounding box.center)}]
    \node[biw] at (0,0) {$0$\nodepart{lower}$\infty_1$};
\end{tikzpicture}$,
$\begin{tikzpicture}[baseline={([yshift=-.5ex]current bounding box.center)}]
    \node[biw] at (0,0) {$\infty_1$\nodepart{lower}$0$};
\end{tikzpicture}$, and
$\begin{tikzpicture}[baseline={([yshift=-.5ex]current bounding box.center)}]
    \node[biw] at (0,0) {$0$\nodepart{lower}$0$};
\end{tikzpicture}$ subject to the following conditions:
\begin{enumerate}
    \item If $x\in V(\Gamma)$ is univalent, then
    \begin{equation*}
        x=\begin{tikzpicture}[shorten >=1pt,node distance=1.2cm,auto,>=angle 90,baseline=-0.1cm]
            \node[biw] (a) at (0,0) {$\infty_1$\nodepart{lower}$\infty_1$};
            \node[] (d) at (0,-1) {}
                edge [->] (a);
        \end{tikzpicture}
        \text{ or }
        x=\begin{tikzpicture}[shorten >=1pt,node distance=1.2cm,auto,>=angle 90,baseline=-0.1cm]
            \node[biw] (a) at (0,0) {$\infty_1$\nodepart{lower}$\infty_1$};
            \node[] (d) at (0,1) {}
                edge [<-] (a);
        \end{tikzpicture}
    \end{equation*}
    \item If $x\in V(\Gamma)$ is a source with at least two out-edges, then
    \begin{equation*}
        x=\begin{tikzpicture}[shorten >=1pt,node distance=1.2cm,auto,>=angle 90,baseline=-0.1cm]
            \node[biw] (a) at (0,0) {$\infty_1$\nodepart{lower}$\infty_1$};
            \node[] (up1) at (-0.4,0.8) {}
                edge [<-] (a);
            \node[] (up2) at (0,0.6) {$\scriptstyle\cdots$};
            \node[] (up2) at (0,0.9) {$\scriptstyle \geq 2$};
            \node[] (up3) at (0.4,0.8) {}
                edge [<-] (a);
        \end{tikzpicture}
        \text{ or }
        x=\begin{tikzpicture}[shorten >=1pt,node distance=1.2cm,auto,>=angle 90,baseline=-0.1cm]
            \node[biw] (a) at (0,0) {$0$\nodepart{lower}$\infty_1$};
            \node[] (up1) at (-0.4,0.8) {}
                edge [<-] (a);
            \node[] (up2) at (0,0.6) {$\scriptstyle\cdots$};
            \node[] (up2) at (0,0.9) {$\scriptstyle \geq 2$};
            \node[] (up3) at (0.4,0.8) {}
                edge [<-] (a);
        \end{tikzpicture}
    \end{equation*}
    \item If $x\in V(\Gamma)$ is a target with at least two in-edges, then
     \begin{equation*}
        x=\begin{tikzpicture}[shorten >=1pt,node distance=1.2cm,auto,>=angle 90,baseline=-0.1cm]
            \node[biw] (a) at (0,0) {$\infty_1$\nodepart{lower}$\infty_1$};
            \node[] (up1) at (-0.4,-0.8) {}
                edge [->] (a);
            \node[] (up2) at (0,-0.6) {$\scriptstyle\cdots$};
            \node[] (up2) at (0,-0.9) {$\scriptstyle \geq 2$};
            \node[] (up3) at (0.4,-0.8) {}
                edge [->] (a);
        \end{tikzpicture}
        \text{ or }
        x=\begin{tikzpicture}[shorten >=1pt,node distance=1.2cm,auto,>=angle 90,baseline=-0.1cm]
            \node[biw] (a) at (0,0) {$\infty_1$\nodepart{lower}$0$};
            \node[] (up1) at (-0.4,-0.8) {}
                edge [->] (a);
            \node[] (up2) at (0,-0.6) {$\scriptstyle\cdots$};
            \node[] (up2) at (0,-0.9) {$\scriptstyle \geq 2$};
            \node[] (up3) at (0.4,-0.8) {}
                edge [->] (a);
        \end{tikzpicture}
    \end{equation*}
    \item If $x\in V(\Gamma)$ is passing (one in-edge and one out-edge), then
    \begin{equation*}
        x=\begin{tikzpicture}[shorten >=1pt,node distance=1.2cm,auto,>=angle 90,baseline=-0.1cm]
            \node[biw] (a) at (0,0) {$\infty_1$\nodepart{lower}$\infty_1$};
            \node[] (up) at (0,1) {}
                edge [<-] (a);
            \node[] (down) at (0,-1) {}
                edge [->] (a);
        \end{tikzpicture}
        \ , \
        x=\begin{tikzpicture}[shorten >=1pt,node distance=1.2cm,auto,>=angle 90,baseline=-0.1cm]
            \node[biw] (a) at (0,0) {$0$\nodepart{lower}$\infty_1$};
            \node[] (up) at (0,1) {}
                edge [<-] (a);
            \node[] (down) at (0,-1) {}
                edge [->] (a);
        \end{tikzpicture}
        \text{ or }
        x=\begin{tikzpicture}[shorten >=1pt,node distance=1.2cm,auto,>=angle 90,baseline=-0.1cm]
            \node[biw] (a) at (0,0) {$\infty_1$\nodepart{lower}$0$};
            \node[] (up) at (0,1) {}
                edge [<-] (a);
            \node[] (down) at (0,-1) {}
                edge [->] (a);
        \end{tikzpicture}
    \end{equation*}
    \item If $x\in V(\Gamma)$ is none of the above types (i.e $x$ is at least trivalent and has at least one in-edge and at least one out-edge), then
    \begin{equation*}
        x=\begin{tikzpicture}[shorten >=1pt,node distance=1.2cm,auto,>=angle 90,baseline=-0.1cm]
            \node[biw] (a) at (0,0) {$\infty_1$\nodepart{lower}$\infty_1$};
            \node[] (down1) at (-0.4,-0.8) {}
                edge [->] (a);
            \node[] (down2) at (0,-0.6) {$\scriptstyle\cdots$};
            \node[] (down3) at (0.4,-0.8) {}
                edge [->] (a);
            \node[] (up1) at (-0.4,0.8) {}
                edge [<-] (a);
            \node[] (up2) at (0,0.6) {$\scriptstyle\cdots$};
            \node[] (up3) at (0.4,0.8) {}
                edge [<-] (a);
        \end{tikzpicture}
        \ , \
        x=\begin{tikzpicture}[shorten >=1pt,node distance=1.2cm,auto,>=angle 90,baseline=-0.1cm]
            \node[biw] (a) at (0,0) {$0$\nodepart{lower}$\infty_1$};
            \node[] (down1) at (-0.4,-0.8) {}
                edge [->] (a);
            \node[] (down2) at (0,-0.6) {$\scriptstyle\cdots$};
            \node[] (down3) at (0.4,-0.8) {}
                edge [->] (a);
            \node[] (up1) at (-0.4,0.8) {}
                edge [<-] (a);
            \node[] (up2) at (0,0.6) {$\scriptstyle\cdots$};
            \node[] (up3) at (0.4,0.8) {}
                edge [<-] (a);
        \end{tikzpicture}
        \ , \
        x=\begin{tikzpicture}[shorten >=1pt,node distance=1.2cm,auto,>=angle 90,baseline=-0.1cm]
            \node[biw] (a) at (0,0) {$\infty_1$\nodepart{lower}$0$};
            \node[] (down1) at (-0.4,-0.8) {}
                edge [->] (a);
            \node[] (down2) at (0,-0.6) {$\scriptstyle\cdots$};
            \node[] (down3) at (0.4,-0.8) {}
                edge [->] (a);
            \node[] (up1) at (-0.4,0.8) {}
                edge [<-] (a);
            \node[] (up2) at (0,0.6) {$\scriptstyle\cdots$};
            \node[] (up3) at (0.4,0.8) {}
                edge [<-] (a);
        \end{tikzpicture}
        \text{ or }
        x=\begin{tikzpicture}[shorten >=1pt,node distance=1.2cm,auto,>=angle 90,baseline=-0.1cm]
            \node[biw] (a) at (0,0) {$0$\nodepart{lower}$0$};
            \node[] (down1) at (-0.4,-0.8) {}
                edge [->] (a);
            \node[] (down2) at (0,-0.6) {$\scriptstyle\cdots$};
            \node[] (down3) at (0.4,-0.8) {}
                edge [->] (a);
            \node[] (up1) at (-0.4,0.8) {}
                edge [<-] (a);
            \node[] (up2) at (0,0.6) {$\scriptstyle\cdots$};
            \node[] (up3) at (0.4,0.8) {}
                edge [<-] (a);
        \end{tikzpicture}
    \end{equation*}
\end{enumerate}
Let $\Gamma$ be a graph in $S^{out}_1\wGC_d$ and $x$ some vertex of $\Gamma$. Further let $a,b,c$ and $d$ be either of the symbols $\infty_1$ or $0$. Then we set $\Big(\frac{a}{b},\frac{c}{d}\Big)_x$ to denote the sum over all possible reattachements of the edges attached to $x$ among two new vertices $x'$ and $x''$ (connected by a single edge going from $x'$ to $x''$) of bi-weight $\frac{a}{b}$ and $\frac{c}{d}$ respectively. The reattachements are done in such a way that the resulting vertices are of the types $(1)-(5)$ described above. Using earlier notation for vertex splitting we get
\begin{equation*}
\Big(\frac{a}{b},\frac{c}{d}\Big)_x:=
\begin{tikzpicture}[shorten >=1pt,>=angle 90,baseline={([yshift=-.5ex]current bounding box.center)}]
        \node[ellipse,
            draw = black,
            minimum width = 3cm, 
            minimum height = 1.8cm,
            dotted] (e) at (0,0) {};
        \node[biw] (vLeft) at (-0.8,0) {$a$\nodepart{lower}$b$};
        \node[biw] (vRight) at (0.8,0) {$c$\nodepart{lower}$d$}
            edge [<-] (vLeft);
        \node[invisible] (v1) at (-0.6,1.4) {}
            edge [<-] (e);
        \node[] at (0,1.2) {$\cdots$};
        \node[invisible] (v3) at (0.6,1.4) {}
            edge [<-] (e);
        \node[invisible] (v4) at (-0.6,-1.4) {}
            edge [->] (e);
        \node[] at (0,-1.2) {$\cdots$};
        \node[invisible] (v6) at (0.6,-1.4) {}
            edge [->] (e);
    \end{tikzpicture}    
\end{equation*}
The differential $d$ acts on graphs $\Gamma\in S^{out}_1\wGC_k$ as $d(\Gamma)=\sum_{x\in V(\Gamma)}d_x(\Gamma)$. The map $d_x$ act on vertices with the four different bi-weights in the following way:
\begin{flalign*}
    d_x\Big(\ 
    \begin{tikzpicture}[shorten >=1pt,>=angle 90,baseline={([yshift=-.5ex]current bounding box.center)}]
        \node[biw] (v0) at (0,0) {$\infty_1$\nodepart{lower}$\infty_1$};
        \node[invisible] (v1) at (-0.4,0.8) {}
            edge [<-] (v0);
        \node[] at (0,0.6) {$\scriptstyle\cdots$};
        \node[invisible] (v3) at (0.4,0.8) {}
            edge [<-] (v0);
        \node[invisible] (v4) at (-0.4,-0.8) {}
            edge [->] (v0);
        \node[] at (0,-0.6) {$\scriptstyle\cdots$};
        \node[invisible] (v6) at (0.4,-0.8) {}
            edge [->] (v0);
    \end{tikzpicture}\ \Big) \
    &=\Big(\frac{\infty_1}{\infty_1},\frac{\infty_1}{\infty_1}\Big)_x
    +\Big(\frac{0}{\infty_1},\frac{\infty_1}{\infty_1}\Big)_x
    +\Big(\frac{\infty_1}{0},\frac{\infty_1}{\infty_1}\Big)_x
    +\Big(\frac{\infty_1}{\infty_1},\frac{0}{\infty_1}\Big)_x
    +\Big(\frac{\infty_1}{\infty_1},\frac{\infty_1}{0}\Big)_x\\
    &+\Big(\frac{0}{0},\frac{\infty_1}{\infty_1}\Big)_x
    +\Big(\frac{\infty_1}{0},\frac{0}{\infty_1}\Big)_x
    +\Big(\frac{0}{\infty_1},\frac{\infty_1}{0}\Big)_x
    +\Big(\frac{\infty_1}{\infty_1},\frac{0}{0}\Big)_x\\
    & - \Bigg( \ \ 
    \begin{tikzpicture}[shorten >=1pt,>=angle 90,baseline={(4ex,-0.5ex)}]
        \node[biw] (v0) at (0,0) {$\infty_1$\nodepart{lower}$\infty_1$};
        \node[biw] (new) at (1.5,0.7) {$\infty_1$\nodepart{lower}$\infty_1$}
            edge [<-] (v0);
        \node[invisible] (v1) at (-0.4,0.8) {}
            edge [<-] (v0);
        \node[] at (0,0.6) {$\scriptstyle\cdots$};
        \node[invisible] (v3) at (0.4,0.8) {}
            edge [<-] (v0);
        \node[invisible] (v4) at (-0.4,-0.8) {}
            edge [->] (v0);
        \node[] at (0,-0.6) {$\scriptstyle\cdots$};
        \node[invisible] (v6) at (0.4,-0.8) {}
            edge [->] (v0);
    \end{tikzpicture}
    \ +\ 
    \begin{tikzpicture}[shorten >=1pt,>=angle 90,baseline={(4ex,-0.5ex)}]
        \node[biw] (v0) at (0,0) {$0$\nodepart{lower}$\infty_1$};
        \node[biw] (new) at (1.5,0.7) {$\infty_1$\nodepart{lower}$\infty_1$}
            edge [<-] (v0);
        \node[invisible] (v1) at (-0.4,0.8) {}
            edge [<-] (v0);
        \node[] at (0,0.6) {$\scriptstyle\cdots$};
        \node[invisible] (v3) at (0.4,0.8) {}
            edge [<-] (v0);
        \node[invisible] (v4) at (-0.4,-0.8) {}
            edge [->] (v0);
        \node[] at (0,-0.6) {$\scriptstyle\cdots$};
        \node[invisible] (v6) at (0.4,-0.8) {}
            edge [->] (v0);
    \end{tikzpicture}
    \ +\ 
    \begin{tikzpicture}[shorten >=1pt,>=angle 90,baseline={(4ex,-0.5ex)}]
        \node[biw] (v0) at (0,0) {$\infty_1$\nodepart{lower}$\infty_1$};
        \node[biw] (new) at (1.5,-0.7) {$\infty_1$\nodepart{lower}$\infty_1$}
            edge [->] (v0);
        \node[invisible] (v1) at (-0.4,0.8) {}
            edge [<-] (v0);
        \node[] at (0,0.6) {$\scriptstyle\cdots$};
        \node[invisible] (v3) at (0.4,0.8) {}
            edge [<-] (v0);
        \node[invisible] (v4) at (-0.4,-0.8) {}
            edge [->] (v0);
        \node[] at (0,-0.6) {$\scriptstyle\cdots$};
        \node[invisible] (v6) at (0.4,-0.8) {}
            edge [->] (v0);
    \end{tikzpicture}
    \ +\ 
    \begin{tikzpicture}[shorten >=1pt,>=angle 90,baseline={(4ex,-0.5ex)}]
        \node[biw] (v0) at (0,0) {$\infty_1$\nodepart{lower}$0$};
        \node[biw] (new) at (1.5,-0.7) {$\infty_1$\nodepart{lower}$\infty_1$}
            edge [->] (v0);
        \node[invisible] (v1) at (-0.4,0.8) {}
            edge [<-] (v0);
        \node[] at (0,0.6) {$\scriptstyle\cdots$};
        \node[invisible] (v3) at (0.4,0.8) {}
            edge [<-] (v0);
        \node[invisible] (v4) at (-0.4,-0.8) {}
            edge [->] (v0);
        \node[] at (0,-0.6) {$\scriptstyle\cdots$};
        \node[invisible] (v6) at (0.4,-0.8) {}
            edge [->] (v0);
    \end{tikzpicture}
    \ \ \Bigg) &&
    \\
    \\
    d_x\Big(\ 
    \begin{tikzpicture}[shorten >=1pt,>=angle 90,baseline={([yshift=-.5ex]current bounding box.center)}]
        \node[biw] (v0) at (0,0) {$0$\nodepart{lower}$\infty_1$};
        \node[invisible] (v1) at (-0.4,0.8) {}
            edge [<-] (v0);
        \node[] at (0,0.6) {$\scriptstyle\cdots$};
        \node[invisible] (v3) at (0.4,0.8) {}
            edge [<-] (v0);
        \node[invisible] (v4) at (-0.4,-0.8) {}
            edge [->] (v0);
        \node[] at (0,-0.6) {$\scriptstyle\cdots$};
        \node[invisible] (v6) at (0.4,-0.8) {}
            edge [->] (v0);
    \end{tikzpicture}\ \Big) \
    &=\Big(\frac{0}{\infty_1},\frac{0}{\infty_1}\Big)_x
    +\Big(\frac{0}{0},\frac{0}{\infty_1}\Big)_x
    +\Big(\frac{0}{\infty_1},\frac{0}{0}\Big)_x\\
    & - \Bigg( \ \ 
    \begin{tikzpicture}[shorten >=1pt,>=angle 90,baseline={(4ex,-0.5ex)}]
        \node[biw] (v0) at (0,0) {$0$\nodepart{lower}$\infty_1$};
        \node[biw] (new) at (1.5,-0.7) {$\infty_1$\nodepart{lower}$\infty_1$}
            edge [->] (v0);
        \node[invisible] (v1) at (-0.4,0.8) {}
            edge [<-] (v0);
        \node[] at (0,0.6) {$\scriptstyle\cdots$};
        \node[invisible] (v3) at (0.4,0.8) {}
            edge [<-] (v0);
        \node[invisible] (v4) at (-0.4,-0.8) {}
            edge [->] (v0);
        \node[] at (0,-0.6) {$\scriptstyle\cdots$};
        \node[invisible] (v6) at (0.4,-0.8) {}
            edge [->] (v0);
    \end{tikzpicture}
    \ +\ 
    \begin{tikzpicture}[shorten >=1pt,>=angle 90,baseline={(4ex,-0.5ex)}]
        \node[biw] (v0) at (0,0) {$0$\nodepart{lower}$0$};
        \node[biw] (new) at (1.5,-0.7) {$\infty_1$\nodepart{lower}$\infty_1$}
            edge [->] (v0);
        \node[invisible] (v1) at (-0.4,0.8) {}
            edge [<-] (v0);
        \node[] at (0,0.6) {$\scriptstyle\cdots$};
        \node[invisible] (v3) at (0.4,0.8) {}
            edge [<-] (v0);
        \node[invisible] (v4) at (-0.4,-0.8) {}
            edge [->] (v0);
        \node[] at (0,-0.6) {$\scriptstyle\cdots$};
        \node[invisible] (v6) at (0.4,-0.8) {}
            edge [->] (v0);
    \end{tikzpicture}
    \ \ \Bigg) &&
    \end{flalign*}
\begin{flalign*}
    d_x\Big(\ 
    \begin{tikzpicture}[shorten >=1pt,>=angle 90,baseline={([yshift=-.5ex]current bounding box.center)}]
        \node[biw] (v0) at (0,0) {$\infty_1$\nodepart{lower}$0$};
        \node[invisible] (v1) at (-0.4,0.8) {}
            edge [<-] (v0);
        \node[] at (0,0.6) {$\scriptstyle\cdots$};
        \node[invisible] (v3) at (0.4,0.8) {}
            edge [<-] (v0);
        \node[invisible] (v4) at (-0.4,-0.8) {}
            edge [->] (v0);
        \node[] at (0,-0.6) {$\scriptstyle\cdots$};
        \node[invisible] (v6) at (0.4,-0.8) {}
            edge [->] (v0);
    \end{tikzpicture}\ \Big) \
    &=\Big(\frac{\infty_1}{0},\frac{\infty_1}{0}\Big)_x
    +\Big(\frac{0}{0},\frac{\infty_1}{0}\Big)_x
    +\Big(\frac{\infty_1}{0},\frac{0}{0}\Big)_x\\
    & - \Bigg( \ \ 
    \begin{tikzpicture}[shorten >=1pt,>=angle 90,baseline={(4ex,-0.5ex)}]
        \node[biw] (v0) at (0,0) {$\infty_1$\nodepart{lower}$0$};
        \node[biw] (new) at (1.5,0.7) {$\infty_1$\nodepart{lower}$\infty_1$}
            edge [<-] (v0);
        \node[invisible] (v1) at (-0.4,0.8) {}
            edge [<-] (v0);
        \node[] at (0,0.6) {$\scriptstyle\cdots$};
        \node[invisible] (v3) at (0.4,0.8) {}
            edge [<-] (v0);
        \node[invisible] (v4) at (-0.4,-0.8) {}
            edge [->] (v0);
        \node[] at (0,-0.6) {$\scriptstyle\cdots$};
        \node[invisible] (v6) at (0.4,-0.8) {}
            edge [->] (v0);
    \end{tikzpicture}
    \ +\ 
    \begin{tikzpicture}[shorten >=1pt,>=angle 90,baseline={(4ex,-0.5ex)}]
        \node[biw] (v0) at (0,0) {$0$\nodepart{lower}$0$};
        \node[biw] (new) at (1.5,0.7) {$\infty_1$\nodepart{lower}$\infty_1$}
            edge [<-] (v0);
        \node[invisible] (v1) at (-0.4,0.8) {}
            edge [<-] (v0);
        \node[] at (0,0.6) {$\scriptstyle\cdots$};
        \node[invisible] (v3) at (0.4,0.8) {}
            edge [<-] (v0);
        \node[invisible] (v4) at (-0.4,-0.8) {}
            edge [->] (v0);
        \node[] at (0,-0.6) {$\scriptstyle\cdots$};
        \node[invisible] (v6) at (0.4,-0.8) {}
            edge [->] (v0);
    \end{tikzpicture}
    \ \ \Bigg) &&
    \\
    \\
    d_x\Big(\ 
    \begin{tikzpicture}[shorten >=1pt,>=angle 90,baseline={([yshift=-.5ex]current bounding box.center)}]
        \node[biw] (v0) at (0,0) {$0$\nodepart{lower}$0$};
        \node[invisible] (v1) at (-0.4,0.8) {}
            edge [<-] (v0);
        \node[] at (0,0.6) {$\scriptstyle\cdots$};
        \node[invisible] (v3) at (0.4,0.8) {}
            edge [<-] (v0);
        \node[invisible] (v4) at (-0.4,-0.8) {}
            edge [->] (v0);
        \node[] at (0,-0.6) {$\scriptstyle\cdots$};
        \node[invisible] (v6) at (0.4,-0.8) {}
            edge [->] (v0);
    \end{tikzpicture}\ \Big) \
    &=\Big(\frac{0}{0},\frac{0}{0}\Big)_x &&
\end{flalign*}
\end{proposition}
The first page $S^{out}_0\wGC_k$ decomposes into a directed sum parameterized by the set of all possible core graphs
\begin{equation*}
    S^{out}_0\wGC_k\cong\bigoplus_{\gamma}\mathsf{outCore}(\gamma)
\end{equation*}
where $\mathsf{outCore}(\gamma)$ is the subcomplex of $S^{out}_0\wGC_k$ of graphs with associated core graph $\gamma$. Further $\mathsf{outCore}(\gamma)$ decomposes into a tensor product of complexes
\begin{equation*}
    \mathsf{outCore}(\gamma)\cong\Big(\bigotimes_{x\in V(\gamma)}\mathcal{T}^{out}_x\Big)^{\mathrm{Aut}(\gamma)}
\end{equation*}
with one complex $\mathcal{T}^{out}_x$ for each vertex $x$ in $V(\gamma)$ and $\mathrm{Aut}(\gamma)$ is the group of automorphisms of $\gamma$ acting on the tensor product. Each complex $\mathcal{T}^{out}_x$ consists of trees of special out-vertices attached to a core vertex $x$ and the differential acts by only creating special out-vertices on $x$ and in the tree. The complex $\mathcal{T}^{out}_x$ depends on $x$ only via the number of outgoing and incoming edges attached to $x$ in the core graph as well as on the in-weight $w^{in}_x$. Note that here the in-weight can be assigned one of the symbols $0$ or $\infty_1$. The complexes $\mathcal{T}^{out}_x$ having the same values of the parameters $|x|_{out}$, $|x|_{in}$ and $w^{in}_x$ are isomorphic, and we often write $\mathcal{T}^{out}_{x}\cong \mathcal{T}^{out}_{|x|_{out},|x|_{in},w^{in}_x}$. Hence we study the family of complexes $\mathcal{T}^{out}_{a,b,c}$ parameterized by integers $a,b\geq0$ and $c\in\{0,\infty_1\}$ such that $a\geq 0,\ b\geq 1$ and $a+b\geq 2$ when $c=0$, and $a,b\geq 0$ and $a+b\geq 1$ when $c=\infty_1$. The out-weight of the core vertex in $\mathcal{T}^{out}_{a,b,c}$ is any number $w^{out}_x$ satisfying the conditions
\begin{align*}
    w^{out}_x+a+\text{\#(out-edges from special out-vertices)}\geq 1\\
    a+b+w^{out}_x+|c|+\text{\#(out-edges from special out-vertices)}\geq 3
\end{align*} where $|0|:=0$ and $|\infty_1|:=1$.
Proposition \ref{prop:specialOut} follows from these remarks together with the following lemma.
\begin{lemma}\label{lemma:outtreecohom}
    The cohomology of $\mathcal{T}^{out}_{a,b,c}$ is generated by one or two classes containing the out-core vertex decorated by some bi-weights depending on the parameters $a,b,c$. More precisely
    \begin{align*}
        H(\mathcal{T}^{out}_{1,0,\infty_1})&=
            \Big\langle
                \tikz[baseline=-2.07cm]
                \node[biw] at (0,-2) {$\infty_1$\nodepart{lower}$\infty_1$};
            \Big\rangle\\
        H(\mathcal{T}^{out}_{0,1,\infty_1})&=
            \Big\langle
                \tikz[baseline=-2.07cm]
                \node[biw] at (0,-2) {$\infty_1$\nodepart{lower}$\infty_1$};
            \Big\rangle \\
        H(\mathcal{T}^{out}_{1,1,\infty_1})&=
            \Big\langle
                \tikz[baseline=-2.07cm]
                \node[biw] at (0,-2) {$\infty_1$\nodepart{lower}$\infty_1$};
                \, ,\
                \tikz[baseline=-2.07cm]
                \node[biw] at (0,-2) {$0$\nodepart{lower}$\infty_1$};
            \Big\rangle
            \ ,&&&
            H(\mathcal{T}^{out}_{1,1,0})&=
            \Big\langle
                \tikz[baseline=-2.07cm]
                \node[biw] at (0,-2) {$\infty_1$\nodepart{lower}$0$};
            \Big\rangle\\
        \text{For }a\geq2\quad H(\mathcal{T}^{out}_{a,0,\infty_1})&=
            \Big\langle
                \tikz[baseline=-2.07cm]
                \node[biw] at (0,-2) {$\infty_1$\nodepart{lower}$\infty_1$};
                \, ,\
                \tikz[baseline=-2.07cm]
                \node[biw] at (0,-2) {$0$\nodepart{lower}$\infty_1$};
            \Big\rangle\\
        \text{For }b\geq2\quad H(\mathcal{T}^{out}_{0,b,\infty_1})&=
            \Big\langle
                \tikz[baseline=-2.07cm]
                \node[biw] at (0,-2) {$\infty_1$\nodepart{lower}$\infty_1$};
            \Big\rangle
            \ ,&&&
        H(\mathcal{T}^{out}_{0,b,0})&=
            \Big\langle
                \tikz[baseline=-2.07cm]
                \node[biw] at (0,-2) {$\infty_1$\nodepart{lower}$0$};
            \Big\rangle\\
        \text{For }a,b\geq 1\text{ and }a+b\geq 3\quad H(\mathcal{T}^{out}_{a,b,\infty_1})&=
            \Big\langle
                \tikz[baseline=-2.07cm]
                \node[biw] at (0,-2) {$\infty_1$\nodepart{lower}$\infty_1$};
                \, ,\
                \tikz[baseline=-2.07cm]
                \node[biw] at (0,-2) {$0$\nodepart{lower}$\infty_1$};
            \Big\rangle
            \ ,&&&
        H(\mathcal{T}^{out}_{a,b,0})&=
            \Big\langle
                \tikz[baseline=-2.07cm]
                \node[biw] at (0,-2) {$\infty_1$\nodepart{lower}$0$};
                \, ,\
                \tikz[baseline=-2.07cm]
                \node[biw] at (0,-2) {$0$\nodepart{lower}$0$};
            \Big\rangle
    \end{align*}
\end{lemma}
\begin{proof}
The proof is similar to the proof for proposition \ref{prop:specialOut}, using the same filtrations and decompositions together with corollary \ref{cor:Helpcomplex}.
\end{proof}

Now we turn to the complexes $S^{in}_1\wGC_k^*$ and $S^{in}_1\wGC_k^+$. Let $\{S^{out}_r\wGC^*_k\}_{r\geq 0}$ and $\{S^{out}_r\wGC^+_k\}_{r\geq 0}$ be the spectral sequences associated to the filtrations over the number of out-core vertices.
\begin{proposition}\label{prop:specialOut+}
    The page one complex $S^{out}_1\wGC^*_k$ is a subcomplex of $S^{out}_1\wGC_k$ generated by directed graphs whose vertices are independently decorated by the bi-weights
$\begin{tikzpicture}[baseline={([yshift=-.5ex]current bounding box.center)}]
    \node[biw] at (0,0) {$\infty_1$\nodepart{lower}$\infty_1$};
\end{tikzpicture}$, 
$\begin{tikzpicture}[baseline={([yshift=-.5ex]current bounding box.center)}]
    \node[biw] at (0,0) {$0$\nodepart{lower}$\infty_1$};
\end{tikzpicture}$,
$\begin{tikzpicture}[baseline={([yshift=-.5ex]current bounding box.center)}]
    \node[biw] at (0,0) {$\infty_1$\nodepart{lower}$0$};
\end{tikzpicture}$, and
$\begin{tikzpicture}[baseline={([yshift=-.5ex]current bounding box.center)}]
    \node[biw] at (0,0) {$0$\nodepart{lower}$0$};
\end{tikzpicture}$ subject to the conditions of proposition \ref{prop:specialOut} as well as the additional condition that at least one vertex is not decorated by $\begin{tikzpicture}[baseline={([yshift=-.5ex]current bounding box.center)}]
    \node[biw] at (0,0) {$0$\nodepart{lower}$0$};
    \end{tikzpicture}$.
Furthermore, the page one complex $S^{out}_1\wGC^+_k$ is a subcomplex of $S^{out}_1\wGC^*_k$ where each graph additionally satisfies that either
\begin{itemize}
    \item at least one vertex is decorated with the bi-weight
    $\begin{tikzpicture}[baseline={([yshift=-.5ex]current bounding box.center)}]
    \node[biw] at (0,0) {$\infty_1$\nodepart{lower}$\infty_1$};
    \end{tikzpicture}$
    , or
    \item at least two vertices are decorated with the bi-weights
    $\begin{tikzpicture}[baseline={([yshift=-.5ex]current bounding box.center)}]
    \node[biw] at (0,0) {$0$\nodepart{lower}$\infty_1$};
    \end{tikzpicture}$
     and 
     $\begin{tikzpicture}[baseline={([yshift=-.5ex]current bounding box.center)}]
    \node[biw] at (0,0) {$\infty_1$\nodepart{lower}$0$};
    \end{tikzpicture}$ respectively.
\end{itemize}
\end{proposition}
\begin{proof}
    The proof is analogous to that of proposition \ref{prop:OutIn1}.
\end{proof}

\subsection{The 4-type graph complex $\qGC_k$}
\begin{definition}
    Let $\qGC_k$ be the subcomplex of $\wGC_k$ consisting of graphs whose vertices can independently be decorated by four types of decorations $\frac{\infty_1}{\infty_1},\ \frac{\infty_1}{0},\ \frac{0}{\infty_1}$ and $\frac{0}{0}$. The possible decorations of a vertex depend on its type. More concretely
    \begin{itemize}
        \item A univalent vertex can only be decorated by $\frac{\infty_1}{\infty_1}$.
        \item A source vertex can be decorated by $\frac{\infty_1}{\infty_1}$ and $\frac{0}{\infty_1}$.
        \item A target vertex can be decorated by $\frac{\infty_1}{\infty_1}$ and $\frac{\infty_1}{0}$.
        \item A passing vertex can be decorated by $\frac{\infty_1}{\infty_1}$, $\frac{\infty_1}{0}$ and $\frac{0}{\infty_1}$.
        \item A generic vertex can be decorated by $\frac{\infty_1}{\infty_1}$, $\frac{\infty_1}{0}$, $\frac{0}{\infty_1}$ and $\frac{0}{0}$.
    \end{itemize}
    Let $\qGC_k^*$ and $\qGC_k^+$ be the subcomplexes of $\qGC_k$ where
    \begin{itemize}
        \item $\qGC^*_k$ is generated by graphs where at least one vertex is not decorated by $\frac{0}{0}$.
        \item $\qGC^+_k$ is generated by graphs with at least one vertex decorated by $\frac{\infty_1}{\infty_1}$ or two vertices decorated by $\frac{0}{\infty_1}$ and $\frac{\infty_1}{0}$ respectively.
    \end{itemize}
\end{definition}
The differential act on vertices in the same way as the differential in proposition \ref{prop:specialOut}. Note that $\qGC^*_k\subset\wGC^*_k$ and $\qGC^+_k\subset\wGC^+_k$.
\begin{remark}
    We will use the following convention: when specifying the decorations of a general graph, the univalent vertices are excluded from this specification. For example, "A graph where all vertices are decorated by $\frac{\infty_1}{0}$" should be interpreted as a graph where all non-univalent vertices are decorated by $\frac{\infty_1}{0}$, and the univalent vertices decorated by their only possible decoration $\frac{\infty_1}{\infty_1}$.
\end{remark}
\begin{proposition}\label{prop:4-types}
    The three inclusions $\qGC_k\hookrightarrow\wGC_k $, $\qGC^*_k\hookrightarrow\wGC^*_k $ and $\qGC^+_k\hookrightarrow\wGC^+_k$
    are quasi-isomorphims.
\end{proposition}
\begin{proof}
    Consider a filtration over special in-vertices as seen in the previous section of both $\wGC_k$ and $\qGC_k$. The filtration is preserved by the inclusions since no graph in $\qGC_k$ contain special in-vertices, and furthermore the differential is trivial in the associated complex of $\qGC_k$. On the second page, we consider a second filtration over special out-vertices. The two complexes agree on the second page of the associated spectral sequence by proposition \ref{prop:specialOut}. This proves that the inclusion is a quasi-isomorphism. The proof is analogous for the other two inclusions with the help of proposition \ref{prop:specialOut+}
\end{proof}
\begin{remark}
    We notice the splitting of $\qGC_k=\qGC_k^0\oplus\qGC_k^*$ where $\qGC_k^0$ is the complex of graphs where all vertices are decorated by $\frac{0}{0}$ and $\qGC_k^*$ is the complex of graphs with at least one vertex not decorated by $\frac{0}{0}$. This is analogous to the splitting $\fwGC_k=\fWGC_k^0\oplus\fwGC_k^*$ in section \ref{sec:BiWCohom}, and it is immediate by proposition \ref{prop:0complex} that $\qGC^0_k\cong \dGC_k^\circlearrowleft$. In the remainder of the paper we focus on studying $\qGC_k^*$ and $\qGC_k^+$.
\end{remark}

\section{Reducing $\qGC_k^*$ to the mono-decorated graph complex $\fM_k^*$}

\subsection{Removing $\frac{0}{0}$ decorations from $\qGC_k^*$}
    Let $\qGC^{*,0}_k\subset\qGC^*_k$ be the subcomplex generated by graphs with at least one vertex decorated by $\frac{0}{0}$. Consider the short exact sequence
    \begin{equation*}
        0\rightarrow \qGC^{*,0}_k \hookrightarrow \qGC^{*}_k \rightarrow \tGC^{*}_k \rightarrow 0
    \end{equation*}
    where the quotient complex $\tGC_k^*$ is generated by graphs with no decoration $\frac{0}{0}$. Similarly let $\qGC^{+,0}\subset\qGC^+_k$ be the complex spanned by graphs with at least one decoration $\frac{0}{0}$ and $\tGC^{+}_k:=\qGC^{+}_k/\qGC^{+,0}_k$ the quotient complex of graphs with at least one vertex decorated by $\frac{\infty_1}{\infty_1}$, or a pair of vertices decorated by $\frac{0}{\infty_1}$ and $\frac{\infty_1}{0}$, with no vertices decorated by $\frac{0}{0}$.
\begin{proposition}\label{prop:0-decorations}
     The projections $\qGC^{*}_k\rightarrow\tGC_k^*$ and $\qGC^{+}_k\rightarrow\tGC_k^+$ are quasi-isomorphisms.
\end{proposition}
\begin{proof}
    It is enough to show that $\qGC^{*,0}_k$ and $\qGC^{+,0}_k$ are acyclic.
    Consider a filtration of $\qGC^{*,0}_k$ over the number of non-passing vertices and let $P_0\qGC^{*,0}_k$ be the first page of the associated spectral sequence. The differential act by only creating passing vertices. Consider a filtration of $P_0\qGC^{*,0}_k$ over the number of vertices \textit{not} decorated by $\frac{\infty_1}{\infty_1}$. The differential on the first page of the associated spectral sequence $D_0P_0\qGC^{*,0}_k$ only creates passing vertices decorated by $\frac{\infty_1}{\infty_1}$.
    To each graph in this complex, we can associated a $\frac{\infty_1}{\infty_1}$\textit{-skeleton graph} by removing passing vertices decorated by $\frac{\infty_1}{\infty_1}$ and replace them by a single edge. The $\frac{\infty_1}{\infty_1}$-skeleton graph is invariant under the action of the differential, hence the complex split as
    \begin{equation*}
        gr(P_0\qGC^{*,0}_k)=\bigoplus_{\gamma}\mathcal{C}^*(\gamma)
    \end{equation*}
    where $\mathcal{C}^*(\gamma)$ is the associated complex of graphs with $\frac{\infty_1}{\infty_1}$-skeleton $\gamma$. Note that no graph has the associated $\frac{\infty_1}{\infty_1}$-skeleton graph with one single vertex and one edge, since any such graph has no vertex decorated by $\frac{0}{0}$. We claim the following:
    \begin{itemize}
        \item If $\gamma$ has at least one vertex decorated by $\frac{\infty_1}{\infty_1}$, then $H(\mathcal{C}^*(\gamma))=0$.
        \item If $\gamma$ has no vertices decorated by $\frac{\infty_1}{\infty_1}$ and at least one decorated by $\frac{\infty_1}{0}$ or $\frac{0}{\infty_1}$, then $H(\mathcal{C}^*(\gamma))=\langle\gamma\rangle$.
        \item If $\gamma$ is only decorated by $\frac{0}{0}$ and has at least one univalent vertex, then $H(\mathcal{C}^*(\gamma))=\langle\gamma\rangle$.
        \item If $\gamma$ is only decorated by $\frac{0}{0}$, then $H(\mathcal{C}^*(\gamma))=0$.
    \end{itemize}
    In the three first cases, the complex can be written as
    \begin{equation}\label{eq:ZeroTensor}
        \mathcal{C}^*(\gamma)\cong\Big(\bigotimes_{e\in E(\gamma)}\mathcal{E}_e\Big)^{\mathrm{Aut}(\gamma)}
    \end{equation}
    where $\mathcal{E}_e$ is the associated complex of passing vertices decorated by $\frac{\infty_1}{\infty_1}$ on the edge $e$ in the skeleton. In the first case, there is one edge $e'$ in $\gamma$ such that one of its adjacent vertices is decorated by $\frac{\infty_1}{\infty_1}$ and the other by another decoration. One computes that $\mathcal{E}_{e'}$ is acyclic.
    In the second and third case, all $\mathcal{E}_e$ are isomorphic, consisting of the complex of passing vertices decorated by $\frac{\infty_1}{\infty_1}$. One sees that $H(\mathcal{E}_e)$ is generated by the graph of two vertices and an edge, with no passing vertices, giving the desired result.
    In the fourth case, the complex does not split as a tensor, but a direct sum of such tensors where in each one, at least one complex $\mathcal{E}_e$ have at least one passing vertex decorated by $\frac{\infty_1}{\infty_1}$. This complex is equivalent to the tensor above when removing the initial graph with no passing vertices, which we saw was a cycle in the second case. Hence the complex is acyclic.
    The cohomology of $D_0P_0\qGC^{*,0}$ now consists of graphs only decorated by $\frac{\infty_1}{0}$, $\frac{0}{\infty_1}$ and $\frac{0}{0}$ with at least one vertex decorated by $\frac{0}{0}$. The differential act by creating passing vertices.
    It is easy to see that these graphs form a subcomplex of $P_0\qGC^{*,0}_k$, and so the second page $D_1P_0\qGC^{*,0}$ of the spectral sequence is described with the full differential. We claim that $D_1P_0\qGC^{*,0}$ is acyclic. By considering a similar filtration over the number of vertices not decorated by $\frac{\infty_1}{0}$, one finds that this page of the spectral sequence is acyclic, finishing the proof.
    The proof to show that $\qGC^{+,0}_k$ is acyclic follows the same argument above using the same filtrations.
\end{proof}
\subsection{Subcomplex of monodecorated graphs}
Let $\Gamma$ be an undecorated directed graph. We define $\Gamma(\frac{0}{\infty_1})$ to be the graph $\Gamma$ where all non-univalent vertices are decorated by $\frac{0}{\infty_1}$ and the univalent vertices are decorated by $\frac{\infty_1}{\infty_1}$. If $\Gamma$ contains a target vertex (which cannot be decorated by $\frac{0}{\infty_1}$), then we set $\Gamma(\frac{0}{\infty_1})=0$. Similarly define $\Gamma(\frac{\infty_1}{0})$.
Let $\Gamma$ be a directed graph without any univalent vertices and let $\tGC^+_k(\Gamma)\subset\tGC_k^+$ be the subset of decorated graphs whose underlying shape is $\Gamma$. Define $\Gamma(\omega)\in\tGC_k^+$ as the sum of graphs
\begin{equation*}
    \Gamma(\omega)=\sum_{\gamma\in\tGC^+_k(\Gamma)}\gamma.
\end{equation*}
Consider the decomposition $d=d_s+d_u$ of the differential of $\tGC_k^+$ where $d_u$ is the part of where a new univalent vertex is created and $d_s$ the part where no new univalent vertices are created (also known as splitting). We see that $d_s\Gamma(\omega)=(d_s\Gamma)(\omega)$. Further we see that $d_u\Gamma(\omega)=-(d_u\Gamma(\frac{\infty_1}{0})+d_u\Gamma(\frac{0}{\infty_1}))$. Due to $\Gamma(\frac{\infty_1}{0})$ and $\Gamma(\frac{0}{\infty_1})$ being zero for some graphs, we have more specifically that
\begin{itemize}
    \item If $\Gamma$ contain both at least one source and one target, then $d_u\Gamma(\omega)=0$.
    \item If $\Gamma$ contain at least one source but no targets, then $d_u\Gamma(\omega)=-d_u\Gamma(\frac{0}{\infty_1})$.
    \item If $\Gamma$ contain at least one target but no sources, then $d_u\Gamma(\omega)=-d_u\Gamma(\frac{\infty_1}{0})$.
    \item If $\Gamma$ contain neither sources nor targets, then $d_u\Gamma(\omega)=-(d_u\Gamma(\frac{0}{\infty_1})+d_u\Gamma(\frac{\infty_1}{0}))$.
\end{itemize}
Define $\mGC_k^+\subset\tGC_k^+$ to be the subcomplex generated by graphs on three forms:
\begin{enumerate}
    \item Graphs on the form $\Gamma(\omega)$ with $\Gamma$ a directed graph with no univalent vertices.
    \item Graphs on the form $\Gamma(\frac{0}{\infty_1})$ with $\Gamma$ a directed graph with at least one univalent source, and no targets of any valency.
    \item Graphs on the form $\Gamma(\frac{\infty_1}{0})$ with $\Gamma$ a directed graph with at least one univalent target, and no sources of any valency.
\end{enumerate}
Similarly we define $\mGC_k^*\subset\tGC_k^*$ to be the subcomplex generated by graphs on three forms:
\begin{enumerate}
    \item Graphs on the form $\Gamma(\omega)$ with $\Gamma$ a directed graph with no univalent vertices.
    \item Graphs on the form $\Gamma(\frac{0}{\infty_1})$ with $\Gamma$ a directed graph containing no targets of any valency.
    \item Graphs on the form $\Gamma(\frac{\infty_1}{0})$ with $\Gamma$ a directed graph containing no sources of any valency.
\end{enumerate}
\begin{proposition}\label{prop:mGC}
    The inclusions $\mGC_k^+\rightarrow\tGC_k^+$ and $\mGC_k^*\rightarrow\tGC_k^*$ are quasi-isomorphisms.
\end{proposition}
Consider the short exact sequences
\begin{align*}
    0 \rightarrow \mGC_k^+\rightarrow \tGC^+_k\rightarrow \mathsf{Q}^+_k \rightarrow 0\\
    0 \rightarrow \mGC_k^*\rightarrow \tGC^*_k\rightarrow \mathsf{Q}^*_k \rightarrow 0
\end{align*}
where $\mathsf{Q}^*_k=\tGC^*_k/\mGC_k^*$ and $\mathsf{Q}^+_k=\tGC^+_k/\mGC_k^+$. We observe that $\mathsf{Q}^*_k=\mathsf{Q}^+_k$. The complex $\mathsf{Q}^+_k$ is generated by graphs $\Gamma\in\tGC_k^+$ on the forms
\begin{enumerate}
    \item $\Gamma$ has no univalent vertices and has at least one vertex not decorated by $\frac{\infty_1}{\infty_1}$.
    \item $\Gamma=\Gamma(\frac{\infty_1}{0})$ and has at least one univalent target.
    \item $\Gamma=\Gamma(\frac{0}{\infty_1})$ and has at least one univalent source.
    \item $\Gamma$ has at least one univalent vertex and at least two non-univalent vertices with different decorations.
\end{enumerate}
The proposition follows if we show that $\mathsf{Q}_k^+$ is acyclic. Consider the subcomplex $\mathsf{Q}_k^{+,1}\subset\mathsf{Q}_k^+$ of graphs with at least one univalent vertex. We get the induced short exact sequence
\begin{equation*}
    0 \rightarrow \mathsf{Q}^{+,1}_k \rightarrow \mathsf{Q}^{+}_k \rightarrow \mathsf{Q}^{+,\geq 2}_k \rightarrow 0.
\end{equation*}
The complex $\mathsf{Q}^{+,\geq 2}_k$ is spanned by graphs on the form (1) as above and $\mathsf{Q}^{+,1}_k$ is spanned by graphs on the form (2)-(4). The acyclicity of $\mathsf{Q}^+_k$ follows from the following proposition.
\begin{proposition}
    The complexes $\mathsf{Q}^{+,1}_k$ and $\mathsf{Q}^{+,\geq2}_k$ are acyclic.
\end{proposition}
\begin{proof}
    First consider a filtration of $\mathsf{Q}^{+,\geq2}_k$ over the number of non-passing vertices and let $P_0\mathsf{Q}^{+,\geq2}_k$ be the first page of the spectral sequence. On this page the differential acts by only creating passing vertices. We will show that this page is acyclic. Consider a filtration on $P_0\mathsf{Q}^{+,\geq2}_k$ over the number of vertices not decorated by $\frac{\infty_1}{\infty_1}$ and let $D_0P_0\mathsf{Q}^{+,\geq2}_k$ be the first page of the associated spectral sequence. Here the differential act by only creating passing vertices decorated by $\frac{\infty_1}{\infty_1}$. Similar to proposition \ref{prop:0-decorations}, the complex decompose over $\frac{\infty_1}{\infty_1}$-skeleton graphs as
    \begin{equation*}
        D_0P_0\mathsf{Q}^{+,\geq2}_k=\bigoplus_\gamma\mathcal{C}^+(\gamma)
    \end{equation*}
    where $\mathcal{C}^+(\gamma)$ is the complex of graphs with $\frac{\infty_1}{\infty_1}$-skeleton $\gamma$. We claim the following
    \begin{itemize}
        \item If $\gamma$ has at least one vertex decorated by $\frac{\infty_1}{\infty_1}$, then $\mathcal{C}^+(\gamma)\simeq 0$.
        \item If $\gamma$ is only decorated by $\frac{\infty_1}{0}$, then $\mathcal{C}^+(\gamma)\simeq0$.
        \item If $\gamma$ is only decorated by $\frac{0}{\infty_1}$, then $\mathcal{C}^+(\gamma)\simeq0$.
        \item If $\gamma$ is decorated by both $\frac{\infty_1}{0}$ and $\frac{0}{\infty}$ (but not $\frac{\infty_1}{\infty_1}$), then $H(\mathcal{C}^+(\gamma))=\langle\gamma\rangle$.
    \end{itemize}
    These result follows by using similar arguments as in proposition $\ref{prop:0-decorations}$.  Let $D_1P_0\mathsf{Q}_k^{+,\geq2}$ be the second page of the spectral sequence. By the result above, it consists of graphs with at least a pair of vertices decorated by $\frac{\infty_1}{0}$ and $\frac{0}{\infty_1}$. Consider a filtration over the number of vertices not decorated by $\frac{\infty_1}{0}$. Using similar arguments as above on the associated spectral sequence, we get that $D_1P_0\mathsf{Q}_k^{+,\geq2}$ is acyclic.
    The proof showing that $\mathsf{Q}^{+,1}_k$ is acyclic is similar, noting that the differential acting on univalent sources only create passing vertices decorated by $\frac{0}{\infty_1}$ and univalent targets only create passing vertices decorated by $\frac{\infty_1}{0}$.
\end{proof}
\subsection{Removing long antennas}
Consider the two subcomplexes $\mGC_k^{*}(\frac{\infty_1}{0})$ and $\mGC_k^{*}(\frac{0}{\infty_1})$ of $\mGC_k^*$ of graphs whose non-univalent vertices are decorated by $\frac{\infty_1}{0}$ and $\frac{0}{\infty_1}$ respectively.
\begin{proposition}\label{prop:mGC*(0)}
    The complexes $\mGC_k^{*}(\frac{\infty_1}{0})$ and $\mGC_k^{*}(\frac{0}{\infty_1})$ are acyclic.
\end{proposition}
\begin{proof}
    The two complexes are isomorphic by swapping decorations and reversing the orientation of all edges, hence it is enough to show that $\mGC_k^*(\frac{\infty_1}{0})$ is acyclic.
    Let the core of a graph be the graph remaining after iterative removal of univalent vertices. Any vertex in the core graph is called a core-vertex. Vertices which are not core vertices are called antenna vertices. The number of core vertices can not decrease under the action of the differential. Consider a filtration of $\mGC_k^{*}(\frac{\infty_1}{0})$ over the number of core-vertices in a graph. Let $gr( \mGC_k^{*}(\frac{\infty_1}{0}))$ be the associated graded complex. The differential acts by creating antenna vertices, but the core graph is invariant. Hence we get the decomposition
    \begin{equation*}
        gr(\mGC_k^{*}(\frac{\infty_1}{0}))=\bigoplus_\gamma\mathsf{Core}(\gamma)
    \end{equation*}
    where the summation is over all possible core graphs $\gamma$ and $\mathsf{Core}(\gamma)$ is the complex of graphs with core graph $\gamma$. This complex further decompose as
    \begin{equation*}
        \mathsf{Core}(\gamma)\cong\Big(\bigotimes_{x\in V(\gamma)}\mathcal{T}_x\Big)^{\mathrm{Aut}(\gamma)}
    \end{equation*}
    where $\mathcal{T}_x$ is the associated complex of antenna-vertices attached to the vertex $x$. The complex $\mathcal{T}_x$ is composed of directed trees with a designated core vertex $x$. All edges in any such graph are directed away from $x$. The proof follows if we show that $\mathcal{T}_x$ is acyclic. Consider a filtration on $\mathcal{T}_x$ over the number of univalent vertices (considering the core vertex $x$ to be univalent only when it is the only vertex in the graph). Let $gr\ \mathcal{T}_x$ be the associated graded complex. The differential acts by splitting vertices such that no new univalent vertices are created (except in the case of the one vertex graph). Hence it decomposes as $gr\ \mathcal{T}_x=\bigoplus_{N\geq 1}u_N\mathcal{T}_x$ where $u_N\mathcal{T}_x$ is the complex of graphs with $N$ univalent vertices. It is easy to see that $u_1\mathcal{T}_x$ is acyclic. Consider $u_N\mathcal{T}_x$ for $N\geq2$. Similar to the proof of lemma \ref{lemma:uNTin}, call a vertex $y$ of $\Gamma\in u_N\mathcal{T}_x$ a \textit{branch vertex} if there are at least two outgoing edges from $y$, or if there is a directed path from $y$ to a vertex $z$ which has at least two outgoing edges from it. The number of such vertices can not decrease under the differential. Consider a filtration over the number of branch vertices on $u_N\mathcal{T}_x$. Following the same arguments as in the proof of lemma \ref{lemma:uNTin}, we get that $u_N\mathcal{T}_x$ is acyclic, finishing the proof.
\end{proof}
We have the following decomposition
\begin{align*}
    \mGC_k^* &= \fM_k^* \oplus \fM_k^{a\geq2}\\
    \mGC_k^+ &= \fM_k^+ \oplus \fM_k^{a\geq2}
\end{align*}
where $\fM_k^{a\geq2}$ is the subcomplex of graphs containing at least one antenna with two or more vertices.
\begin{proposition}
    The injections $\fM_k^*\hookrightarrow\mGC_k^*$ and $\fM_k^+\hookrightarrow\mGC_k^+$ are quasi-isomorphisms.
\end{proposition}
\begin{proof}
    The statement follows if we show that $\fM_k^{a\geq2}$ is acyclic.
    Note that $\fM_k^{a\geq2}$ decompose over graphs whose non-univalent vertices are decorated by $\frac{\infty_1}{0}$ and $\frac{0}{\infty_1}$ respectively, i.e $\fM_k^{a\geq2}=\fM_k^{a\geq2}(\frac{\infty_1}{0})\oplus\fM_k^{a\geq2}(\frac{0}{\infty_1})$.
    Further note that $\mGC_k^*(\frac{\infty_1}{0})=\fM_k^{*,a=0,1}(\frac{\infty_1}{0})\oplus\fM_k^{a\geq2}(\frac{\infty_1}{0})$. Hence acyclicity of $\mGC_k^*(\frac{\infty_1}{0})$ from proposition \ref{prop:mGC*(0)} gives that $\fM_k^{a\geq2}(\frac{\infty_1}{0})$ is also acyclic.
\end{proof}'
\section{Cohomology of $\fM_k^*$ and $\fM_k^+$}
\subsection{A commutative diagram of graph complexes}
Let $\fM_k^{+,1}$ be the subcomplex of $\fM_k^+$ of graphs with at least one univalent vertex.
Further let $\fM_k^*(\frac{\infty_1}{0})$ and $\fM_k^*(\frac{0}{\infty_1})$ be the subcomplexes of $\fM_k^*$ of graphs whose non-univalent vertices are decorated by $\frac{\infty_1}{0}$ and $\frac{0}{\infty_1}$ respectively. Note that $\fM_k^{+,1}$ is a subcomplex of $\fM_k^*(\frac{\infty_1}{0})\oplus\fM_k^*(\frac{0}{\infty_1})$ and can be rewritten as $\fM_k^{+,1}=\fM_k^{*,1}(\frac{\infty_1}{0})\oplus\fM_k^{*,1}(\frac{0}{\infty_1})$ where $\fM_k^{*,1}(\frac{\infty_1}{0})$ is the subcomplex of $\fM_k^*$ of graphs with at least one univalent vertex and where all non-univalent vertices are decorated by $\frac{\infty_1}{0}$, and similarly for $\fM_k^{*,1}(\frac{0}{\infty_1})$. We get the following commutative diagram
\begin{equation}\label{eq:+*-Diagram}
    \xymatrix{
    & 0 \ar[d] & 0 \ar[d] & 0 \ar[d] & \\
    0 \ar[r] & {\fM^{+,1}_k} \ar[r] \ar[d] & \fM^{+}_k \ar[r] \ar[d] & {\fM^{+,\geq2}_k} \ar[d]^{id} \ar[r] & 0  \\
    0 \ar[r] & \fM^{*}_k(\frac{\infty_1}{0})\oplus\fM^{*}_k(\frac{0}{\infty_1}) \ar[d] \ar[r] & \fM^{*}_k \ar[d] \ar[r] & {\fM^{+,\geq2}_k} \ar[r] \ar[d] & 0 \\
    0 \ar[r] & {\fM^{*,\geq2}_k(\frac{\infty_1}{0})\oplus\fM^{*,\geq2}_k(\frac{0}{\infty_1})} \ar[r]^{id} \ar[d] & {\fM^{*,\geq2}_k(\frac{\infty_1}{0})\oplus\fM^{*,\geq2}_k(\frac{0}{\infty_1})} \ar[r] \ar[d] &  0  \ar[r] \ar[d] & 0 \\
    & 0 & 0 & 0 &
    }
\end{equation}
where $\fM^{*,\geq2}_k(\frac{\infty_1}{0})$ is the quotient complex of graphs with no univalent vertices, and all vertices are decorated by $\frac{\infty_1}{0}$, and similarly for $\fM^{*,\geq2}_k(\frac{0}{\infty_1})$). Further $\fM_k^{+,\geq 2}$ is the quotient complex of graphs with no univalent vertices, and where each graph is on the form $\Gamma(\omega)$.
We already showed in proposition \ref{prop:mGC*(0)} that $\fM_k^*(\frac{\infty_1}{0})$ and $\fM_k^*(\frac{0}{\infty_1})$ are acyclic, and so we get the following corollaries.
\begin{corollary}\label{cor:mGCDiagram}
    \hspace{1cm}
    \begin{itemize}
        \item The projection $\fM_k^*\rightarrow\fM_k^{+,\geq2}$ is a quasi-isomorphism.
        \item The connecting morphisms $\delta:H^\bullet(\fM_k^{*,\geq2}(\frac{\infty_1}{0})\oplus\fM_k^{*,\geq2}(\frac{0}{\infty_1}))\rightarrow H^{\bullet+1}(\fM_k^{+,1})$ is an isomorphism.
    \end{itemize}
\end{corollary}

\subsection{The cohomology of $\fM_k^*$}
Let $\fM^*_k=\mathsf{b}_1\M_k^*\oplus\M_k^*$ where $\mathsf{b}_1\M_k^*$ is the subcomplex of graphs with loop number one, and $\M_k^*$ is the subcomplex of graphs with loop number two and higher. All complexes in the commutative diagram \ref{eq:+*-Diagram} above split in the same manner, and we use an analogous notation for their splittings.
\begin{proposition}\label{prop:M*cohom}
    \hspace{1cm}
    \begin{enumerate}
        \item The map $b:\dGC_k\rightarrow\M_k^{+,\geq2}$ where $b(\Gamma)=\Gamma(\omega)$ is a quasi-isomorphism.
        \item The map $b_1:\mathsf{g}_1\dGC_k\rightarrow \mathsf{b}_1\M_k^{+,\geq 2}$ where $b_1(\Gamma)=\Gamma(\omega)$ is an isomorphism. In particular
        \begin{equation*}
        H^\bullet(\mathsf{b}_{1}\M_k^*)\cong H^\bullet(\mathsf{b}_1\dGC_k)\cong \bigoplus_{\substack{i\geq 1\\ i\equiv 2k+1 \mod 4}}
        \mathbb{K}[k-i]
        \end{equation*}
        where the component $\mathbb{K}[k-i]$ denotes the space spanned by the sum of all possible graphs consisting of $i$ bivalent vertices attached in a loop.
    \end{enumerate}
\end{proposition}
\begin{proof}
    The proof showing that $b$ is an quasi-isomorphism follows by considering a filtration over the number of non-passing vertices. By inspection one sees that $b_1$ is in fact an isomorphism.
\end{proof}
This proposition together with corollary \ref{cor:mGCDiagram} gives us Theorem 1.
\begin{theorem}    $H^\bullet(Def^*(\mathcal{H}olieb_{c,d}^\circlearrowleft))=H^\bullet(\GC_{c+d+1})\oplus \Big(\bigoplus_{\substack{i\geq 1\\ i\equiv 2(c+d+1)+1 \mod 4}}
        \mathbb{K}[c+d+1-i]\Big)$.
\end{theorem}
\subsection{The cohomology of $\fM_k^+$}
\begin{proposition}\label{prop:dGC/dGCs}
    \hspace{1cm}
    \begin{enumerate}
        \item Let $f^s:\dGC_k/\dGC^{s}_k\rightarrow \M_k^{*,\geq2}(\frac{\infty_1}{0})$ be the map where $f^s(\Gamma)=\Gamma(\frac{\infty_1}{0})$, and let $f^t:\dGC_k/\dGC^{t}_k\rightarrow \M_k^{*,\geq2}(\frac{0}{\infty_1})$ be the map where $f^t(\Gamma)=\Gamma(\frac{0}{\infty_1})$. These maps are quasi-isomorphisms.
        \item The differentials of the complexes $\mathsf{b}_1\M_k^{*,\geq2}(\frac{\infty_1}{0})$ and $\mathsf{b}_1\M_k^{*,\geq2}(\frac{0}{\infty_1})$ are zero. In particular
        \begin{equation*}
            H^\bullet(\mathsf{b}_1\M_k^{*,\geq2}(\frac{\infty_1}{0}))\cong H^\bullet(\mathsf{b}_1\M_k^{*,\geq2}(\frac{0}{\infty_1}))\cong \bigoplus_{\substack{i\geq 1 \\ i \equiv 1 \mod 2}}\mathbb{K}[k-i]
        \end{equation*}
        where each component $\mathbb{K}[k-i]$ is generated by the loop graph with $i$ vertices that are all passing and decorated by $\frac{\infty_1}{0}$ and $\frac{0}{\infty_1}$ respectively.
    \end{enumerate}
\end{proposition}
\begin{proof}
    The first part follows by a filtration over the number of non-passing vertices. The latter part follows by noting that graphs with loop number one and no univalent vertices have either both a source and a target or neither. 
\end{proof}
We do already at this point get main theorem 3.
\begin{theorem}
    $H^0(Def^+(\mathcal{H}olieb_{1,1}^\circlearrowleft))=H^\bullet(\M_3^+)=\mathfrak{grt}\oplus\mathfrak{grt}$.
\end{theorem}
\begin{proof}
    Consider the short exact sequence
    \begin{equation*}
        \xymatrix{
        0 \ar[r] & \M_3^{+,\geq1} \ar[r] & \M_3^+ \ar[r] & \M_3^{+,\geq2} \ar[r] & 0}.
    \end{equation*}
    We extract the following exact sequence from the induced long exact sequence
    \begin{equation*}
        \xymatrix{
        H^{-1}(\mathsf{dGC}_3) \ar[r] & H^{-1}(\mathsf{dGC}_3/\mathsf{dGC}^s_3)\oplus H^{-1}(\mathsf{dGC}_3/\mathsf{dGC}^t_3) \ar[r] & H^0(\M_k^+) \ar[r] & H^0(\mathsf{dGC}_3)}
    \end{equation*}
    by using proposition \ref{prop:dGC/dGCs} and the quasi-isomorphisms induced from the diagram \ref{eq:+*-Diagram}.
    Now since $H^k(\dGC_3)=0$ for $k\geq -2$, we get the isomorphism $H^0(\mGC_3)\cong H^{-1}(\mathsf{dGC}_3/\mathsf{dGC}^s_3)\oplus H^{-1}(\mathsf{dGC}_3/\mathsf{dGC}^t_3)$.
    Using the same trick on the long exact sequence on cohomology induced by the short exact sequence
    \begin{equation*}
        0 \rightarrow \dGC_3^s\rightarrow \dGC_3\rightarrow \dGC_3/\dGC_3^s \rightarrow 0 
    \end{equation*}
    we get that $H^{-1}(\mathsf{dGC}_3/\mathsf{dGC}^s_3)\cong H^0(\dGC_3^s)$. Similarly $H^{-1}(\mathsf{dGC}_3/\mathsf{dGC}^t_3)\cong H^0(\dGC_3^t)$.
    It was shown in \cite{Z2} that there are quasi-isomorphisms $\dGC_k\rightarrow \dGC_{k+1}^s$ and $\dGC_k\rightarrow \dGC_{k+1}^t$, and so in particular $H^0(\dGC_3^s)=H^0(\dGC_3^t)=H^0(\dGC_2)=\mathfrak{grt}$, giving the theorem.
\end{proof}
From the same short exact sequence of complexes of loop number one graphs
\begin{equation*}
        \xymatrix{
        0 \ar[r] & \mathsf{b}_1\M_k^{+,1} \ar[r] & \mathsf{b}_1\M_k^+ \ar[r] & \mathsf{b}_1\M_k^{+,\geq2} \ar[r] & 0}.
\end{equation*}
we can derive the cohomology of $\mathsf{b}_1\M_k^{+}$.
\begin{proposition}
    The cohomology of $\mathsf{b}_1\M_k^+$ is given by
    \begin{equation*}
        H^\bullet(\mathsf{b}_1\M_k^+)=\bigoplus_{\substack{i\geq 1 \\ i \equiv 1 \mod 2}}\mathcal{C}_i, \qquad\text{where } \mathcal{C}_i=
        \begin{cases}
            \mathbb{K}[k-(i+1)] & \text{if }i\equiv 2k+1 \mod 4\\
            \mathbb{K}[k-(i+1)]\oplus\mathbb{K}[k-(i+1)] & \text{if }i\equiv 2k+3 \mod 4\\
        \end{cases}
    \end{equation*}
\end{proposition}
\begin{proof}
    We first show that the connecting morphism $\delta:H^i(\mathsf{b}_1\M_k^{+,\geq2})\rightarrow H^{i+1}(\mathsf{b}_1\M_k^{+,1})$ is injective.
    The injection $b_1:\mathsf{b}_1\GC_k\rightarrow\mathsf{b}_1\M_k^{+,\geq2}$ mapping an undirected graph $\Gamma$ to the sum of graphs over all possible ways of adding directions on the edges of $\Gamma$ is a quasi-isomorphism. Let $[\Gamma]\in H^i(\mathsf{b}_1\GC_k)$ be a non-zero equivalence class, where $\Gamma$ is a loop graph with $i$ vertices. Then $i\geq 1$ and $i\equiv 2k+1\mod4$, since otherwise $[\Gamma]=0$. The sum $b_1(\Gamma)$ contains the loop graph with only passing vertices, call it $\Gamma^\circlearrowleft$, while all other terms are graphs with at least one source and one target. The connecting morphism $\delta:H^i(\mathsf{b}_1\M_k^{+,\geq2})\rightarrow H^{i+1}(\mathsf{b}_1\M_k^{+,1})$ is the univalent part $d_u$ of the differential, and is hence zero on graphs with both a source and a target. Hence we gather $\delta b_1(\Gamma)=d_u \Gamma^\circlearrowleft$. We can clearly see that this element belongs in the diagonal of $H^{i+1}(\mathsf{b}_1\M_k^{+,1})=\mathsf{b}_1\M_k^{+,1}(\frac{\infty_1}{0})\oplus \mathsf{b}_1\M_k^{+,1}(\frac{0}{\infty_1})$, that it is non-zero and this assignment is unique.
    Hence the long exact sequence on cohomology splits as short exact sequences
    \begin{equation*}
        \xymatrix{
        0 \ar[r] & H^{i-k}(\mathsf{b}_1\M_k^{+,\geq2}) \ar[r] & H^{i-k+1}(\mathsf{b}_1\M_k^{+,1}) \ar[r] & H^{i-k+1}(\mathsf{b}_1\M_k^{+}) \ar[r] & 0}.
    \end{equation*}
    When $i$ is even, the whole sequence is zero. When $i\equiv2k+3\mod 4$, then $H^{i-k}(\mathsf{b}_1\M_k^{+,\geq2})$ acyclic, and so $H^{i-k+1}(\mathsf{b}_1\M_k^{+})\cong H^{i-k+1}(\mathsf{b}_1\M_k^{+,1})=\mathbb{K}[k-i-1]\oplus\mathbb{K}[k-i-1]$.
    When $i\equiv 2k+1\mod 4$, then we get the short exact sequence
    \begin{equation*}
        \xymatrix{
        0 \ar[r] & \mathbb{K}[k-i] \ar[r] & \mathbb{K}[k-i-1]\oplus \mathbb{K}[k-i-1] \ar[r] & H^{i-k+1}(\mathsf{b}_1\M_k^{+}) \ar[r] & 0}
    \end{equation*}
    where the first map is the connecting morphism $\delta$, being identified with the suspended diagonal map. Hence we gather that $H^{i-k+1}(\mathsf{b}_1\M_k^{+})\cong (\mathbb{K}[k-i-1]\oplus \mathbb{K}[k-i-1])/\delta(\mathbb{K}[k-i])\cong \mathbb{K}[k-i-1]$.
\end{proof}
\subsection{Describing $\M_k^+$ as a mapping cone}
So far we have only been able to describe the cohomology of $\M_k^+$ via a short exact sequence. In this section we will show where this short exact sequence originates from, and that there is $\M_k^+$ quasi-isomorphic to a mapping cone of directed graph complexes. Let $P:\dGC_k\rightarrow\dGC_k/\dGC_k^s\oplus \dGC_k/\dGC_k^t$ be the natural projection. The suspended mapping cone of $P$ is the complex $Cone(P)[1]=\big(\dGC_k\oplus(\dGC_k/\dGC_k^s[1]\oplus \dGC_k/\dGC_k^t[1]),d_c\big)$ where
\begin{equation*}
    d_c(\Gamma,(\Gamma_1,\Gamma_2))=(d_s\Gamma,(-P(\Gamma)-d_s\Gamma_1,-P(\Gamma)-d_s\Gamma_2)).
\end{equation*}
The mapping cone naturally fits in the short exact sequence
\begin{equation*}
    \xymatrix{
        0 \ar[r] & \dGC_k/\dGC_k^s[1]\oplus \dGC_k/\dGC_k^t[1] \ar[r] & Cone(P)[1] \ar[r] & \dGC_k \ar[r] & 0}
\end{equation*}
which resembles the short exact sequence
\begin{equation*}
    \xymatrix{
        0 \ar[r] & \M_k^{+,1}\ar[r] & \M_k^+ \ar[r] & \M_k^{+,\geq2} \ar[r] & 0}
\end{equation*}
in that the cohomology of the leftmost and rightmost complexes are the same. We will relate these two short exact sequences with injective chain maps, giving us the commutative diagram
\begin{equation*}
    \xymatrix{
    & 0 \ar[d] & 0 \ar[d] & 0 \ar[d] & \\
    0 \ar[r] & {\mathsf{dGC}_k/\mathsf{dGC}_k^s[1]\oplus\mathsf{dGC}_k/\mathsf{dGC}_k^t[1]} \ar[d]^{a} \ar[r] & Cone(P)[1] \ar[r] \ar[d]^{a\oplus b} & \mathsf{dGC}_k \ar[d]^{b} \ar[r] & 0  \\
    0 \ar[r] & {\M_k^{+,1}} \ar[r] \ar[d] & \M_k^{+} \ar[r] \ar[d] & {\M_k^{+,\geq2}} \ar[d] \ar[r] & 0 \\
    0 \ar[r] & \mathsf{Q}_1 \ar[r] \ar[d] & \mathsf{Q}_2 \ar[r] \ar[d] & \mathsf{Q}_3 \ar[r] \ar[d] & 0\\
    & 0 & 0 & 0 & 
    }
\end{equation*}
If we show that $a$ and $b$ are quasi-isomorphisms,then the lower short exact sequence is acyclic, implying main theorem 2.
\begin{theorem}
    The map $a\oplus b:Cone(P)[1]\rightarrow \M_k^+$ is a quasi-isomorphism.
\end{theorem}
The proof follows from the two following propositions.
\begin{proposition}
    Let $a:\dGC_k/\dGC^{s}_k[1]\oplus\dGC_k/\dGC^t_k[1]\rightarrow \M_k^{+,1}$ be the map defined by $a(\Gamma_1,\Gamma_2)=d_u(\Gamma_1(\frac{0}{\infty_1})=\Gamma_2(\frac{\infty_1}{0}))$. Then $a$ is a chain map, and furthermore it is a quasi-isomorphism.
\end{proposition}
\begin{proof}
    For abbreviation, let $\mathcal{C}_k=\dGC_k/\dGC^{s}_k\oplus\dGC_k/\dGC^t_k$, and $\mathcal{C}_k[1]$ be the suspension. We first check that the differentials commute. The differential of the degree shifted complex $\mathcal{C}[1]$ is on the form $-d_s$. So 
    \begin{equation*}
        a\circ d(\Gamma_1,\Gamma_2)=a(-d_s\Gamma_1,-d_s\gamma_2)=-d_u d_s(\Gamma_1(\frac{0}{\infty_1})+\Gamma_2(\frac{\infty_1}{0})).
    \end{equation*}
    On the other hand $d_u\circ d_u=0$, and $d_sd_u+d_ud_s=0$ in $\M_k^+$, so
    \begin{equation*}
        d\circ a(\Gamma_1,\Gamma_2)=d_sd_u(\Gamma_1(\frac{0}{\infty_1})+\Gamma_2(\frac{\infty_1}{0}))+d_ud_u(\Gamma_1(\frac{0}{\infty_1})+\Gamma_2(\frac{\infty_1}{0}))=-d_ud_s(\Gamma_1(\frac{0}{\infty_1})+\Gamma_2(\frac{\infty_1}{0})).
    \end{equation*}
    Hence $a\circ d= d\circ a$.
    Lastly we show that $a^*:H^\bullet(\mathcal{C}_k[1])\rightarrow H^\bullet(\M_k^{+,1})$ is an isomorphism. We note that the chain map $h:\mathcal{C}_k\rightarrow \M_k^{*,\geq2}(\frac{0}{\infty_1})\oplus \M_k^{*,\geq2}(\frac{\infty_1}{0})$ where $(\Gamma_1,\Gamma_2)\mapsto (\Gamma_1(\frac{0}{\infty_1}),\Gamma_2(\frac{\infty_1}{0}))$ is an isomorphism. Furthermore, corollary \ref{cor:mGCDiagram} gives that the connecting morphism $\delta:H^\bullet(\M_k^{*,\geq2}(\frac{0}{\infty_1})\oplus \M_k^{*,\geq2}(\frac{\infty_1}{0}))\rightarrow H^{\bullet+1}(\M_k^{+,1})$ is an isomorphism, and is given by $\delta[\Gamma]=[d_u\Gamma]$. Hence we get the following chain of isomorphisms of cohomology groups
    \begin{equation*}
        \xymatrix{
        0 \ar[r] & {H^{\bullet}(\mathcal{C}_k[1])} \ar[r]^{id[1]} & H^{\bullet-1}(\mathcal{C}_k) \ar[r]^-h & {H^{\bullet-1}(\M_k^{*,\geq2}(\frac{0}{\infty_1})\oplus \M_k^{*,\geq2}(\frac{\infty_1}{0}))} \ar[r]^-\delta & H^\bullet(\M_k^{+,1}) \ar[r] & 0 }.
    \end{equation*}
    We can easily see that the composition of these maps is equal to $a^*$, finishing the proof.
\end{proof}
Let $b:\dGC_k\rightarrow\mGC_k^{+,\geq2}$ be the quasi-isomorphism from proposition \ref{prop:M*cohom}. From the short exact sequence
\begin{equation*}
    \xymatrix{
    0 \ar[r] & \M_k^{+,1} \ar[r] & \M_k^+ \ar[r] & \M_k^{+,\geq2} \ar[r] & 0}
\end{equation*}
we see that $\M_k^+\cong \M_k^{+,\geq2}\oplus\M_k^{+,1} $ as a vector space. Let $a\oplus b:Cone(P)[1]\rightarrow\M_k^+$ be the linear map defined using $a$ and $b$ the decomposition of $\M_k^+$.
\begin{proposition}
    The map $a\oplus b$ is a chain map.
\end{proposition}
\begin{proof}
    For abbreviation, set $g:=a\oplus b$.
    Let $(\Gamma,(\Gamma_1,\Gamma_2))\in Cone(P)[1]$. Then
    \begin{align*}
        g\circ d_c((\Gamma,(\Gamma_1,\Gamma_2)))&=g(d_s\Gamma,(-f(\Gamma)-d_s\Gamma_1,-f(\Gamma)-d_s\Gamma_2))\\
        &=d_s\Gamma(\omega)-d_u\Gamma(\frac{\infty_1}{0})-d_ud_s\Gamma_1(\frac{\infty_1}{0})-d_u\Gamma(\frac{0}{\infty_1})-d_ud_s\Gamma_2(\frac{0}{\infty_1})\\
        &=d_s\Gamma(\omega)+d_u\Gamma(\omega)+d_sd_u\Gamma_1(\frac{\infty_1}{0})+d_sd_u\Gamma_2(\frac{0}{\infty_1})\\
        &=d(\Gamma(\omega)+d_u\Gamma_1(\frac{\infty_1}{0})+d_u\Gamma_2(\frac{0}{\infty_1})\\
        &=d\circ g\big(\Gamma,(\Gamma_1,\Gamma_2)\big),
    \end{align*}
    showing that $g$ is a chain map.
\end{proof}
\subsection{A second long exact sequence with $H^\bullet(\M_k^+)$}\label{sec:dGCst}
We finish of this chapter by describing a second long exact sequence with $H^\bullet(\M_k^+)$.
Let us use the abbreviation $\mathsf{C}_k$ for $\dGC_k$, and similarly for other subcomplexes of $\dGC_k$. Further for two complexes $A,B$ (with an assumed natural morphism $A\rightarrow B$, we write $A\oplus_c B[1]$ for the corresponding suspended mapping cone.
We remark that $\C^{st}_k$ is the kernel of the map $P:\C_k\rightarrow\C_k/\C_k^s\oplus \C_k/\C_k^t$, and hence it is a subcomplex of $Cone(P)[1]$.
Further the mapping cone $\C^{s+t}_k\oplus_c(\C^{s+t}_k/\C^{st}_k)[1]$ is a subcomplex of $Cone(P)[1]$. We can check that the following diagram commutes
 \begin{equation*}
        \xymatrix{
        & 0 \ar[d] & 0 \ar[d] & 0 \ar[d] & \\
        0 \ar[r] & \mathsf{C}^{st}\oplus_c 0 \ar[r] \ar[d]^{id} & {\mathsf{C}^{s+t}_k\oplus_c (\mathsf{C}^{s+t}_k/\mathsf{C}^{st}_k)[1]} \ar[r] \ar[d] & {\mathsf{C}^{s+t}_k/\mathsf{C}^{st}_k\oplus_c (\mathsf{C}^{s+t}_k/\mathsf{C}^{st})[1]} \ar[d] \ar[r] & 0 \\
        0 \ar[r] & \mathsf{C}^{st}\oplus_c 0 \ar[r] \ar[d] & {\mathsf{C}_k\oplus_c (\mathsf{C}_k/\mathsf{C}^{s}_k\oplus\mathsf{C}_k/\mathsf{C}^{t}_k)[1]} \ar[r] \ar[d] & {\mathsf{C}_k/\C_k^{st}\oplus_c (\mathsf{C}_k/\mathsf{C}^{s}_k\oplus\mathsf{C}_k/\mathsf{C}^{t}_k)[1]} \ar[d] \ar[r] & 0 \\
        0 \ar[r] & 0 \ar[r] \ar[d] & {\mathsf{C}_k^\circlearrowleft\oplus_c(\mathsf{C}_k^\circlearrowleft\oplus\mathsf{C}_k^\circlearrowleft)[1]} \ar[r]^{id} \ar[d] & \mathsf{C}_k^\circlearrowleft\oplus_c(\mathsf{C}_k^\circlearrowleft\oplus\mathsf{C}_k^\circlearrowleft)[1] \ar[r]\ar[d] & 0\\
         & 0 & 0 & 0 & 
        }
\end{equation*}
where we recall that $\C^\circlearrowleft_k:=\C_k/\C_k^{s+t}$.
\begin{proposition}
    \hspace{1cm}
    \begin{enumerate}
    \item $\C^{st}_k\oplus_c 0 \cong \C^{st}$.
    \item The complex ${\mathsf{C}_k^\circlearrowleft\oplus_c(\mathsf{C}_k^\circlearrowleft\oplus\mathsf{C}_k^\circlearrowleft)[1]} $ is quasi-isomorphic to $\C_k^\circlearrowleft[1]$.
    \item The complex $\mathsf{C}^{s+t}_k/\mathsf{C}^{st}_k\oplus_c (\mathsf{C}^{s+t}_k/\mathsf{C}^{st})[1]$ is acyclic.
\end{enumerate}
\end{proposition}
\begin{proof}
    (1) follows from direct inspection.
    (2) follows by first considering the short exact sequence
    \begin{equation*}
        \xymatrix{
        0 \ar[r] & \C^\circlearrowleft_k[1]\oplus \C^\circlearrowleft_k[1] \ar[r ] & \mathsf{C}_k^\circlearrowleft\oplus_c(\mathsf{C}_k^\circlearrowleft\oplus\mathsf{C}_k^\circlearrowleft)[1] \ar[r] & \C_k^\circlearrowleft \ar[r]& 0}
    \end{equation*}
    noting that the the connecting morphism is the diagonal map $H^\bullet(\C_k^\circlearrowleft)\rightarrow H^\bullet(\C_k^\circlearrowleft)\oplus H^\bullet(\C_k^\circlearrowleft)$, and thus it is injective. In particular, the induced map $H^\bullet(\mathsf{C}_k^\circlearrowleft\oplus_c(\mathsf{C}_k^\circlearrowleft\oplus\mathsf{C}_k^\circlearrowleft)[1])\rightarrow \C_k^\circlearrowleft$ is zero. By exactness of the long exact sequence, we get $H^\bullet(\mathsf{C}_k^\circlearrowleft\oplus_c(\mathsf{C}_k^\circlearrowleft\oplus\mathsf{C}_k^\circlearrowleft)[1])\cong H^\bullet(\C^\circlearrowleft_k[1]\oplus \C^\circlearrowleft_k[1])/H^\bullet(\C^\circlearrowleft_k[1])\cong\C^\circlearrowleft_k[1]$.
    (3) follows from the fact that a mapping cone $Cone(f)$ is acyclic if and only if $f$ is a quasi-isomorphism. This complex is the mapping cone of the identity morphism, and hence the cone is acyclic.
\end{proof}
By reading the diagram, we get the following corollary by considering the induced long exact sequence produced by the middle column or row.
\begin{corollary}\label{cor:dGCstM+}
    There is a long exact sequence on cohomology
    \begin{equation*}
    \xymatrix{
        \cdots \ar[r] & H^{\bullet-2}(\dGC^\circlearrowleft_k) \ar[r] & H^{\bullet}(\dGC^{st}_{k}) \ar[r] & H^\bullet(\M_k^+) \ar[r] & H^{\bullet-1}(\dGC_k^\circlearrowleft) \ar[r] & H^{\bullet+1}(\mathsf{dGC}^{st}_{k}) \ar[r] & \cdots}
    \end{equation*}
\end{corollary}
By corollary \ref{cor:CohomDGC3}, $H^l(\dGC_3^\circlearrowleft)=0$ for $-2\leq l \leq 1$. In particular $H^0(\M_3^+)\cong H^0(\dGC^{st}_3)$. Finally $H^0(\dGC_3^{st})=\mathfrak{grt}\oplus\mathfrak{grt}$ by proposition \ref{prop:mGC}, giving us a second proof of the main theorem.

\newcommand{\Ed}[0]{
    \ \begin{tikzpicture}[shorten <=1pt,>=stealth,semithick,baseline={([yshift=-.5ex]current bounding box.center)}]
    \node[mblack] (r) at (0.7,0) {};
    \node[mblack] (l) at (0,0) {}
        edge [->] (r);
    \end{tikzpicture} \
}
\newcommand{\halfEd}[0]{
    \ \begin{tikzpicture}[semithick,baseline={([yshift=-.5ex]current bounding box.center)}]
    \node[mblack] (r) at (0.7,0) {};
    \node[mblack] (l) at (0,0) {};
    \draw[OneMark] (l) -- (r);
    \end{tikzpicture} \
}
\newcommand{\sEd}[0]{
    \ \begin{tikzpicture}[shorten <=1pt,>=stealth,semithick,baseline=0
    ]
    \node[mblack] (l) at (0,0) {};
    \node[mblack] (r) at (0.7,0) {};
    \draw[dotted,->] (l) to node[auto] {$\scriptstyle{s}$} (r);
    \end{tikzpicture} \
}
\newcommand{\sEdO}[0]{
    \ \begin{tikzpicture}[shorten <=1pt,>=stealth,semithick,baseline=0
    ]
    \node[mblack] (l) at (0,0) {};
    \node[mblack] (r) at (0.7,0) {};
    \draw[dotted,<-] (l) to node[auto] {$\scriptstyle{s}$} (r);
    \end{tikzpicture} \
}
\newcommand{\tEd}[0]{
    \ \begin{tikzpicture}[shorten <=1pt,>=stealth,semithick,baseline=0
    ]
    \node[mblack] (l) at (0,0) {};
    \node[mblack] (r) at (0.7,0) {};
    \draw[dotted,->] (l) to node[auto] {$\scriptstyle{t}$} (r);
    \end{tikzpicture} \
}
\newcommand{\tEdO}[0]{
    \ \begin{tikzpicture}[shorten <=1pt,>=stealth,semithick,baseline=0
    ]
    \node[mblack] (l) at (0,0) {};
    \node[mblack] (r) at (0.7,0) {};
    \draw[dotted,<-] (l) to node[auto] {$\scriptstyle{t}$} (r);
    \end{tikzpicture} \
}
\usetikzlibrary{snakes}
\newcommand{\wEd}[0]{
    \ \begin{tikzpicture}[shorten <=1pt,>=stealth,semithick,baseline=0
    ]
    \node[mblack] (l) at (0,0) {};
    \node[mblack] (r) at (0.7,0) {};
    \draw [->,snake=snake,
segment amplitude=0.6mm,
segment length=2.2mm,
line after snake=1.3mm
] (l) to (r);
    \end{tikzpicture} \
}
\newcommand{\wEdO}[0]{
    \ \begin{tikzpicture}[shorten <=1pt,>=stealth,semithick,baseline=0
    ]
    \node[mblack] (l) at (0,0) {};
    \node[mblack] (r) at (0.7,0) {};
    \draw [->,snake=snake,
segment amplitude=0.6mm,
segment length=2.2mm,
line after snake=1.3mm
] (r) to (l);
    \end{tikzpicture} \
}

\newcommand{\sBi}[0]{
    \ \begin{tikzpicture}[shorten <=1pt,>=stealth,semithick,baseline={([yshift=-.5ex]current bounding box.center)}]
        \node[mblack] (l) at (-0.6,0) {};
        \node[mblack] (r) at (0.6,0) {};
        \node[mblack] (m) at (0,0.1) {}
            edge [->] (r)
            edge [->] (l);
    \end{tikzpicture} \
}
\newcommand{\tBi}[0]{
    \ \begin{tikzpicture}[shorten <=1pt,>=stealth,semithick,baseline={([yshift=-.5ex]current bounding box.center)}]
        \node[mblack] (l) at (-0.6,0) {};
        \node[mblack] (r) at (0.6,0) {};
        \node[mblack] (m) at (0,-0.1) {}
            edge [<-] (r)
            edge [<-] (l);
    \end{tikzpicture} \
}

\newcommand{\stBi}[0]{
    \ \begin{tikzpicture}[shorten <=1pt,>=stealth,semithick,baseline={([yshift=-.5ex]current bounding box.center)}]
        \node[mblack] (v1) at (0,0) {};
        \node[mblack] (v2) at (0.6,-0.1) {};
        \node[mblack] (v3) at (1.2,0) {};
        \node[mblack] (v4) at (1.8,-0.1) {};
        \draw [->] (v1) to (v2);
        \draw [<-] (v2) to (v3);
        \draw [->] (v3) to (v4);
    \end{tikzpicture} \
    +(-1)^{k}
    \ \begin{tikzpicture}[shorten <=1pt,>=stealth,semithick,baseline={([yshift=-.5ex]current bounding box.center)}]
        \node[mblack] (v1) at (0,-0.1) {};
        \node[mblack] (v2) at (0.7,0) {};
        \node[mblack] (v3) at (1.4,-0.1) {};
        \node[mblack] (v4) at (2.1,0) {};
        \draw [<-] (v1) to (v2);
        \draw [->] (v2) to (v3);
        \draw [<-] (v3) to (v4);
    \end{tikzpicture} \
}

\section{Explicit example of two cohomology classes of $Der(\mathcal{H}olieb_{p,q})$}
One of the simplest cohomology classes in $H^0(\GC_2)=\mathfrak{grt}$ is given by the tetrahedron class
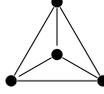
\begin{equation*}
    \begin{tikzpicture}[shorten >=1pt,node distance=1.2cm,auto,>=angle 90,scale=0.7,baseline={([yshift=-.5ex]current bounding box.center)}]
    \node[sblack] (t0) at (0,1) {};
    \node[sblack] (t1) at (0.87,-0.5) {};
    \node[sblack] (t2) at (-0.87,-0.5) {};
    \node[sblack] (c) at (0,0) {};
    \draw (t0) -- (t1) -- (t2) -- (t0) -- (c);
    \draw (t1) -- (c) -- (t2);
    \end{tikzpicture}
\end{equation*}
which has loop number three. In accordance with our Main Theorem, $H^0(Def(\mathcal{H}olieb_{1,1}^\circlearrowleft))=H^0(\GC_2)\oplus H^0(\GC_2)$, so the tetrahderon class should act on $\mathcal{H}olieb_{1,1}^\circlearrowleft$ in two homotopy inequivalent ways.

We have an explicit morphism of complexes $H^0(\dGC^{st}_3)\rightarrow H^0(Der(\mathcal{H}olieb_{1,1}^\circlearrowleft))$ given by attaching hairs to a cohomology class of $\dGC^{st}_3$ as described in \ref{cor:dGCstM+}. Therefore, to see the above mentioned two homotopy inequivalent actions of the tetrahedron class on $\mathcal{H}olieb_{1,1}^\circlearrowleft$, we have to find \textit{explicit} incarnations of the tetrahedron class in $\dGC^{st}_3$. In this section we explicitly describe these two incarnations denoted by $\alpha^s$ and $\alpha^t$, both having loop number three.

\subsection{A reduced version of $\dGC^{st}_k$}
Following S. Merkulovs paper \cite{M3}, we consider a "smaller" version $\widehat{\dGC}^{st}_k$ of the complex $\dGC^{st}_k$ which is quasi-isomorphic to it.
First, define the graph complex $\overline{\dGC}_k$ generated by graphs whose vertices are at least trivalent and which have four kind of edges:
\begin{enumerate}
    \item Solid edges of degree zero $\Ed$
    \item Dotted $s$-edges of degree one $\sEd$
    \item Dotted $t$-edges of degree one $\tEd$
    \item wavy edges of degree two $\wEd$
\end{enumerate}
Let $\Gamma$ be a graph and let $e$ be the total number of edges, $v$ the number of vertices, $e_1$ the number of $s$-edges and $t$-edges, and $e_2$ the number of wavy edges. Then the degree of $\Gamma$ is
\begin{equation*}
    |\Gamma|=(v-1)k+(1-k)e+e_1+2e_2.
\end{equation*}
The space is defined as the graphs invariant under permutations of vertices or edges depending on the parity of $k$, with the change that labels of non-solid edges can only be permuted with edges of the same type. Further permutations of dotted edges gives the sign of the permutation for $k$ odd.
Further the non-solid edges satisfy the relations
\begin{enumerate}
    \item $\sEd=(-1)^{k+1}\sEdO$
    \item $\tEd=(-1)^{k+1}\tEdO$
    \item $\wEd=(-1)^{k+1}\wEdO$
\end{enumerate}
The differential is defined as $d=d_V+d_E$, where $d_V$ acts by splitting vertices so that no univalent nor bivalent vertices are created, and with the same sign rule as in $\dGC_d$. The term $d_E$ acts on edges accordingly:
\begin{enumerate}
    \item $d_E \Ed=\tEd-\sEd$
    \item $d_E \sEd=\wEd$
    \item $d_E \tEd=\wEd$
    \item $d_E \wEd=0$
\end{enumerate}
We say that a vertex in such a graph is a \textit{solid source} if the attached edges are solid and outgoing or $t$-dotted. A vertex is a \textit{solid target} if the attached edges are solid and incoming or $s$-dotted. Consider the subcomplex $\overline{\dGC}^{st}_k$ of $\overline{\dGC}_k$ generated by graphs that either
\begin{enumerate}
    \item have at least one solid source and one solid target,
    \item have at least one dotted $s$-edge and one solid target,
    \item have at least one dotted $t$-edge and one solid source,
    \item have at least one dotted $s$-edge and one dotted $t$-edge,
    \item have at least one wavy edge.
\end{enumerate}
Consider the map $f:\overline{\dGC}^{st}_k\rightarrow\dGC^{st}_k$ where a graph is mapped to the graph in $\dGC^{st}_k$ where solid edges remain the same, but $s$-dotted, $t$-dotted and wavy edges are replaced by the following edges:
\begin{enumerate}
    \item $\sEd \mapsto \sBi$
    \item $\tEd \mapsto \tBi$
    \item $\wEd \mapsto \stBi$
\end{enumerate}
Then the map $f$ is a quasi-isomorphism \cite{M3}.
Consider the subcomplex $\mathsf{Z}_d$ of $\overline{\dGC}^{st}_k$ of graphs that either
\begin{enumerate}
    \item have at least one $s$-edge and one $t$-edge,
    \item have at least one $s$-edge and one wavy edge,
    \item have at least one $t$edge and one wavy edge,
    \item have at least two wavy edges.
\end{enumerate}
This subcomplex is acyclic, hence the projection  $\overline{\dGC}^{st}_k\rightarrow\overline{\dGC}^{st}_k/\mathsf{Z}_d$ is a quasi-isomorphism. Let $\widehat{\dGC}^{st}_k=\overline{\dGC}^{st}_k/\mathsf{Z}_d$. This complex consists of graphs that either
\begin{enumerate}
    \item have only solid edges such that there is at least one solid source and one solid target,
    \item have only solid and dotted $s$-edges such that there is at least one solid target and one dotted $s$-edge,
    \item have only solid and dotted $t$-edges such that there is at least one solid source and one dotted $t$-edge,
    \item have only solid and wavy edges such that there is at least one wavy edge.
\end{enumerate}
We will use this complex when finding the cohomology classes.
\subsection{Calculating the cohomology classes}
Let $\gamma^s$ and $\gamma^t$ be the graphs
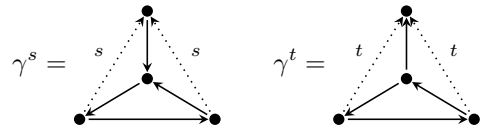
\begin{equation*}
    \gamma^s=
    \begin{tikzpicture}[shorten <=1pt,>=stealth,semithick,baseline={([yshift=-.5ex]current bounding box.center)},scale=0.6]
        \node[sblack] (t) at (0,0) {};
        \node[sblack] (m) at (0,-1.5) {};
        \node[sblack] (l) at (-1.5,-2.4) {};
        \node[sblack] (r) at (1.5,-2.4) {};
        \draw[->] (t) to (m);
        \draw[->] (m) to (l);
        \draw[->] (r) to (m);
        \draw[->] (l) to (r);
        \draw[dotted,->] (l) to node[auto] {$\scriptstyle{s}$} (t);
        \draw[dotted,->] (r) to node[auto,swap] {$\scriptstyle{s}$} (t);
    \end{tikzpicture}
    \qquad
    \gamma^t=
    \begin{tikzpicture}[shorten <=1pt,>=stealth,semithick,baseline={([yshift=-.5ex]current bounding box.center)},scale=0.6]
        \node[sblack] (t) at (0,0) {};
        \node[sblack] (m) at (0,-1.5) {};
        \node[sblack] (l) at (-1.5,-2.4) {};
        \node[sblack] (r) at (1.5,-2.4) {};
        \draw[->] (m) to (t);
        \draw[->] (m) to (l);
        \draw[->] (r) to (m);
        \draw[->] (l) to (r);
        \draw[dotted,->] (l) to node[auto] {$\scriptstyle{t}$} (t);
        \draw[dotted,->] (r) to node[auto,swap] {$\scriptstyle{t}$} (t);
    \end{tikzpicture}
\end{equation*}
Note that $\gamma^s$ is in $\widehat{\dGC}^{s}_3$ but not in $\widehat{\dGC}^{st}_3$. Similarly $\gamma^t$ is in $\widehat{\dGC}^{t}_3$ but not in $\widehat{\dGC}^{st}_3$.
The vertices of $\gamma^s$ are all trivalent, and so the differential acts by only changing solid edges to $s$-edges. We have omitted the $s$ on $s$-edges for clarity in the following pictures. Hence
\begin{align*}
    \alpha^s:=d(\gamma^s)&=
    \begin{tikzpicture}[shorten <=1pt,>=stealth,semithick,baseline={([yshift=-.5ex]current bounding box.center)},scale=0.6]
        \node[sblack] (t) at (0,0) {};
        \node[sblack] (m) at (0,-1.5) {};
        \node[sblack] (l) at (-1.5,-2.4) {};
        \node[sblack] (r) at (1.5,-2.4) {};
        \draw[dotted,->] (t) to (m);
        \draw[->] (m) to (l);
        \draw[->] (r) to (m);
        \draw[->] (l) to (r);
        \draw[dotted,->] (l) to (t);
        \draw[dotted,->] (r) to (t);
    \end{tikzpicture}
    +
    \begin{tikzpicture}[shorten <=1pt,>=stealth,semithick,baseline={([yshift=-.5ex]current bounding box.center)},scale=0.6]
        \node[sblack] (t) at (0,0) {};
        \node[sblack] (m) at (0,-1.5) {};
        \node[sblack] (l) at (-1.5,-2.4) {};
        \node[sblack] (r) at (1.5,-2.4) {};
        \draw[->] (t) to (m);
        \draw[dotted,->] (m) to (l);
        \draw[->] (r) to (m);
        \draw[->] (l) to (r);
        \draw[dotted,->] (l) to (t);
        \draw[dotted,->] (r) to (t);
    \end{tikzpicture}
    +
    \begin{tikzpicture}[shorten <=1pt,>=stealth,semithick,baseline={([yshift=-.5ex]current bounding box.center)},scale=0.6]
        \node[sblack] (t) at (0,0) {};
        \node[sblack] (m) at (0,-1.5) {};
        \node[sblack] (l) at (-1.5,-2.4) {};
        \node[sblack] (r) at (1.5,-2.4) {};
        \draw[->] (t) to (m);
        \draw[->] (m) to (l);
        \draw[dotted,->] (r) to (m);
        \draw[->] (l) to (r);
        \draw[dotted,->] (l) to (t);
        \draw[dotted,->] (r) to (t);
    \end{tikzpicture}
    +
    \begin{tikzpicture}[shorten <=1pt,>=stealth,semithick,baseline={([yshift=-.5ex]current bounding box.center)},scale=0.6]
        \node[sblack] (t) at (0,0) {};
        \node[sblack] (m) at (0,-1.5) {};
        \node[sblack] (l) at (-1.5,-2.4) {};
        \node[sblack] (r) at (1.5,-2.4) {};
        \draw[->] (t) to (m);
        \draw[->] (m) to (l);
        \draw[->] (r) to (m);
        \draw[dotted,->] (l) to (r);
        \draw[dotted,->] (l) to (t);
        \draw[dotted,->] (r) to (t);
    \end{tikzpicture}\\
    \alpha^t:=d(\gamma^t)&=
    \begin{tikzpicture}[shorten <=1pt,>=stealth,semithick,baseline={([yshift=-.5ex]current bounding box.center)},scale=0.6]
        \node[sblack] (t) at (0,0) {};
        \node[sblack] (m) at (0,-1.5) {};
        \node[sblack] (l) at (-1.5,-2.4) {};
        \node[sblack] (r) at (1.5,-2.4) {};
        \draw[dotted,<-] (t) to (m);
        \draw[->] (m) to (l);
        \draw[->] (r) to (m);
        \draw[->] (l) to (r);
        \draw[dotted,->] (l) to (t);
        \draw[dotted,->] (r) to (t);
    \end{tikzpicture}
    +
    \begin{tikzpicture}[shorten <=1pt,>=stealth,semithick,baseline={([yshift=-.5ex]current bounding box.center)},scale=0.6]
        \node[sblack] (t) at (0,0) {};
        \node[sblack] (m) at (0,-1.5) {};
        \node[sblack] (l) at (-1.5,-2.4) {};
        \node[sblack] (r) at (1.5,-2.4) {};
        \draw[<-] (t) to (m);
        \draw[dotted,->] (m) to (l);
        \draw[->] (r) to (m);
        \draw[->] (l) to (r);
        \draw[dotted,->] (l) to (t);
        \draw[dotted,->] (r) to (t);
    \end{tikzpicture}
    +
    \begin{tikzpicture}[shorten <=1pt,>=stealth,semithick,baseline={([yshift=-.5ex]current bounding box.center)},scale=0.6]
        \node[sblack] (t) at (0,0) {};
        \node[sblack] (m) at (0,-1.5) {};
        \node[sblack] (l) at (-1.5,-2.4) {};
        \node[sblack] (r) at (1.5,-2.4) {};
        \draw[<-] (t) to (m);
        \draw[->] (m) to (l);
        \draw[dotted,->] (r) to (m);
        \draw[->] (l) to (r);
        \draw[dotted,->] (l) to (t);
        \draw[dotted,->] (r) to (t);
    \end{tikzpicture}
    +
    \begin{tikzpicture}[shorten <=1pt,>=stealth,semithick,baseline={([yshift=-.5ex]current bounding box.center)},scale=0.6]
        \node[sblack] (t) at (0,0) {};
        \node[sblack] (m) at (0,-1.5) {};
        \node[sblack] (l) at (-1.5,-2.4) {};
        \node[sblack] (r) at (1.5,-2.4) {};
        \draw[<-] (t) to (m);
        \draw[->] (m) to (l);
        \draw[->] (r) to (m);
        \draw[dotted,->] (l) to (r);
        \draw[dotted,->] (l) to (t);
        \draw[dotted,->] (r) to (t);
    \end{tikzpicture}
\end{align*}
Note that each graph contains a source and a target vertex, and hence $\alpha^s$ and $\alpha^t$ are 
cycles in $\widehat{\dGC}^{st}_3$.
\begin{theorem}\label{thm:tetra}
    The elements $\alpha^s$ and $\alpha^t$ are non-trivial cycles of $\widehat{\dGC}^{st}_3$. Furthermore, they represent two different cohomology classes in $H(\widehat{\dGC}^{st}_3)$.
\end{theorem}
Let $X$ and $A$ be the spaces of graphs of loop number three with at least one source and target vertex of degree $-1$ and $0$ respectively, and the differential is a linear map $d:X\rightarrow A$. We want to show that there is no $\beta^s\in X$ so that $d(\beta^s)=\alpha^s$. To prove this at this stage, we would need to compute the differential of hundreds of graphs in $X$ and check that $\alpha^s$ is linearly independent from them. But we can reduce this problem to give us only a handful of graphs to study.
Let $A^s$ be the subspace of $A$ of graphs spanned by tetrahedron graphs with three $s$-edges. Note that $\alpha^s\in A^s$. The graphs $a_1...a_{10}$ in figure \ref{fig:A} is basis for $A^s$.

Let $\overline{A}$ be the orthogonal complement of $A^s$ so that $A=A^s\oplus\overline{A}$ and let $p_s:A^s\oplus\overline{A}\rightarrow A^s$ be the projection. Let $\overline{X}$ be the kernel of the map $p_s\circ d:X\rightarrow A^s$ and $X^s$ its orthogonal complement so that $X=X^s\oplus X\overline{X}$. The space $X^s$ is spanned by the graphs $x_1,...,x_{11}$ in figure \ref{fig:X}.
Let $\iota^s:X^s\rightarrow X^s\oplus\overline{X}$ be the inclusion map. We get the following commutative diagram:
\begin{equation*}
    \xymatrix{
        X^s\oplus\overline{X} \ar[r]^d & A^s\oplus \overline{A} \ar[d]^{p_s} \\
        X^s \ar[u]^{\iota_s} \ar[r]^{p_a\circ \ d\ \circ\iota_s} & A^s } 
    \end{equation*}
\begin{lemma}\label{lemma:tetra}
    Let $\alpha^s\in A^s$. If there is an element $\beta\in X=X^s\oplus\overline{X}$ such that $d(\beta)=\alpha^s$, then there is $\beta^s\in X^s$ such that $\iota_s\circ d\circ p_s(\beta^s)=\alpha^s$.
\end{lemma}
\begin{proof}
    Suppose that such an element $\beta$ exists. Then $\beta=\beta^s+\overline{\beta}$ where $\beta^s\in X^s$ and $\overline{\beta}\in\overline{X}$. Now $d(\beta^s+\overline{\beta})=\alpha$. The space $\overline{X}$ is the kernel space of $p_s\circ d$, and so $p_s\circ d(\overline{\beta})=0$. Hence $p_s\circ d(\beta)=p_s\circ d(\beta^s)=\alpha^s$. Finally $\beta^s\in X^s$ and so $p_s\circ d\circ\iota_s(\beta^s)=\alpha$.
\end{proof}
\begin{proof}[Proof of theorem \ref{thm:tetra}]
By the contrapositive statement of lemma \ref{lemma:tetra}, it is sufficient for us to show that there is no $\beta^s\in X^s$ so that $p_s\circ d\circ\iota_s(\beta^s)=\alpha^s$. We have already established the basis $a_1...a_{10}$ of $A^s$ and $x_1...x_{11}$ of $X^s$. In this basis $\alpha^s=a_1+a_5+a_7-a_{10}$. We compute the image of each vector $x_1...x_{11}$ under the differential and get a matrix representation of the map $d_s\circ d\circ p_s:X^s\rightarrow A^s$ as seen in figure \ref{fig:matrix}.
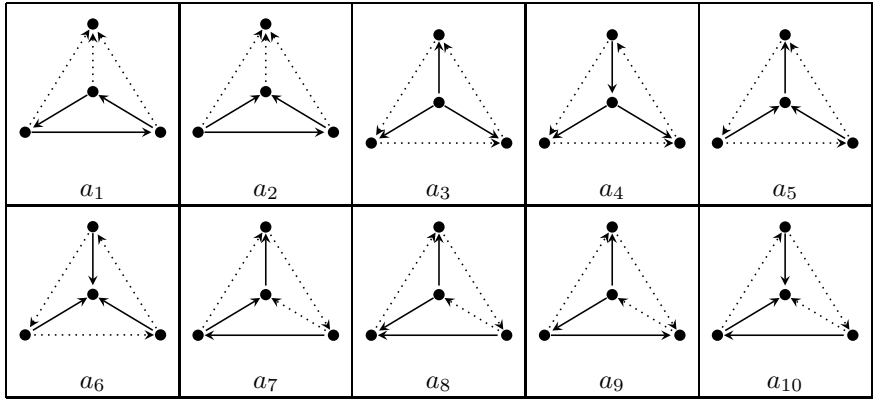
\begin{figure}
\begin{center}
    \begin{tabular}{|c|c|c|c|c|}
        \hline
\begin{tikzpicture}[
shorten <=1pt,>=stealth,semithick,
baseline=-18.6mm,
scale=0.6]
        \node[sblack] (t) at (0,0) {};
        \node[sblack] (m) at (0,-1.5) {};
        \node[sblack] (l) at (-1.5,-2.4) {}
            edge [<-] (m);
        \node[sblack] (r) at (1.5,-2.4) {}
            edge [->] (m)
            edge [<-] (l);
        \draw (t)[dotted,<-] -- (m);
        \draw (l)[dotted,->] -- (t);
        \draw (t)[dotted,<-] -- (r);
        \addvmargin{2mm}
    \end{tikzpicture} & 
\begin{tikzpicture}[
shorten <=1pt,>=stealth,semithick,
baseline=-18.6mm,
scale=0.6]
        \node[sblack] (t) at (0,0) {};
        \node[sblack] (m) at (0,-1.5) {};        \node[sblack] (l) at (-1.5,-2.4) {}
            edge [->] (m);
        \node[sblack] (r) at (1.5,-2.4) {}
            edge [->] (m)
            edge [<-] (l);
        \draw (t)[dotted,<-] -- (m);
        \draw (l)[dotted,->] -- (t);
        \draw (t)[dotted,<-] -- (r);
        \addvmargin{2mm}

    \end{tikzpicture} &
\begin{tikzpicture}[
shorten <=1pt,>=stealth,semithick,
scale=0.6]
        \node[sblack] (t) at (0,0) {};
        \node[sblack] (m) at (0,-1.5) {}
            edge [->] (t);
        \node[sblack] (l) at (-1.5,-2.4) {}
            edge [<-] (m);
        \node[sblack] (r) at (1.5,-2.4) {}
            edge [<-] (m);
        \draw (l)[dotted,->] -- (r);
        \draw (l)[dotted,<-] -- (t);
        \draw (t)[dotted,<-] -- (r);
        \addvmargin{2mm}
    \end{tikzpicture} &
\begin{tikzpicture}[
shorten <=1pt,>=stealth,semithick,
scale=0.6]
        \node[sblack] (t) at (0,0) {};
        \node[sblack] (m) at (0,-1.5) {}
            edge [<-] (t);
        \node[sblack] (l) at (-1.5,-2.4) {}
            edge [<-] (m);
        \node[sblack] (r) at (1.5,-2.4) {}
            edge [<-] (m);
        \draw (l)[dotted,->] -- (r);
        \draw (l)[dotted,<-] -- (t);
        \draw (t)[dotted,<-] -- (r);
        \addvmargin{2mm}
    \end{tikzpicture} &
\begin{tikzpicture}[
shorten <=1pt,>=stealth,semithick,
scale=0.6]
        \node[sblack] (t) at (0,0) {};
        \node[sblack] (m) at (0,-1.5) {}
            edge [->] (t);
        \node[sblack] (l) at (-1.5,-2.4) {}
            edge [->] (m);
        \node[sblack] (r) at (1.5,-2.4) {}
            edge [->] (m);
        \draw (l)[dotted,->] -- (r);
        \draw (l)[dotted,<-] -- (t);
        \draw (t)[dotted,<-] -- (r);
        \addvmargin{2mm}
    \end{tikzpicture} \\
        $a_1$ & $a_2$ & $a_3$ & $a_4$ & $a_5$ \\
        \hline
\begin{tikzpicture}[
shorten <=1pt,>=stealth,semithick,
baseline=-17.2mm,
scale=0.6]
        \node[sblack] (t) at (0,0) {};
        \node[sblack] (m) at (0,-1.5) {}
            edge [<-] (t);
        \node[sblack] (l) at (-1.5,-2.4) {}
            edge [->] (m);
        \node[sblack] (r) at (1.5,-2.4) {}
            edge [->] (m);
        \draw (l)[dotted,->] -- (r);
        \draw (l)[dotted,<-] -- (t);
        \draw (t)[dotted,<-] -- (r);
        \addvmargin{2mm}
    \end{tikzpicture} &
\begin{tikzpicture}[
shorten <=1pt,>=stealth,semithick,
scale=0.6]
        \node[sblack] (t) at (0,0) {};
        \node[sblack] (m) at (0,-1.5) {}
            edge [->] (t);
        \node[sblack] (l) at (-1.5,-2.4) {}
            edge [->] (m);
        \node[sblack] (r) at (1.5,-2.4) {}
            edge [->] (l);
        \draw (r)[dotted,->] -- (m);
        \draw (l)[dotted,->] -- (t);
        \draw (t)[dotted,->] -- (r);
        \addvmargin{2mm}
    \end{tikzpicture} &
\begin{tikzpicture}[
shorten <=1pt,>=stealth,semithick,
scale=0.6]
        \node[sblack] (t) at (0,0) {};
        \node[sblack] (m) at (0,-1.5) {}
            edge [->] (t);
        \node[sblack] (l) at (-1.5,-2.4) {}
            edge [<-] (m);
        \node[sblack] (r) at (1.5,-2.4) {}
            edge [->] (l);
        \draw (r)[dotted,->] -- (m);
        \draw (l)[dotted,->] -- (t);
        \draw (t)[dotted,->] -- (r);
        \addvmargin{2mm}
    \end{tikzpicture} &
\begin{tikzpicture}[
shorten <=1pt,>=stealth,semithick,
scale=0.6]
        \node[sblack] (t) at (0,0) {};
        \node[sblack] (m) at (0,-1.5) {}
            edge [->] (t);
        \node[sblack] (l) at (-1.5,-2.4) {}
            edge [<-] (m);
        \node[sblack] (r) at (1.5,-2.4) {}
            edge [<-] (l);
        \draw (r)[dotted,->] -- (m);
        \draw (l)[dotted,->] -- (t);
        \draw (t)[dotted,->] -- (r);
        \addvmargin{2mm}
    \end{tikzpicture} &
\begin{tikzpicture}[
shorten <=1pt,>=stealth,semithick,
scale=0.6]
        \node[sblack] (t) at (0,0) {};
        \node[sblack] (m) at (0,-1.5) {}
            edge [<-] (t);
        \node[sblack] (l) at (-1.5,-2.4) {}
            edge [->] (m);
        \node[sblack] (r) at (1.5,-2.4) {}
            edge [->] (l);
        \draw (r)[dotted,->] -- (m);
        \draw (l)[dotted,->] -- (t);
        \draw (t)[dotted,->] -- (r);
        \addvmargin{2mm}
    \end{tikzpicture} \\
        $a_6$ & $a_7$ & $a_8$ & $a_9$ & $a_{10}$ \\
        \hline
    \end{tabular}
\end{center}
    \caption{Basis of $A^s$}
    \label{fig:A}
\end{figure}
\begin{figure}
\begin{center}
    \begin{tabular}{|c|c|c|c|c|c|}
        \hline
\begin{tikzpicture}[
shorten <=1pt,>=stealth,semithick,
baseline={([yshift=-.5ex]current bounding box.center)},scale=0.6]
        \node[sblack] (t) at (0,0) {};
        \node[sblack] (m) at (0,-1.5) {}
            edge [->] (t);
        \node[sblack] (l) at (-1.5,-2.4) {}
            edge [<-] (m);
        \node[sblack] (r) at (1.5,-2.4) {}
            edge [->] (m)
            edge [<-] (l);
        \draw (l)[dotted,->] -- (t);
        \draw (t)[dotted,<-] -- (r);
        \addvmargin{2mm}
    \end{tikzpicture} &
\begin{tikzpicture}[shorten <=1pt,>=stealth,semithick,baseline={([yshift=-.5ex]current bounding box.center)},scale=0.6]
        \node[sblack] (t) at (0,0) {};
        \node[sblack] (m) at (0,-1.5) {}
            edge [->] (t);
        \node[sblack] (l) at (-1.5,-2.4) {}
            edge [->] (m);
        \node[sblack] (r) at (1.5,-2.4) {}
            edge [<-] (m)
            edge [<-] (l);
        \draw (l)[dotted,->] -- (t);
        \draw (t)[dotted,<-] -- (r);
    \end{tikzpicture} &
\begin{tikzpicture}[shorten <=1pt,>=stealth,semithick,
baseline={([yshift=-.5ex]current bounding box.center)},
scale=0.6]
        \node[sblack] (t) at (0,0) {};
        \node[sblack] (m) at (0,-1.5) {}
            edge [->] (t);
        \node[sblack] (l) at (-1.5,-2.4) {}
            edge [<-] (m);
        \node[sblack] (r) at (1.5,-2.4) {}
            edge [<-] (m)
            edge [<-] (l);
        \draw (l)[dotted,->] -- (t);
        \draw (t)[dotted,<-] -- (r);
    \end{tikzpicture} &
\begin{tikzpicture}[shorten <=1pt,>=stealth,semithick,baseline={([yshift=-.5ex]current bounding box.center)},scale=0.6]
        \node[sblack] (t) at (0,0) {};
        \node[sblack] (m) at (0,-1.5) {}
            edge [->] (t);
        \node[sblack] (l) at (-1.5,-2.4) {}
            edge [->] (m);
        \node[sblack] (r) at (1.5,-2.4) {}
            edge [->] (m)
            edge [<-] (l);
        \draw (l)[dotted,->] -- (t);
        \draw (t)[dotted,<-] -- (r);
    \end{tikzpicture} &
\begin{tikzpicture}[shorten <=1pt,>=stealth,semithick,baseline={([yshift=-.5ex]current bounding box.center)},scale=0.6]
        \node[sblack] (t) at (0,0) {};
        \node[sblack] (m) at (0,-1.5) {}
            edge [<-] (t);
        \node[sblack] (l) at (-1.5,-2.4) {}
            edge [->] (m);
        \node[sblack] (r) at (1.5,-2.4) {}
            edge [->] (m)
            edge [<-] (l);
        \draw (l)[dotted,->] -- (t);
        \draw (t)[dotted,<-] -- (r);
    \end{tikzpicture} &
\begin{tikzpicture}[shorten <=1pt,>=stealth,semithick,baseline={([yshift=-.5ex]current bounding box.center)},scale=0.6]
        \node[sblack] (t) at (0,0) {};
        \node[sblack] (m) at (0,-1.5) {}
            edge [<-] (t);
        \node[sblack] (l) at (-1.5,-2.4) {}
            edge [->] (m);
        \node[sblack] (r) at (1.5,-2.4) {}
            edge [<-] (m)
            edge [<-] (l);
        \draw (l)[dotted,->] -- (t);
        \draw (t)[dotted,<-] -- (r);
    \end{tikzpicture} \\
        $x_1$ & $x_2$ & $x_3$ & $x_4$ & $x_5$ & $x_6$ \\
        \hline
\begin{tikzpicture}[shorten <=1pt,>=stealth,semithick
,baseline=-16.7mm
,scale=0.6]
        \node[sblack] (t) at (0,0) {};
        \node[sblack] (m) at (0,-1.5) {}
            edge [<-] (t);
        \node[sblack] (l) at (-1.5,-2.4) {}
            edge [<-] (m);
        \node[sblack] (r) at (1.5,-2.4) {}
            edge [<-] (m)
            edge [<-] (l);
        \draw (l)[dotted,->] -- (t);
        \draw (t)[dotted,<-] -- (r);
        \addvmargin{2mm}
    \end{tikzpicture} &
\begin{tikzpicture}[shorten <=1pt,>=stealth,semithick
,scale=0.6]
        \node[sblack] (t) at (0,0) {};
        \node[sblack] (m) at (0,-1.5) {};
        \node[sblack] (l) at (-1.5,-2.4) {}
            edge [->] (t)
            edge [->] (m);
        \node[sblack] (r) at (1.5,-2.4) {}
            edge [->] (t)
            edge [<-] (m);
        \draw (l)[dotted,->] -- (r);
        \draw (t)[dotted,<-] -- (m);
    \end{tikzpicture} &
\begin{tikzpicture}[shorten <=1pt,>=stealth,semithick
,scale=0.6]
        \node[sblack] (l) at (-1.5,0) {};
        \node[sblack] (r) at (1.5,0) {};
        \node[sblack] (m) at (0,0) {};
        \draw[dotted,->] (l) .. controls (-1.2,2) and (1.2,2) .. (r);
        \draw[->] (l) .. controls (-0.90,0.5) and (-0.60,0.5) ..  (m);
        \draw[dotted,->] (l) .. controls (-0.90,-0.5) and (-0.60,-0.5) ..  (m);
        \draw[->] (m) .. controls (0.60,0.5) and (0.90,0.5) ..  (r);
        \draw[dotted,->] (m) .. controls (0.60,-0.5) and (0.90,-0.5) ..  (r);
    \end{tikzpicture} &
\begin{tikzpicture}[shorten <=1pt,>=stealth,semithick
,scale=0.6]
        \node[sblack] (l) at (-1.5,0) {};
        \node[sblack] (r) at (1.5,0) {};
        \node[sblack] (m) at (0,0) {};
        \draw[dotted,->] (l) .. controls (-1.2,2) and (1.2,2) .. (r);
        \draw[->] (l) .. controls (-0.90,0.5) and (-0.60,0.5) ..  (m);
        \draw[dotted,->] (l) .. controls (-0.90,-0.5) and (-0.60,-0.5) ..  (m);
        \draw[<-] (m) .. controls (0.60,0.5) and (0.90,0.5) ..  (r);
        \draw[dotted,->] (m) .. controls (0.60,-0.5) and (0.90,-0.5) ..  (r);
    \end{tikzpicture}&
\begin{tikzpicture}[shorten <=1pt,>=stealth,semithick
,scale=0.6]
        \node[sblack] (l) at (-1.5,0) {};
        \node[sblack] (r) at (1.5,0) {};
        \node[sblack] (m) at (0,0) {};
        \draw[dotted,->] (l) .. controls (-1.2,2) and (1.2,2) .. (r);
        \draw[<-] (l) .. controls (-0.90,0.5) and (-0.60,0.5) ..  (m);
        \draw[dotted,->] (l) .. controls (-0.90,-0.5) and (-0.60,-0.5) ..  (m);
        \draw[->] (m) .. controls (0.60,0.5) and (0.90,0.5) ..  (r);
        \draw[dotted,->] (m) .. controls (0.60,-0.5) and (0.90,-0.5) ..  (r);
    \end{tikzpicture}& \\
        $x_7$ & $x_8$ & $x_9$ & $x_{10}$ & $x_{11}$ & \\
        \hline
    \end{tabular}
\end{center}
    \caption{Basis of $X^s$}
    \label{fig:X}
\end{figure}
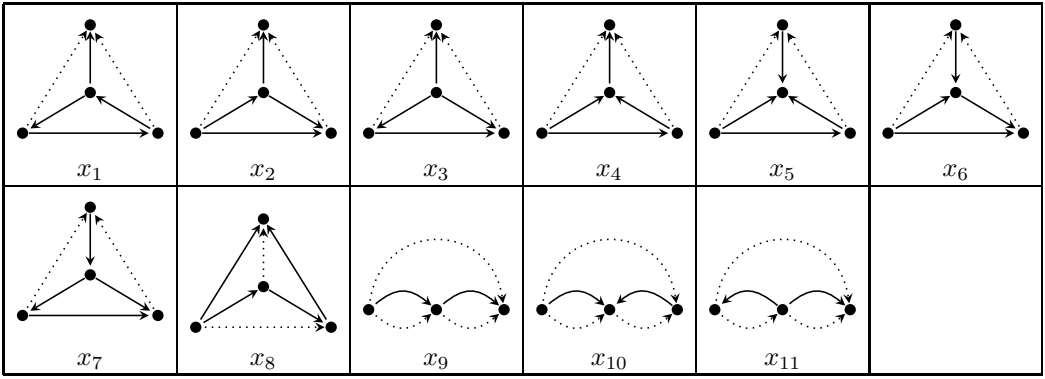
\begin{figure}
\begin{equation*}
    \begin{cases}
        d_s(x_1)&=a_1+a_4-a_7-a_9\\
        d_s(x_2)&=a_2+a_4-a_8+a_9\\
        d_s(x_3)&=a_2+a_3-a_8-a_9\\
        d_s(x_4)&=a_2+a_5-a_7+a_9\\
        d_s(x_5)&=a_2+a_6+a_8-a_{10}\\
        d_s(x_6)&=a_2+a_5+a_8+a_{10}\\
        d_s(x_7)&=a_2+a_4+a_7+a_{10}\\
        d_s(x_8)&=a_7+a_8-a_9+a_{10}\\
        d_s(x_9)&=-a_4+a_5-a_7+a_8\\ 
        d_s(x_{10})&=-a_5+a_6-2a_{10}\\ 
        d_s(x_{11})&=-a_3+a_4+2a_9\\ 
    \end{cases}
    \qquad
    \setcounter{MaxMatrixCols}{20}
    \begin{pmatrix}
        1 & 0 & 0 & 0 & 0 & 0 & 0 & 0 & 0 & 0 & 0 \\
        0 & 1 & 1 & 1 & 1 & 1 & 1 & 0 & 0 & 0 & 0 \\
        0 & 0 & 1 & 0 & 0 & 0 & 0 & 0 & 0 & 0 & -1 \\
        1 & 1 & 0 & 0 & 0 & 0 & 1 & 0 & -1 & 0 & 1 \\
        0 & 0 & 0 & 1 & 0 & 1 & 0 & 0 & 1 & -1 & 0 \\
        0 & 0 & 0 & 0 & 1 & 0 & 0 & 0 & 0 & 1 & 0 \\
        -1 & 0 & 0 & -1 & 0 & 0 & 1 & 1 & -1 & 0 & 0 \\
        0 & -1 & -1 & 0 & 1 & 1 & 0 & 1 & 1 & 0 & 0 \\
        -1 & 1 & -1 & 1 & 0 & 0 & 0 & -1 & 0 & 0 & 2 \\
        0 & 0 & 0 & 0 & -1 & 1 & 1 & 1 & 0 & -2 & 0 \\
        0 & 0 & 0 & 0 & 0 & 0 & 0 & 0 & -1 & 0 & 0 
    \end{pmatrix}
\end{equation*}
    \caption{Matrix representation of the map $d_s=p_s\circ d\circ\iota_s:X^s\rightarrow A^s.$}
    \label{fig:matrix}
\end{figure}
By using the application wolfram-alpha, we can with Gaussian elmimination see that $\alpha^s$ is not in the image of $d_s\circ d\circ p_s$. Hence we conclude that $\alpha^s$ is a non-trivial cycle of $\widehat{\dGC}^{st}_3$.
The argument to show that $\alpha^t$ is a non-trivial cycle is analogous by considering the same graphs as for $\alpha^t$ but with all solid edges reversed and all $s$-edges turned into $t$-edges.
Finally lets show that $\alpha^s$ and $\alpha^t$ represent different equivalence classes in $H(\widehat{\dGC}^{st}_3)$. For a contraction, suppose that there exists $\beta\in X$ such that $d(\beta)=\alpha^s-\alpha^t$. Let $A^t$ and $X^t$ be "targeted" versions of the spaces $A^s$ and $X^s$. More concretely, they are the vector spaces generated by the graphs in figure 4 and 5 but where the solid edges have the opposite direction and the $s$-edges have been replaced by $t$-edges. It is clear that $A^s\cap A^t=X^s\cap X^t=\{0\}$. Let $\widehat{A}$ be the obvious (say orthogonal) complement of the subspace $A^s\oplus A^t$ of $A$, and let and $\widehat{X}$ be the similar complement of the subspace $X^s\oplus X^t$ in $X$ so that $A=A^s\oplus A^t\oplus\widehat{A}$ and $X=X^s\oplus X^t\oplus \widehat{X}$. Let $p:A^s\oplus A^t\oplus\widehat{A}\rightarrow A^s\oplus A^t$ be the projection map. We note that $\widehat{X}$ is the kernel of the map $p\circ d:X\rightarrow A^s\oplus A^t$. By lemma \ref{lemma:tetra}, there is a $\beta'\in X^s\oplus X^t$ so that $p\circ d(\beta')=\alpha^s-\alpha^t$. Now $\beta'=\beta^s+\beta^t$ for some $\beta^s\in X^s$ and $\beta^t\in X^t$. Further $X^t$ is a subspace of $\overline{X}$, the kernel of the map $p_s\circ d:X\rightarrow A^s$. Hence $p\circ d(\beta^s)=\alpha^s$. But this contradicts $\alpha^s$ being a non-trivial cycle, finishing the proof.
\end{proof}
The epimorphism $\dGC^{st}_3\rightarrow\widehat{\dGC}^{st}_3$ is a quasi-isomorphism, so it remains to find a lift of $\alpha^s$ and $\alpha^t$ to cycles in $\dGC^{st}_3$.
Let us introduce a new kind of edge, defined as a linear combination of an $s$-edge and a $t$-edge $\halfEd=\sEd-\tEd$. The differential split as $d=d_V+d_E$ where $d_V$ act on vertices by vertex splitting and $d_E$ act on edges accordingly:
\begin{equation*}
    d(\Ed)=\halfEd,\qquad d(\halfEd)=0
\end{equation*}
Consider the graph
\begin{equation*}
    \Gamma^s=
    \begin{tikzpicture}[shorten <=1pt,>=stealth,semithick,baseline={([yshift=-.5ex]current bounding box.center)},scale=0.6]
        \node[sblack] (t) at (0,0) {};
        \node[sblack] (m) at (0,-1.5) {};
        \node[sblack] (l) at (-1.5,-2.4) {};
        \node[sblack] (r) at (1.5,-2.4) {};
        \draw[<-] (t) to (m);
        \draw[->] (m) to (l);
        \draw[->] (r) to (m);
        \draw[->] (l) to (r);
        \draw[OneMark] (l) -- (t);
        \draw[OneMark] (r) -- (t);
    \end{tikzpicture}
    =
    \begin{tikzpicture}[shorten <=1pt,>=stealth,semithick,baseline={([yshift=-.5ex]current bounding box.center)},scale=0.6]
        \node[sblack] (t) at (0,0) {};
        \node[sblack] (m) at (0,-1.5) {};
        \node[sblack] (l) at (-1.5,-2.4) {};
        \node[sblack] (r) at (1.5,-2.4) {};
        \draw[<-] (t) to (m);
        \draw[->] (m) to (l);
        \draw[->] (r) to (m);
        \draw[->] (l) to (r);
        \draw[dotted,->] (l) to node[auto] {$\scriptstyle{s}$} (t);
        \draw[dotted,->] (r) to node[auto,swap] {$\scriptstyle{s}$} (t);
    \end{tikzpicture}
    -
    \begin{tikzpicture}[shorten <=1pt,>=stealth,semithick,baseline={([yshift=-.5ex]current bounding box.center)},scale=0.6]
        \node[sblack] (t) at (0,0) {};
        \node[sblack] (m) at (0,-1.5) {};
        \node[sblack] (l) at (-1.5,-2.4) {};
        \node[sblack] (r) at (1.5,-2.4) {};
        \draw[<-] (t) to (m);
        \draw[->] (m) to (l);
        \draw[->] (r) to (m);
        \draw[->] (l) to (r);
        \draw[dotted,->] (l) to node[auto] {$\scriptstyle{t}$} (t);
        \draw[dotted,->] (r) to node[auto,swap] {$\scriptstyle{s}$} (t);
    \end{tikzpicture}
    -
    \begin{tikzpicture}[shorten <=1pt,>=stealth,semithick,baseline={([yshift=-.5ex]current bounding box.center)},scale=0.6]
        \node[sblack] (t) at (0,0) {};
        \node[sblack] (m) at (0,-1.5) {};
        \node[sblack] (l) at (-1.5,-2.4) {};
        \node[sblack] (r) at (1.5,-2.4) {};
        \draw[<-] (t) to (m);
        \draw[->] (m) to (l);
        \draw[->] (r) to (m);
        \draw[->] (l) to (r);
        \draw[dotted,->] (l) to node[auto] {$\scriptstyle{s}$} (t);
        \draw[dotted,->] (r) to node[auto,swap] {$\scriptstyle{t}$} (t);
    \end{tikzpicture}
    +
    \begin{tikzpicture}[shorten <=1pt,>=stealth,semithick,baseline={([yshift=-.5ex]current bounding box.center)},scale=0.6]
        \node[sblack] (t) at (0,0) {};
        \node[sblack] (m) at (0,-1.5) {};
        \node[sblack] (l) at (-1.5,-2.4) {};
        \node[sblack] (r) at (1.5,-2.4) {};
        \draw[<-] (t) to (m);
        \draw[->] (m) to (l);
        \draw[->] (r) to (m);
        \draw[->] (l) to (r);
        \draw[dotted,->] (l) to node[auto] {$\scriptstyle{t}$} (t);
        \draw[dotted,->] (r) to node[auto,swap] {$\scriptstyle{t}$} (t);
    \end{tikzpicture}
\end{equation*}
Note that $\gamma^s$ is the first term, while the other terms $\overline{\gamma}^s$ are in $\dGC^{st}_3$. Hence $d(\Gamma^s)=d(\gamma^s)+d(\overline{\gamma}^s)=\alpha^s+d(\overline{\gamma}^s)$ is a cycle, and it is a lift of $\alpha^{s}$.
The corresponding lift of $\alpha^t$ is $d(\Gamma^t)$ where
\begin{equation*}
    \Gamma^t=
    \begin{tikzpicture}[shorten <=1pt,>=stealth,semithick,baseline={([yshift=-.5ex]current bounding box.center)},scale=0.6]
        \node[sblack] (t) at (0,0) {};
        \node[sblack] (m) at (0,-1.5) {};
        \node[sblack] (l) at (-1.5,-2.4) {};
        \node[sblack] (r) at (1.5,-2.4) {};
        \draw[->] (t) to (m);
        \draw[->] (m) to (l);
        \draw[->] (r) to (m);
        \draw[->] (l) to (r);
        \draw[OneMark] (l) -- (t);
        \draw[OneMark] (r) -- (t);
    \end{tikzpicture}
    =
    \begin{tikzpicture}[shorten <=1pt,>=stealth,semithick,baseline={([yshift=-.5ex]current bounding box.center)},scale=0.6]
        \node[sblack] (t) at (0,0) {};
        \node[sblack] (m) at (0,-1.5) {};
        \node[sblack] (l) at (-1.5,-2.4) {};
        \node[sblack] (r) at (1.5,-2.4) {};
        \draw[->] (t) to (m);
        \draw[->] (m) to (l);
        \draw[->] (r) to (m);
        \draw[->] (l) to (r);
        \draw[dotted,->] (l) to node[auto] {$\scriptstyle{s}$} (t);
        \draw[dotted,->] (r) to node[auto,swap] {$\scriptstyle{s}$} (t);
    \end{tikzpicture}
    -
    \begin{tikzpicture}[shorten <=1pt,>=stealth,semithick,baseline={([yshift=-.5ex]current bounding box.center)},scale=0.6]
        \node[sblack] (t) at (0,0) {};
        \node[sblack] (m) at (0,-1.5) {};
        \node[sblack] (l) at (-1.5,-2.4) {};
        \node[sblack] (r) at (1.5,-2.4) {};
        \draw[->] (t) to (m);
        \draw[->] (m) to (l);
        \draw[->] (r) to (m);
        \draw[->] (l) to (r);
        \draw[dotted,->] (l) to node[auto] {$\scriptstyle{t}$} (t);
        \draw[dotted,->] (r) to node[auto,swap] {$\scriptstyle{s}$} (t);
    \end{tikzpicture}
    -
    \begin{tikzpicture}[shorten <=1pt,>=stealth,semithick,baseline={([yshift=-.5ex]current bounding box.center)},scale=0.6]
        \node[sblack] (t) at (0,0) {};
        \node[sblack] (m) at (0,-1.5) {};
        \node[sblack] (l) at (-1.5,-2.4) {};
        \node[sblack] (r) at (1.5,-2.4) {};
        \draw[->] (t) to (m);
        \draw[->] (m) to (l);
        \draw[->] (r) to (m);
        \draw[->] (l) to (r);
        \draw[dotted,->] (l) to node[auto] {$\scriptstyle{s}$} (t);
        \draw[dotted,->] (r) to node[auto,swap] {$\scriptstyle{t}$} (t);
    \end{tikzpicture}
    +
    \begin{tikzpicture}[shorten <=1pt,>=stealth,semithick,baseline={([yshift=-.5ex]current bounding box.center)},scale=0.6]
        \node[sblack] (t) at (0,0) {};
        \node[sblack] (m) at (0,-1.5) {};
        \node[sblack] (l) at (-1.5,-2.4) {};
        \node[sblack] (r) at (1.5,-2.4) {};
        \draw[->] (t) to (m);
        \draw[->] (m) to (l);
        \draw[->] (r) to (m);
        \draw[->] (l) to (r);
        \draw[dotted,->] (l) to node[auto] {$\scriptstyle{t}$} (t);
        \draw[dotted,->] (r) to node[auto,swap] {$\scriptstyle{t}$} (t);
    \end{tikzpicture}
\end{equation*}
Now we get an explicit action of these classes on $\mathcal{H}olieb_{1,1}^{\circlearrowleft}$ as the derivations $D_1$ and $D_2$ which respectively act on $(m,n)$ corollas as
\begin{equation*}
    D_1\Bigg(\
\begin{tikzpicture}[shorten <=1pt,
>=stealth,semithick,
baseline={([yshift=-.5ex]current bounding box.center)},scale=0.6]
    \node[sblack] (c) at (0,0) {};
    \node[] at (-1.4,1.2) {$\scriptstyle{1}$}
        edge [<-] (c);
    \node[] at (-0.7,1.2) {$\scriptstyle{2}$}
        edge [<-] (c);
    \node[] at (0.2,0.9) {$\scriptstyle{\cdots}$};
    \node[] at (1.5,1.2) {$\scriptstyle{m}$}
        edge [<-] (c);
    \node[] at (-1.5,-1.2) {$\scriptstyle{1}$}
        edge [->] (c);
    \node[] at (-0.7,-1.2) {$\scriptstyle{2}$}
        edge [->] (c);
    \node[] at (0.2,-0.9) {$\scriptstyle{\cdots}$};
    \node[] at (1.4,-1.2) {$\scriptstyle{n}$}
        edge [->] (c);
\end{tikzpicture}
\Bigg)
=
\sum\ 
\begin{tikzpicture}[shorten <=1pt,>=stealth,semithick,
baseline={([yshift=-.5ex]current bounding box.center)},
scale=0.5]
        \node[ellipse,
            draw = black,
            minimum width = 2.5cm, 
            minimum height = 2.3cm,
            dotted] (e) at (0,-1.5) {};
        \node[] (v1) at (-1.7,2.3) {$\scriptstyle{1}$}
            edge [<-] (e);
        \node[] (v2) at (-0.7,2.4) {$\scriptstyle{2}$}
            edge [<-] (e);
        \node[] at (0.4,1.2) {$\cdots$};
        \node[label] (v3) at (1.7,2.3) {$\scriptstyle{m}$}
            edge [<-] (e);
        \node[] (v4) at (-1.7,-5.3) {$\scriptstyle{1}$}
            edge [->] (e);
        \node[] (v5) at (-0.7,-5.4) {$\scriptstyle{2}$}
            edge [->] (e);
        \node[] at (0.4,-4.2) {$\cdots$};
        \node[label] (v6) at (1.7,-5.3) {$\scriptstyle{n}$}
            edge [->] (e);
        \node[sblack] (t) at (0,0) {};
        \node[sblack] (m) at (0,-1.5) {};
        \node[sblack] (l) at (-1.5,-2.4) {};
        \node[sblack] (r) at (1.5,-2.4) {};
        \draw[OneMark] (t) to (m);
        \draw[->] (m) to (l);
        \draw[->] (r) to (m);
        \draw[->] (l) to (r);
        \draw[OneMark] (l) -- (t);
        \draw[OneMark] (r) -- (t);
\end{tikzpicture}
\ +\
\sum\ 
\begin{tikzpicture}[shorten <=1pt,>=stealth,semithick,
baseline={([yshift=-.5ex]current bounding box.center)},
scale=0.5]
        \node[ellipse,
            draw = black,
            minimum width = 2.5cm, 
            minimum height = 2.3cm,
            dotted] (e) at (0,-1.5) {};
        \node[] (v1) at (-1.7,2.3) {$\scriptstyle{1}$}
            edge [<-] (e);
        \node[] (v2) at (-0.7,2.4) {$\scriptstyle{2}$}
            edge [<-] (e);
        \node[] at (0.4,1.2) {$\cdots$};
        \node[label] (v3) at (1.7,2.3) {$\scriptstyle{m}$}
            edge [<-] (e);
        \node[] (v4) at (-1.7,-5.3) {$\scriptstyle{1}$}
            edge [->] (e);
        \node[] (v5) at (-0.7,-5.4) {$\scriptstyle{2}$}
            edge [->] (e);
        \node[] at (0.4,-4.2) {$\cdots$};
        \node[label] (v6) at (1.7,-5.3) {$\scriptstyle{n}$}
            edge [->] (e);
        \node[sblack] (t) at (0,0) {};
        \node[sblack] (m) at (0,-1.5) {};
        \node[sblack] (l) at (-1.5,-2.4) {};
        \node[sblack] (r) at (1.5,-2.4) {};
        \draw[<-] (t) to (m);
        \draw[OneMark] (m) to (l);
        \draw[->] (r) to (m);
        \draw[->] (l) to (r);
        \draw[OneMark] (l) -- (t);
        \draw[OneMark] (r) -- (t);
\end{tikzpicture}
\ +\
\sum\ 
\begin{tikzpicture}[shorten <=1pt,>=stealth,semithick,
baseline={([yshift=-.5ex]current bounding box.center)},
scale=0.5]
        \node[ellipse,
            draw = black,
            minimum width = 2.5cm, 
            minimum height = 2.3cm,
            dotted] (e) at (0,-1.5) {};
        \node[] (v1) at (-1.7,2.3) {$\scriptstyle{1}$}
            edge [<-] (e);
        \node[] (v2) at (-0.7,2.4) {$\scriptstyle{2}$}
            edge [<-] (e);
        \node[] at (0.4,1.2) {$\cdots$};
        \node[label] (v3) at (1.7,2.3) {$\scriptstyle{m}$}
            edge [<-] (e);
        \node[] (v4) at (-1.7,-5.3) {$\scriptstyle{1}$}
            edge [->] (e);
        \node[] (v5) at (-0.7,-5.4) {$\scriptstyle{2}$}
            edge [->] (e);
        \node[] at (0.4,-4.2) {$\cdots$};
        \node[label] (v6) at (1.7,-5.3) {$\scriptstyle{n}$}
            edge [->] (e);
        \node[sblack] (t) at (0,0) {};
        \node[sblack] (m) at (0,-1.5) {};
        \node[sblack] (l) at (-1.5,-2.4) {};
        \node[sblack] (r) at (1.5,-2.4) {};
        \draw[<-] (t) to (m);
        \draw[->] (m) to (l);
        \draw[OneMark] (r) to (m);
        \draw[->] (l) to (r);
        \draw[OneMark] (l) -- (t);
        \draw[OneMark] (r) -- (t);
\end{tikzpicture}
\ +\
\sum\ 
\begin{tikzpicture}[shorten <=1pt,>=stealth,semithick,
baseline={([yshift=-.5ex]current bounding box.center)},
scale=0.5]
        \node[ellipse,
            draw = black,
            minimum width = 2.5cm, 
            minimum height = 2.3cm,
            dotted] (e) at (0,-1.5) {};
        \node[] (v1) at (-1.7,2.3) {$\scriptstyle{1}$}
            edge [<-] (e);
        \node[] (v2) at (-0.7,2.4) {$\scriptstyle{2}$}
            edge [<-] (e);
        \node[] at (0.4,1.2) {$\cdots$};
        \node[label] (v3) at (1.7,2.3) {$\scriptstyle{m}$}
            edge [<-] (e);
        \node[] (v4) at (-1.7,-5.3) {$\scriptstyle{1}$}
            edge [->] (e);
        \node[] (v5) at (-0.7,-5.4) {$\scriptstyle{2}$}
            edge [->] (e);
        \node[] at (0.4,-4.2) {$\cdots$};
        \node[label] (v6) at (1.7,-5.3) {$\scriptstyle{n}$}
            edge [->] (e);
        \node[sblack] (t) at (0,0) {};
        \node[sblack] (m) at (0,-1.5) {};
        \node[sblack] (l) at (-1.5,-2.4) {};
        \node[sblack] (r) at (1.5,-2.4) {};
        \draw[<-] (t) to (m);
        \draw[->] (m) to (l);
        \draw[->] (r) to (m);
        \draw[OneMark] (l) to (r);
        \draw[OneMark] (l) -- (t);
        \draw[OneMark] (r) -- (t);
\end{tikzpicture}
\end{equation*}
\begin{equation*}
    D_2\Bigg(\
\begin{tikzpicture}[shorten <=1pt,
>=stealth,semithick,
baseline={([yshift=-.5ex]current bounding box.center)},scale=0.6]
    \node[sblack] (c) at (0,0) {};
    \node[] at (-1.4,1.2) {$\scriptstyle{1}$}
        edge [<-] (c);
    \node[] at (-0.7,1.2) {$\scriptstyle{2}$}
        edge [<-] (c);
    \node[] at (0.2,0.9) {$\scriptstyle{\cdots}$};
    \node[] at (1.5,1.2) {$\scriptstyle{m}$}
        edge [<-] (c);
    \node[] at (-1.5,-1.2) {$\scriptstyle{1}$}
        edge [->] (c);
    \node[] at (-0.7,-1.2) {$\scriptstyle{2}$}
        edge [->] (c);
    \node[] at (0.2,-0.9) {$\scriptstyle{\cdots}$};
    \node[] at (1.4,-1.2) {$\scriptstyle{n}$}
        edge [->] (c);
\end{tikzpicture}
\Bigg)
=
\sum\ 
\begin{tikzpicture}[shorten <=1pt,>=stealth,semithick,
baseline={([yshift=-.5ex]current bounding box.center)},
scale=0.5]
        \node[ellipse,
            draw = black,
            minimum width = 2.5cm, 
            minimum height = 2.3cm,
            dotted] (e) at (0,-1.5) {};
        \node[] (v1) at (-1.7,2.3) {$\scriptstyle{1}$}
            edge [<-] (e);
        \node[] (v2) at (-0.7,2.4) {$\scriptstyle{2}$}
            edge [<-] (e);
        \node[] at (0.4,1.2) {$\cdots$};
        \node[label] (v3) at (1.7,2.3) {$\scriptstyle{m}$}
            edge [<-] (e);
        \node[] (v4) at (-1.7,-5.3) {$\scriptstyle{1}$}
            edge [->] (e);
        \node[] (v5) at (-0.7,-5.4) {$\scriptstyle{2}$}
            edge [->] (e);
        \node[] at (0.4,-4.2) {$\cdots$};
        \node[label] (v6) at (1.7,-5.3) {$\scriptstyle{n}$}
            edge [->] (e);
        \node[sblack] (t) at (0,0) {};
        \node[sblack] (m) at (0,-1.5) {};
        \node[sblack] (l) at (-1.5,-2.4) {};
        \node[sblack] (r) at (1.5,-2.4) {};
        \draw[OneMark] (t) to (m);
        \draw[->] (m) to (l);
        \draw[->] (r) to (m);
        \draw[->] (l) to (r);
        \draw[OneMark] (l) -- (t);
        \draw[OneMark] (r) -- (t);
\end{tikzpicture}
\ +\
\sum\ 
\begin{tikzpicture}[shorten <=1pt,>=stealth,semithick,
baseline={([yshift=-.5ex]current bounding box.center)},
scale=0.5]
        \node[ellipse,
            draw = black,
            minimum width = 2.5cm, 
            minimum height = 2.3cm,
            dotted] (e) at (0,-1.5) {};
        \node[] (v1) at (-1.7,2.3) {$\scriptstyle{1}$}
            edge [<-] (e);
        \node[] (v2) at (-0.7,2.4) {$\scriptstyle{2}$}
            edge [<-] (e);
        \node[] at (0.4,1.2) {$\cdots$};
        \node[label] (v3) at (1.7,2.3) {$\scriptstyle{m}$}
            edge [<-] (e);
        \node[] (v4) at (-1.7,-5.3) {$\scriptstyle{1}$}
            edge [->] (e);
        \node[] (v5) at (-0.7,-5.4) {$\scriptstyle{2}$}
            edge [->] (e);
        \node[] at (0.4,-4.2) {$\cdots$};
        \node[label] (v6) at (1.7,-5.3) {$\scriptstyle{n}$}
            edge [->] (e);
        \node[sblack] (t) at (0,0) {};
        \node[sblack] (m) at (0,-1.5) {};
        \node[sblack] (l) at (-1.5,-2.4) {};
        \node[sblack] (r) at (1.5,-2.4) {};
        \draw[->] (t) to (m);
        \draw[OneMark] (m) to (l);
        \draw[->] (r) to (m);
        \draw[->] (l) to (r);
        \draw[OneMark] (l) -- (t);
        \draw[OneMark] (r) -- (t);
\end{tikzpicture}
\ +\
\sum\ 
\begin{tikzpicture}[shorten <=1pt,>=stealth,semithick,
baseline={([yshift=-.5ex]current bounding box.center)},
scale=0.5]
        \node[ellipse,
            draw = black,
            minimum width = 2.5cm, 
            minimum height = 2.3cm,
            dotted] (e) at (0,-1.5) {};
        \node[] (v1) at (-1.7,2.3) {$\scriptstyle{1}$}
            edge [<-] (e);
        \node[] (v2) at (-0.7,2.4) {$\scriptstyle{2}$}
            edge [<-] (e);
        \node[] at (0.4,1.2) {$\cdots$};
        \node[label] (v3) at (1.7,2.3) {$\scriptstyle{m}$}
            edge [<-] (e);
        \node[] (v4) at (-1.7,-5.3) {$\scriptstyle{1}$}
            edge [->] (e);
        \node[] (v5) at (-0.7,-5.4) {$\scriptstyle{2}$}
            edge [->] (e);
        \node[] at (0.4,-4.2) {$\cdots$};
        \node[label] (v6) at (1.7,-5.3) {$\scriptstyle{n}$}
            edge [->] (e);
        \node[sblack] (t) at (0,0) {};
        \node[sblack] (m) at (0,-1.5) {};
        \node[sblack] (l) at (-1.5,-2.4) {};
        \node[sblack] (r) at (1.5,-2.4) {};
        \draw[->] (t) to (m);
        \draw[->] (m) to (l);
        \draw[OneMark] (r) to (m);
        \draw[->] (l) to (r);
        \draw[OneMark] (l) -- (t);
        \draw[OneMark] (r) -- (t);
\end{tikzpicture}
\ +\
\sum\ 
\begin{tikzpicture}[shorten <=1pt,>=stealth,semithick,
baseline={([yshift=-.5ex]current bounding box.center)},
scale=0.5]
        \node[ellipse,
            draw = black,
            minimum width = 2.5cm, 
            minimum height = 2.3cm,
            dotted] (e) at (0,-1.5) {};
        \node[] (v1) at (-1.7,2.3) {$\scriptstyle{1}$}
            edge [<-] (e);
        \node[] (v2) at (-0.7,2.4) {$\scriptstyle{2}$}
            edge [<-] (e);
        \node[] at (0.4,1.2) {$\cdots$};
        \node[label] (v3) at (1.7,2.3) {$\scriptstyle{m}$}
            edge [<-] (e);
        \node[] (v4) at (-1.7,-5.3) {$\scriptstyle{1}$}
            edge [->] (e);
        \node[] (v5) at (-0.7,-5.4) {$\scriptstyle{2}$}
            edge [->] (e);
        \node[] at (0.4,-4.2) {$\cdots$};
        \node[label] (v6) at (1.7,-5.3) {$\scriptstyle{n}$}
            edge [->] (e);
        \node[sblack] (t) at (0,0) {};
        \node[sblack] (m) at (0,-1.5) {};
        \node[sblack] (l) at (-1.5,-2.4) {};
        \node[sblack] (r) at (1.5,-2.4) {};
        \draw[->] (t) to (m);
        \draw[->] (m) to (l);
        \draw[->] (r) to (m);
        \draw[OneMark] (l) to (r);
        \draw[OneMark] (l) -- (t);
        \draw[OneMark] (r) -- (t);
\end{tikzpicture}.
\end{equation*}
These formulae give us the required explicit homotopy inequivalent actions of the tetrahedron class in the Kontsevich graph complex $\GC_2$ as derivations of the wheeled properad $\mathcal{H}olieb^\circlearrowleft_{1,1}$.

\end{document}